\documentclass[openany]{amsbook}

\usepackage{amsfonts, amsmath, amsthm, amssymb, epsfig,
bm, graphics, psfrag, latexsym}
\usepackage[all,cmtip]{xy}

\newtheorem{thm}{Theorem}[section]
\newtheorem{lem}[thm]{Lemma}

\newtheorem{defn}[thm]{Definition}

\newcommand{\R}{\mathbb{R}}
\newcommand{\Z}{\mathbb{Z}}
\newcommand{\F}{\mathbb{F}}
\newcommand{\C}{\mathbb{C}}
\newcommand{\T}{\mathbb{T}}
\newcommand{\wh}{\widehat}
\newcommand{\del}{\partial}
\newcommand{\wt}{\widetilde}
\newcommand{\mc}{\mathcal}
\newcommand{\mb}{\mathbb}
\newcommand{\mf}{\mathfrak}

\begin{document}

\title{Topics in Heegaard Floer homology}
\author{Sucharit Sarkar}
\address{Department of Mathematics, Princeton University, Princeton, NJ 08544, USA}
\email{sucharit@math.princeton.edu}


\date{}

\begin{abstract}
Heegaard Floer homology is an extremely powerful invariant for closed
oriented three-manifolds, introduced by Peter Ozsv\'{a}th and
Zolt\'{a}n Szab\'{o}. This invariant was later generalized by them and
independently by Jacob Rasmussen to an invariant for knots inside
three-manifolds called knot Floer homology, which was later even
further generalized to include the case of links. However
the boundary maps in the Heegaard Floer chain complexes
were defined by counting the number of points in certain moduli
spaces, and there was no algorithm to compute the invariants in general.

The primary aim of this thesis is to address this concern. We begin by
surveying various areas of this theory and providing the background
material to familiarize the reader with the Heegaard Floer homology
world. We then describe the algorithm which was discovered by Jiajun
Wang and me, that computes the hat version of the three-manifold
invariant with coefficients in $\F_2$. For the remainder of the
thesis, we concentrate on the case of knots and links inside the
three-sphere. Based on a grid diagram for a knot and following a
paper by Ciprian Manolescu, Peter Ozsv\'{a}th and me, we give a
another algorithm for computing the knot Floer homology. We conclude
by generalizing the construction to a theory of knot Floer homotopy.

\end{abstract}

\maketitle

\newpage

\section*{Acknowledgement}
\noindent My adviser Zolt\'{a}n Szab\'{o} for introducing me to the
fascinating world of Heegaard Floer homology and for guiding me
throughout the entire course of my graduate studies.

\noindent My collaborators Matthew Hedden, Andr\'{a}s Juh\'{a}sz, Ciprian
Manolescu, Peter Ozsv\'{a}th and Jiajun Wang for all the discoveries
that we made together, which constitute a significant portion of this
thesis.

\noindent The FPO committee members William Browder, David Gabai and
Zolt\'{a}n Szab\'{o} and the thesis readers Peter Ozsv\'{a}th and
Zolt\'{a}n Szab\'{o}.

\noindent Boris Bukh, William Cavendish, David Gabai, Matthew Hedden,
Andr\'{a}s Juh\'{a}sz, Robert Lipshitz, Ciprian Manolescu, Peter
Ozsv\'{a}th, Jacob Rasmussen, Sarah Rasmussen, Zolt\'{a}n Szab\'{o}
and Dylan Thurston for many enjoyable conversations and lots of
interesting remarks.

\noindent The Fine Hall common room for providing the perfect ambience to do
Mathematics.

\noindent My parents, my brother and my sister for everything.

Thank you.

\chapter{Beginning of days}  
Our story starts on a summer day in $2001$, when two Hungarian
mathematicians sat together for a few hours, and came up with one of
the most amazing theories in modern low dimensional topology.

\section{Low dimensional topology}

Low dimensional topology is the branch of differential topology that
deals with three-dimensional and four-dimensional manifolds. It seems
strange at first to concentrate on just these two dimensions, when
there are (countably) infinite number of other dimensions we could
have worked with. The justification of this restricted choice lies in
Smale's h-cobordism theorem. When stated in simple (and incorrect)
terms, it basically says that in high enough dimensions, homotopy
restrictions give information about smooth structures, and hence
differential topology follows from algebraic topology. Stated in a
more mathematical form, it says

\begin{thm}\cite{Smale} If $n\geq 5$ and $W^{n+1}$ is a cobordism
between two simply connected manifolds $M_1^n$ and $M_2^n$, and each
of the inclusions $M_i^n\hookrightarrow W^{n+1}$ induces a homotopy
equivalence, then $W^{n+1}$ is diffeomorphic to $M_1^n\times I$.
\end{thm}

The condition that $n\geq 5$ is very crucial in the statement, and
appears in a very subtle way in the proof. The fact that it is
necessary was established by Donaldson, when he disproved the
h-cobordism statement for $n=4$. The status of the statement in other
smaller dimensions may be of independent interest. For $n=0$, it is
trivial and for $n=1$ it is vacuous. The case $n=2$ was proved
recently by Perelman during his proof of the Poincar\'{e} conjecture.
The case $n=3$ stays unconquered (and as a consequence of Perelman's
work, is now equivalent to the smooth four-dimensional Poincar\'{e}
conjecture).

As mentioned at the beginning of this section, this leaves the story
in dimensions three and four wide open. From the void of uncertainty
to the pristine beauty of an unexplored world, sprang forth low
dimensional topology.

\section{Knot Theory}

One of the greatest treasures in the galleries of low dimensional
topology is the fascinating world of knots. To appreciate fully the
wonders of this new world, we need to familiarize ourselves with a few
basic definitions first. However to avoid pathologies, we always work
in either the smooth category or the piecewise-linear category, and to
remain intentionally vague, we mention this fact only once and never
allude to it again.

\begin{defn} A knot $K$ is an embedding of the circle $S^1$ into the
three-sphere $S^3$.  Two knots $K_1$ and $K_2$ are said to be
equivalent if there is an isotopy of $S^3$ (i.e. an one-parameter
family of diffeomorphisms of $S^3$ to itself) that takes $K_1$ to
$K_2$.
\end{defn}

\begin{defn} A knot diagram is an immersion of the circle $S^1$ into
the two-plane $\R^2$, such that there are no triple points, and at
every double point one of the participating arcs is declared the
overpass (the other one the underpass).
\end{defn}

Knot theory started long before low dimensional topology came into
fashion. Historically knots were always described by knot diagrams.
Given a knot diagram, it is easy to recover a knot from it, by
embedding $\R^2$ into $\R^3$ in a standard way, and then obtaining an
embedded $S^1$ in $\R^3$ from the immersed $S^1$ in $\R^2$ using the
crossing information, and finally one-point compactifying $\R^3$ to
get $S^3$. It is not difficult to see that given a knot, there is
always a knot diagram representing it. Figure \ref{fig:trefoil} shows
a knot diagram representing a right-handed trefoil knot.

\begin{figure}[ht]
\begin{center} \includegraphics[width=100pt]{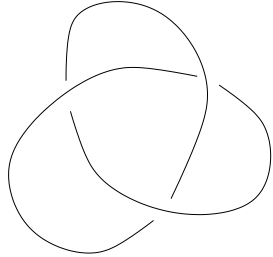}
\end{center}
\caption{The right-handed trefoil}\label{fig:trefoil}
\end{figure}

However it is always the case with these sorts of knot presentations
that while such a presentation exists, it is far from canonical. In
other words, even though every knot can be represented by a knot
diagram, two different knot diagrams can correspond to the same knot.
(Here, by two different knot diagrams, we mean two knot diagrams that
cannot be related by an isotopy of $\R^2$.) Figure \ref{fig:unknot}
illustrates two such knot diagrams, either of which represents the
trivial knot, or the unknot.

\begin{figure}[ht] \psfrag{a}{$(a)$} \psfrag{b}{$(b)$}
\begin{center} \includegraphics[width=200pt]{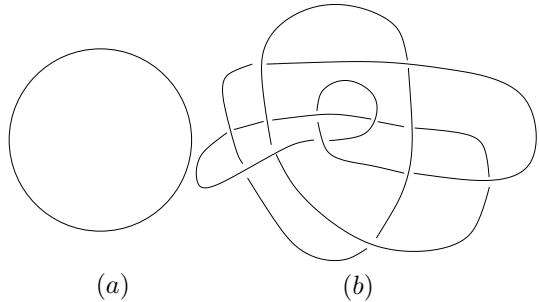}
\end{center}
\caption{Two knot diagrams for the unknot}\label{fig:unknot}
\end{figure}

In 1927, Alexander and Briggs, and independently Reidemeister came up
with essentially three local moves on knot diagrams, such that two
knot diagrams represent the same knot if and only if one can be taken
to the other using only these moves.

\begin{thm}\cite{JAGB, KR} Two knot diagrams represent the same knot
if and only if they can be related by a sequence of Reidemeister
moves (Figure \ref{fig:reidemeister}).
\end{thm}

\begin{figure}[ht] \psfrag{r1}{RI} \psfrag{r2}{RII} \psfrag{r3}{RIII}
\begin{center} \includegraphics[width=250pt]{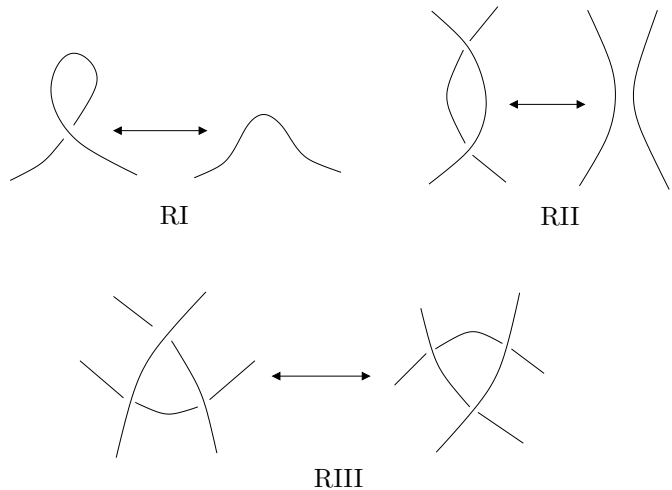}
\end{center}
\caption{The three Reidemeister moves}\label{fig:reidemeister}
\end{figure}

One of the central problems in knot theory is distinguishing two
knots. In other words, given two knot diagrams, we want to know
whether or not they represent the same knot. In case they do, it is
usually very easy to show that they do, simply by relating one knot
diagram to another using Reidemeister moves (however \emph{easy} is
relative, see for example Figure \ref{fig:unknot}). In case they do
not, i.e. the two knot diagrams represent different knots, they are
usually shown to be different using some invariants.

The first and the most classical invariant (and also the most
non-maneuverable one) is the fundamental group of the knot
complement, more commonly known as the knot group. The knot group is
already enough to distinguish the unknot from the trefoil (in fact it
is a theorem that the knot group distinguishes the unknot, but it is
not always easy to check whether or not two groups are
isomorphic). The knot group of the unknot is $\Z$, and a clever
application of Van Kampen shows that the knot group of the trefoil is
given by the group presentation $<a,b|a^3=b^2>$, which has a very
natural surjection to $S_3$, the symmetric group on three letters.

However most of the other classical invariants of knots are defined as
invariants of knot diagrams, which are then shown to remain invariant
under the Reidemeister moves. Perhaps the most famous knot invariant
of all times, the Alexander polynomial, can be argued to belong to
this category. In the original definition by J.W.Alexander
\cite{JAtop} where it is defined as the generator of a principal ideal
domain over $\Z[t,t^{-1}]$, the polynomial is only defined up to a
multiplication by $\pm t^n$. John Conway later showed that the
polynomial satisfies a linear Skein relation, and its value on the
unknot was enough to determine it, and a reparametrized version of the
Alexander polynomial is called the Alexander-Conway polynomial.
Throughout this thesis, we will be referring to the normalized but
unparametrized version of the polynomial as the Alexander polynomial
(even though technically it is a Laurent polynomial). For example, the
Alexander polynomial for the unknot is $1$ and the Alexander
polynomial for the trefoil is $t-1+t^{-1}$.

Much later Kauffman presented a combinatorial description of the
Alexander polynomial without using Skein relation, and defined only in
terms of a knot diagram. Given a knot diagram, let regions be the
connected components of the complement of the immersed circle in
$\R^2$. Let $A$ be the unbounded region, and let $B$ be another region
adjacent to the unbounded region.

\begin{defn}\cite{LK} A Kauffman state is a map which assigns to each
double point of the knot diagram, a region adjacent to it, such that
each region other than $A$ and $B$ is assigned to some double point.
\end{defn}

Let us now work with oriented knots, represented by oriented knot
diagrams. Given a Kauffman state $c$ and a double point $v$, let
$a_{c,v}$ be defined according to Figure \ref{fig:kauffman}.

\begin{figure}[ht] \psfrag{a1}{$-t^{-\frac{1}{2}}$} \psfrag{a2}{$1$}
\psfrag{a3}{$1$} \psfrag{a4}{$t^{\frac{1}{2}}$}
\psfrag{b1}{$-t^{\frac{1}{2}}$} \psfrag{b2}{$1$} \psfrag{b3}{$1$}
\psfrag{b4}{$t^{-\frac{1}{2}}$}
\begin{center} \includegraphics[width=200pt]{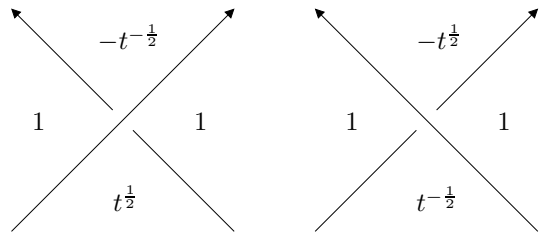}
\end{center}
\caption{The definition of $a_{c,v}$}\label{fig:kauffman}
\end{figure}

\begin{thm}\cite{LK} For a knot presented in an oriented knot diagram,
let $\mc{K}$ be the set of all Kauffman states and let $V$ be the set
of all double points.  Then the Alexander polynomial of the knot is
given by $\sum_{c\in\mc{K}}\prod_{v\in V} a_{c,v}$.
\end{thm}

The other central problem in knot theory is understanding geometric
properties of knots. This is the area where there is the closest
interaction between knot theory and other aspects of low dimensional
topology. It can be argued that understanding three-manifolds
is equivalent to understanding knots inside the three sphere $S^3$. To
state precise mathematical results in support of this claim, we first
need to extend the world of knots to embrace links.

\begin{defn} A link is an embedding of a disjoint union of circles
into $S^3$. Two links are said to be equivalent if there is an isotopy
of $S^3$ that takes one link to another. Each circle in the link is
called a link component.
\end{defn}

Planar link diagrams are defined similarly, and once more two link
diagrams represent the same link if and only if they can be connected
by a sequence of Reidemeister moves. The following theorem by
Alexander is the first indication of how links are related to
three-manifolds.

\begin{thm}\cite{JAriem} Any oriented three-manifold $Y$ is branched
cover of $S^3$ with the branch set being a link.
\end{thm}

However there is an even more subtle relation between links in $S^3$
and three-manifolds. A surgery on a link is a procedure by which we
remove a tubular neighborhood of a link in $S^3$ and then glue back
the neighborhood (which is a disjoint union of solid tori) in a
(possibly) different fashion. It is an amazing theorem that,

\begin{thm} Every oriented three-manifold $Y$ is a surgery along some
link $L$ in $S^3$.
\end{thm}

It is not surprising then that many geometric properties of knots and
links translate to properties of three-manifolds. We end this section
after discussing the geometric property that concerns us the most, the
Seifert genus of a knot.

\begin{defn} A Seifert surface for a knot $K$ is a compact oriented
surface $F$ embedded in $S^3$ such that $\del F=K$.
\end{defn}

Seifert showed \cite{HS} that every knot admits a Seifert surface,
thus leading to the definition of the Seifert genus of a knot.

\begin{defn} The genus of a knot $K$ is the smallest number among the
genera of the Seifert surfaces that bound $K$.
\end{defn}

It is easy to see that the unknot is the only knot of genus
$0$. Figure \ref{fig:seifert} shows a genus one surface bounding the
right-handed trefoil, thus showing that the trefoil has genus $1$.

\begin{figure}[ht]
\begin{center} \includegraphics[width=150pt]{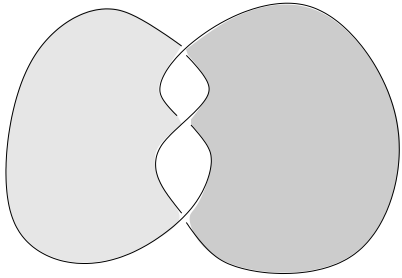}
\end{center}
\caption{Minimal genus Seifert surface of the
trefoil}\label{fig:seifert}
\end{figure}

From the very nature of the definition of the genus of a knot, it is
obvious that it is a knot invariant, but \emph{a priori} it is not
even clear whether or not it can be computed. Amazingly the Alexander
polynomial provides some information about the genus.

\begin{thm} The normalized Alexander polynomial is a symmetric Laurent
polynomial, and the genus of a knot is at least the degree of its
Alexander polynomial.
\end{thm}

For example, the Alexander polynomial for the $(3,4)$-torus knot is
$t^3-t^2+1-t^{-2}+t^3$, which shows that the genus of the
$(3,4)$-torus knot is at least three. (The genus is in fact equal to
three, as seen by cleverly finding a genus three Seifert surface).

Before we conclude this section, we should mention that this section
has been a mere glimpse at the wonderful world of knots and links. We
have only talked of theorems which have some (often minor) connections
with the rest of the thesis, and left hundreds of other stories in
knot theory untold.

\section{Floer homology}

We take leave of low dimensional topology to take a brief detour
to the realms of Floer homology. Historically, Floer homology deals
with two $n$-dimensional Lagrangians inside a $2n$-dimensional
symplectic manifold. However we will be dealing with a slightly
different situation. What follows is one of the simplest versions of
Floer homology, suited to our specific needs.

Let $M^{2n}$ be a closed manifold with a complex structure. Let the
induced almost complex structure be $J$, i.e.  $J$ is a map from the
tangent bundle to itself with $J^2=-Id$. A totally real subspace is a
submanifold $N$ such that if $v$ is a non-zero tangent vector to $N$,
then $J(v)$ is not a tangent vector to $N$. Clearly the dimension of a
totally real subspace it at most $n$. Let $L_1^n$ and $L_2^n$ be two
totally real subspaces which are transverse to one another. Thus $L_1$
and $L_2$ intersect in a finite number of points.

Let us work over a commutative ring $R$ (usually it is $\Z$ or
$\F_2$). The chain complex is the free $R$-module generated by the
finitely many points in $L_1\cap L_2$. Given $x,y\in L_1\cap L_2$, a
Whitney disk joining $x$ to $y$ is a map $\phi$ from the unit disk $D$
in the complex plane $\C$ to $M$ such that $\phi(-i)=x$, $\phi(i)=y$,
$\phi(\del D\cap\{s\in \C|Re(s)>0\})\subset L_1$ and $\phi(\del
D\cap\{s\in \C|Re(s)<0\})\subset L_2$. Two such Whitney disks are said
to be homotopic to one another, if they are homotopic relative the
boundary conditions. Let $\pi_2(x,y)$ be the set of all Whitney disks
joining $x$ to $y$ up to homotopy equivalence.  Note that given a
Whitney disk joining $x$ to $y$ and another Whitney disk joining $y$
to $z$, we can glue them together to get a Whitney disk joining $x$ to
$z$. This gives a natural map (which we denote by $+$) from
$\pi_2(x,y)\times\pi_2(y,z)$ to $\pi_2(x,z)$ which we will need
later.

To define the Floer homology we need a chain complex. We already have
the generators for the chain complex, namely the points in $L_1\cap
L_2$, so all that we need are the boundary maps. This is where things
get complicated. The boundary map $\del$ depends on a function $c$
called the count function, which maps Whitney disks to $R$, and for
$x\in L_1\cap L_2$, $\del x$ can be written as
$$\del x=\sum_{y\in L_1\cap L_2}\sum_{\phi\in\pi_2(x,y)}c(\phi)y$$

This definition immediately leads to further questions. It is not even
clear \emph{a priori} that given $x,y\in L_1\cap L_2$, there are
finitely many $\phi\in\pi_2(x,y)$. Thus for the definition to even
make sense, we must have $c(\phi)=0$ for all but finitely many
$\phi\in\pi_2(x,y)$.

The second and more important issue is that there is no guarantee that
$\del^2=0$. The definition of the count function has to be specially
designed to ensure this. The usual way to define $c(\phi)$ is the
following.

We first choose a number of divisors (complex submanifolds, each with
real dimension $(2n-2)$) $Z_1,\ldots,Z_k$ each disjoint from
$L_1\cup L_2$. The chain homotopy type of the Floer chain complex
would very much depend on the choice of these divisors. Then given a
Whitney disk $\phi$, its algebraic intersection number with each of
the $Z_i$'s is well-defined, since the boundary of the Whitney disk
lies on $L_1\cup L_2$ and $Z_i$'s are disjoint from $L_1\cup L_2$. We
declare $c(\phi)=0$ if $\phi\cdot Z_i\neq 0$ for some $i$.

Given a Whitney disk $\phi$, let its moduli space $\mc{M}(\phi)$ be
the space of all complex maps from the unit disk $D$ in $\mb{C}$ to
$M$ which represent $\phi$. Let the Maslov index $\mu(\phi)$ be the
expected dimension of the moduli space. We once more declare
$c(\phi)=0$ if $\mu(\phi)\neq 1$.

There is a natural action of $\mb{R}$ on $\mc{M}(\phi)$ given by the
precomposition by the one-parameter family of diffeomorphisms of $D$
which fixes $i$ and $-i$. Let $\wh{\mc{M}(\phi)}=\mc{M}(\phi)/\mb{R}$
be the reparametrized moduli space. If $\mu(\phi)=1$, the expected
dimension of $\mc{M}(\phi)$ is one, and hence the expected dimension
of $\wh{\mc{M}(\phi)}$ is zero. Let us assume that the complex
structure on $M$ is generic enough such that whenever $\mu(\phi)=1$,
the actual dimension of $\wh{\mc{M}(\phi)}$ is zero, and it
consists of finitely many points. There is usually an orientation on
$\mc{M}(\phi)$ which induces a sign of $\pm 1$ on these points, and
the aptly named count function $c(\phi)$ is simply the count of these
points with sign. Since we are still in complete awe of the definition
of the Floer chain complex, let us restate it once more in the light
of new knowledge.
$$\del x=\sum_{y\in L_1\cap L_2}\sum_{\begin{subarray}{l}\phi\in\pi_2(x,y)\\ \phi\cdot
  Z_i=0\forall i\\ \mu(\phi)=1\end{subarray}}\#(\wh{\mc{M}(\phi)})y$$

The reason for introducing the divisors $Z_i$'s in the definition is
two fold. Usually if there are enough divisors, then given $x,y$, all but
finitely many  of $\phi\in\pi_2(x,y)$ will not be disjoint from $\cup_i
Z_i$, and hence $c(\phi)$ will be zero for all but finitely many $\phi\in\pi_2(x,y)$.

The second reason is slightly more subtle. Recall that we also need
$\del$ to be a boundary map, i.e. $\del^2=0$. What this translates to
is the following. For all $x,z\in L_1\cap L_2$,
$$\sum_{y\in L_1\cap L_2}\sum_{\begin{subarray}{l}\phi\in\pi_2(x,y)\\
    \psi\in\pi_2(y,z)\end{subarray}} c(\phi)c(\psi)=0 $$

We may in addition assume that both $\phi$ and $\psi$ are disjoint
from the divisors, and either has $\mu=1$. Since the Maslov index is
additive, this would imply $\phi+\psi\in\pi_2(x,z)$ is a Whitney disk
of Maslov index two. Thus given $x,z$ and a Whitney disk
$u\in\pi_2(x,z)$ with $\mu(u)=2$ which avoids all the divisors, it is
enough to show that,
$$\sum_{y\in L_1\cap L_2}\sum_{\begin{subarray}{c}\phi\in\pi_2(x,y)\\
    \psi\in\pi_2(y,z)\\ \phi+\psi=u\\ \mu(\phi)=\mu(\psi)=1\\
    \phi\cdot Z_i=\psi\cdot Z_i =0\forall
    i\end{subarray}}\#(\wh{\mc{M}(\phi)})\#(\wh{\mc{M}(\psi)}) =0 $$

It is clear that to understand
$\#(\wh{\mc{M}(\phi)})\#(\wh{\mc{M}(\psi)})$ , we need to understand
$\wh{\mc{M}(u)}$. Recall that for a Whitney disk $\varphi$, the
expected dimension of $\wh{\mc{M}(\varphi)}$ is $(\mu(\varphi)-1)$. So
assume that the complex structure on $M$ is generic enough, such that
$\wh{\mc{M}(\varphi)}=\varnothing$ for all Whitney disks with
$\mu(\varphi)<1$, it is a collection of finitely many points when
$\mu(\varphi)=1$, and it a compact one-manifold when $\mu(\varphi)=2$.

Let us now analyze the boundary degenerations of $\wh{\mc{M}(u)}$. The
Maslov indices of the different components in a boundary degeneration
has to add up to $\mu(u)=2$, and the index of each component has to be
at least one, so there has to be exactly two components in each
boundary degeneration. Thus only three types of boundary degenerations
as shown in Figure \ref{fig:degenerations}, are possible.

\begin{figure}[ht] 
\psfrag{x1}{$x$}
\psfrag{y1}{$y$}
\psfrag{z1}{$z$}
\psfrag{x2}{$x$}
\psfrag{y2}{$y$}
\psfrag{z2}{$z$}
\psfrag{x3}{$x$}
\psfrag{y3}{$y$}
\psfrag{z3}{$z$}
\psfrag{a}{$(a)$}
\psfrag{b}{$(b)$}
\psfrag{c}{$(c)$}
\begin{center} \includegraphics[width=180pt]{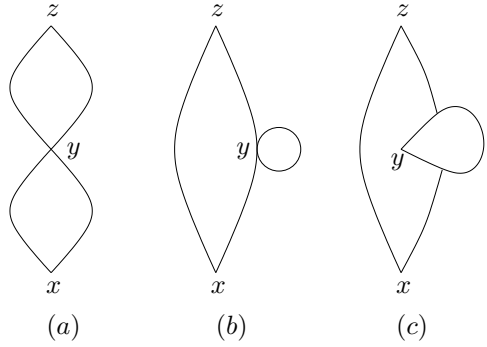}
\end{center}
\caption{The boundary degenerations of $u$}\label{fig:degenerations}
\end{figure}

Somehow by a miracle, if there are enough divisors such that every
holomorphic Maslov index one Whitney disk whose boundary lies entirely
in one of $L_1$ and $L_2$, intersects one of $Z_i$'s, then the Cases
$(b)$ and $(c)$ of Figure \ref{fig:degenerations} cannot occur. Then
the sum $\sum_{\phi+\psi=u}\#(\wh{\mc{M}(\phi)})\#(\wh{\mc{M}(\psi)})$
counts precisely the number of boundary points of $\wh{\mc{M}(u)}$
(with orientation). However since $\wh{\mc{M}(u)}$ is a compact one-manifold,
it has an even number of boundary points, and  hence the sum (even
with sign) is zero, leading to a proof that $\del^2=0$.

\section{Heegaard Floer homology}

Heegaard Floer homology is an amazing application of the techniques of
Floer homology where all these miracles do indeed come true. It was
introduced in a couple of revolutionary papers \cite{POZSz3manifolds,
  POZSzapplications} by Peter Ozsv\'{a}th and Zolt\'{a}n Szab\'{o},
primarily as an invariant for closed three-manifolds. From now on,
assume all the three-manifolds are closed, connected and oriented.

\begin{defn}
A genus $g$ Heegaard splitting of a three-manifold $Y$ is a
decomposition of $Y$ into a union of two oriented genus $g$ handlebodies
$U_g$ and $V_g$, which are glued together by an orientation reversing
diffeomorphism $h:\del U_g\rightarrow \del V_g$.
\end{defn}

It is clear that given two handlebodies and a gluing map between them,
we get a three-manifold. It is perhaps not that clear that every
three-manifold admits a Heegaard decomposition. However it is a well
known theorem that,

\begin{thm}
Every oriented three-manifold admits a Heegaard decomposition.
\end{thm}

One way to see this is by constructing Morse function on the
three-manifold $Y$.

\begin{defn}
  A Morse function on a manifold $M$ is a smooth function
  $f:M\rightarrow \R$, such that at every critical point (i.e. where
  $df=0$), the Hessian $d^2 f$ is non-singular. The index of a
  critical point is the number of negative eigenvalues of the
  Hessian. A Morse function is said to be self-indexing if at every
  critical point the value of the Morse function equals the index of
  the critical point.
\end{defn}

\begin{defn}
  A gradient-like flow associated to a Morse function $f$ on $M$ is a
  flow whose singularities are precisely the Morse critical points,
  and furthermore the flow agrees with a gradient flow induced from
  some metric in a neighborhood of the critical points, and the Morse
  function is a strictly decreasing function along any flowline.
\end{defn}

It is an extremely important result that every oriented smooth
manifold admits a self-indexing Morse function and a gradient-like
flow associated to it. In fact given a natural number $k$, we can even
ensure that the Morse function has exactly $k$ maxima and $k$ minima.
Thus to find a Heegaard decomposition of a three-manifold $Y$, all we
need to do is to find a self-indexing Morse function $f:Y\rightarrow
[0,3]$, and then define the handlebodies $U$ and $V$ as
$f^{-1}[0,\frac{3}{2}]$ and $f^{-1}[\frac{3}{2},3]$ respectively.

We choose the Morse function $f$ to have exactly $k$ maxima and $k$
minima (usually we choose $k=1$). This implies (since $\chi(Y)=0$)
that the number of index $1$ critical points must equal the number of
index $2$ critical points.  Let the common number be $(g+k-1)$. Then
$f^{-1}(\frac{3}{2})$ is a genus $g$ surface $\Sigma_g$ and the
Heegaard decomposition described in the previous paragraph is a genus
$g$ Heegaard decomposition.

In addition, if we are given a gradient like flow associated to this
Morse function, then we can represent the whole picture by a single
combinatorial diagram on the Heegaard surface $\Sigma_g$. Let
$\alpha_1,\ldots,\alpha_{g+k-1}$ (numbered arbitrarily) be the
disjoint circles on $\Sigma_g$ that flow down to the $(g+k-1)$ index
one critical points, and let $\beta_1,\ldots,\beta_{g+k-1}$ (also
numbered arbitrarily) be the circles that flow up to the $(g+k-1)$
index two critical points. While choosing the gradient-like flow, we
ensure that the $\alpha$ circles intersect the $\beta$ circles
transversely. Clearly the $\alpha$ circles are disjoint from one
another, and their complement has $k$ components flowing down to the
$k$ index zero critical points, and thus the $\alpha$ circles generate
a half-dimensional subspace of $H_1(\Sigma_g)$. A similar statement
holds for the $\beta$ circles. We also choose $k$ basepoints
$z_1,\ldots,z_k$ (needless to say, also numbered arbitrarily) such
that each component of $(\Sigma\setminus\alpha)$ contains one
basepoint, and each component of $(\Sigma\setminus\beta)$ contains one
basepoint. Such a diagram is called a Heegaard diagram, but for future
convenience, let us record the definition here.

\begin{defn}
A Heegaard diagram $(\Sigma_g,\alpha_1,\ldots,\alpha_{g+k-1},
\beta_1,\ldots,\beta_{g+k-1}, z_1,\cdots,z_k)$ is genus-$g$ surface
$\Sigma_g$ with two collections of $(g+k-1)$ disjoint curves, called
$\alpha$ curves and $\beta$ curves respectively, and $k$ basepoints
$z_1,\cdots,z_k$ such that $(\Sigma\setminus\alpha)$ has $k$ components
each with a basepoint, and $(\Sigma\setminus\beta)$ also has $k$
components each containing a basepoint. 
\end{defn}

It is reasonably clear that a Heegaard diagram captures all the
information that is needed to reconstruct the three-manifold $Y$. We
thicken $\Sigma_g$ to get $\Sigma_g\times[-1,1]$. We add two-handles
to $\alpha_i\times\{-1\}$ and to $\beta_j\times\{1\}$. This results in
a three-manifold with $2k$ boundary components each homeomorphic to
$S^2$. We add solid balls to each boundary component to recover the
three-manifold $Y$. Figure \ref{fig:s3heegaard} shows a genus-two
Heegaard diagram (with $k=2$) representing $S^3$.

\begin{figure}[ht] 
\psfrag{a1}{$\alpha_1$}
\psfrag{a2}{$\alpha_2$}
\psfrag{a3}{$\alpha_3$}
\psfrag{b1}{$\beta_1$}
\psfrag{b2}{$\beta_2$}
\psfrag{b3}{$\beta_3$}
\psfrag{z1}{$z_1$}
\psfrag{z2}{$z_2$}
\begin{center} 
\includegraphics[width=300pt]{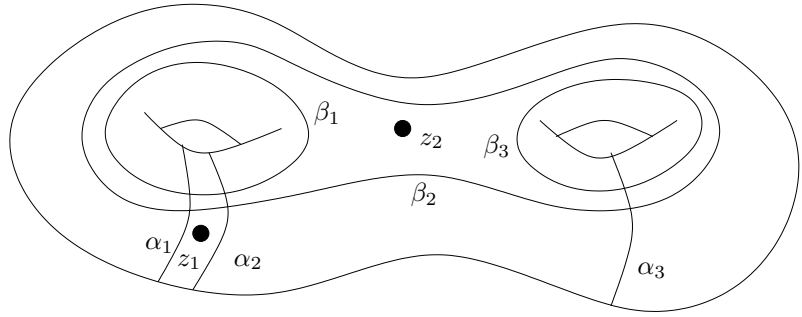}
\end{center}
\caption{A Heegaard diagram of $S^3$}\label{fig:s3heegaard}
\end{figure}

Thus every Heegaard diagram represents a specific three-manifold, and
any three-manifold can be represented by a Heegaard diagram. However
there can be lots of Heegaard diagrams representing the same
three-manifold. It turns out that any two Heegaard diagrams
representing the same three-manifold can be related by a sequence of
moves of the following type.

\begin{defn}
  In an isotopy, the $\alpha$ and the $\beta$ curves move
  independently (i.e. it does not have to be induced from an isotopy
  on the whole surface) by isotopies in the complement of the
  basepoints.
\end{defn}

\begin{defn}
  In a handleslide of the $\alpha$ curves, we take a pair of pants
  region bounded by the curves $c$, $c'$ and $c''$ which does not
  contain any basepoint, and whose intersection with the $\alpha$
  curves is precisely the union of the circles $c$ and $c'$, and we
  then replace the $\alpha$ curve $c'$ with a new $\alpha$ curve
  $c''$. A handleslide of the $\beta$ curves is defined similarly. 
\end{defn}

\begin{defn}
In a stabilization of the first type, we increase the genus of the
Heegaard surface by adding an one-handle, and we add an $\alpha$ curve
and a $\beta$ curve as shown in Figure \ref{fig:heegaardmoves}$(b)$. A
destabilization of the first type is the reverse of this move.
\end{defn}

\begin{defn}
  In a stabilization of the second type, we add one $\alpha$ circle,
  one $\beta$ circle and one basepoint like in
  \ref{fig:heegaardmoves}$(c)$. A destabilization of the second type
  is the reverse of this move.
\end{defn}

The moves (other than isotopy) are shown in Figure
\ref{fig:heegaardmoves}. It is clear that these moves do not change
the underlying three-manifold. Interestingly, the following theorem
shows that some sort of a converse is also true.

\begin{figure}[ht] 
\psfrag{a}{$(a)$}
\psfrag{b}{$(b)$}
\psfrag{c}{$(c)$}
\psfrag{a1}{$\alpha$}
\psfrag{b1}{$\beta$}
\psfrag{a2}{$\alpha$}
\psfrag{b2}{$\beta$}
\psfrag{z}{$z$}
\psfrag{c1}{$c$}
\psfrag{c2}{$c'$}
\psfrag{c3}{$c''$}
\begin{center} 
\includegraphics[width=330pt]{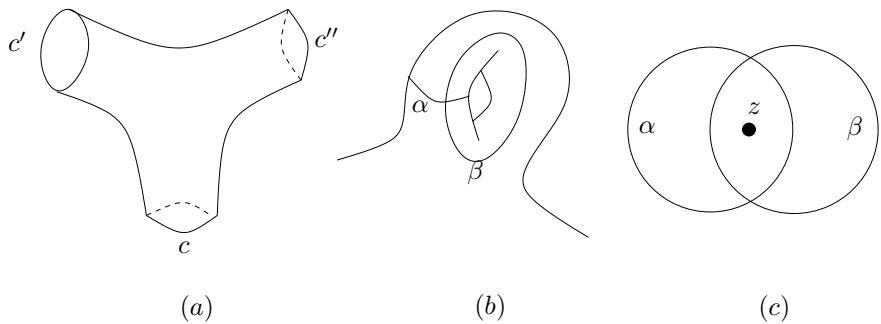}
\end{center}
\caption{Moves on Heegaard diagrams}\label{fig:heegaardmoves}
\end{figure}

\begin{thm}
  Two Heegaard diagrams represent the same three-manifold if and only
  if they are related by a sequence of isotopies, handleslides, and
  stabilizations and destabilizations of either type. In fact two Heegaard
  diagrams with the same number of basepoints representing the same
  manifold can be related by a sequence of isotopies, handleslides and
  stabilizations and destabilizations of the first type only.
\end{thm}

Now, given a three-manifold $Y$, we (essentially by choosing a
specific type of Morse function, and a gradient like flow
corresponding to it) choose a Heegaard diagram $(\Sigma_g,
\alpha,\beta, z)$ representing $Y$. Consider the symmetric product
$Sym^{g+k-1}(\Sigma_g)=\Sigma_g\times \cdots \times\Sigma_g/ S_{g+k-1}$,
where $S_{g+k-1}$ is the group of permutations on $(g+k-1)$ letters
acting naturally on the Cartesian product. Even though the action of
$S_{g+k-1}$ on the Cartesian product is far from a free action, the
quotient turns out to be manifold (this follows from the observation
$\C^n=Sym^n(\C)$, a consequence of the fundamental theorem of
algebra). We choose a complex structure on $\Sigma_g$, which in turn
induces a complex structure on $Sym^{g+k-1}(\Sigma_g)$, and a generic
perturbation (in a precise sense, as described in
\cite{POZSz3manifolds}) of this complex structure is chosen. We are
soon going to apply the heavy machinery of Floer theory, and this
$2(g+k-1)$ dimensional manifold $Sym^{g+k-1}(\Sigma_g)$ is the complex
manifold that we start with.

Given any permutation $\sigma\in S_{g+k-1}$, there is a $(g+k-1)$
dimensional torus $T_{\alpha,\sigma}=\alpha_{\sigma(1)}\times\cdots
\times \alpha_{\sigma(g+k-1)}$ in $\Sigma_g^{g+k-1}$. These $(g+k-1)!$
tori are all disjoint (this is just an extremely fancy way of saying
that the $\alpha$ circles are disjoint), and the action of $S_{g+k-1}$
simply permutes these tori. Thus $\T_{\alpha}$, the quotient of these
tori, lying in $Sym^{g+k-1}(\Sigma_g)$ (and denoted by
$\alpha_1\times\cdots \times\alpha_{g+k-1}$) is also a torus, and is a
half-dimensional totally real subspace. The torus
$\T_{\beta}=\beta_1\times\cdots \times\beta_{g+k-1}$ is defined
similarly.

We are almost set for applying the Floer machinery. We have the
$2(g+k-1)$-dimensional complex manifold, and two totally real
$(g+k-1)$-dimensional subspaces. The divisors are all that we need.
Recall that the symmetric product is just the parametrizing space of
unordered $(g+k-1)$-tuples of points on the surface. Let
$Z_i=\{z_i\}\times Sym^{g+k-2}(\Sigma_g)$ be the codimension-two
holomorphic subspace consisting of all the points in the symmetric
product whose one of the $(g+k-1)$ coordinates is the basepoint
$z_i$. Once more, the statement that $Z_i$ is disjoint from
$\T_{\alpha}\cup\T_{\beta}$ is a fancy restatement of the fact that
$z_i$ lies in the complement of the $\alpha$ and $\beta$ curves.

Now finally, at the end of the beginning, we define the Floer chain
complex. The chain complex is the free $R$-module generated by
$\T_{\alpha}\cap \T_{\beta}$, and for a generator $x$, the boundary
map is given by
$$\wh{\del} x=\sum_{y\in\T_{\alpha}\cap\T_{\beta}}\sum_{\begin{subarray}{c}
    \phi\in\pi_2(x,y)\\ \mu(\phi)=1\\ \phi\cdot Z_i=0 \end{subarray}}
\#(\mc{M}(\phi)/\R)y$$

The chain complex defined above is called the hat version of the
Heegaard Floer chain complex (hence the notation $\wh{\del}$). In
order to complete our eduction, there is another important chain
complex that we need to know of, called the minus version of the
Heegaard Floer chain complex. The new chain complex is the
$R[U_1,\ldots, U_k]$-module generated freely by points of
$\T_{\alpha}\cap\T_{\beta}$, and the boundary map is defined on each
generator $x$ as follows
$$\del^{-} x=\sum_{y\in\T_{\alpha}\cap\T_{\beta}}\sum_{\begin{subarray}{c}
    \phi\in\pi_2(x,y)\\ \mu(\phi)=1\\ \phi\cdot Z_i=n_i \end{subarray}}
\#(\mc{M}(\phi)/\R)U_i^{n_i}y$$

We have made lots of choices on the way. We have chosen a
self-indexing Morse function with $k$ maxima and minima, we have
chosen a gradient-like flow corresponding to it, we have chosen $k$
basepoints (subject to certain restrictions), we have chosen a complex
structure on the Heegaard surface and a generic perturbation of the
induced complex structure on the symmetric product, and finally we
have chosen a ring $R$ which is usually $\Z$ or $\F_2$. If the
three-manifold $Y$ is a rational homology sphere, i.e. if $H^1(Y)=0$,
then this is all we need. If however $b_1(Y)>0$, then for the hat
version, we also need to ensure that the Heegaard diagram is
admissible, and for the minus version, we need to ensure that the
diagram is strongly admissible. These are minor technical restriction
that we do not need to bother ourselves with.

We end this section with the following wonderful theorems, established
by Ozsv\'{a}th and Szab\'{o}, which can easily be named the
Fundamental Theorems of Heegaard Floer Homology.  The theorems
basically say that the homologies of the chain complexes are
three-manifold invariants.

\begin{thm}\cite{POZSz3manifolds}\label{thm:fthf1}
  The map $\wh{\del}$ defined above is a boundary map, i.e.
  $(\wh{\del})^2=0$, and there is an $R$-module $\wh{HF}(Y,R)$
  depending only on $Y$ and $R$, such that the homology of the hat
  version of the Floer chain complex is isomorphic to
  $\wh{HF}(Y,R)\otimes^{k-1}R^2$.
\end{thm}

\begin{thm}\cite{POZSz3manifolds}\label{thm:fthf2}
  The map $\del^{-}$ defined above is also a boundary map, and there
  is an $R[U]$-module $HF^{-}(Y,R)$ depending only on $Y$ and $R$,
  such that the homology of the minus version of the Floer chain
  complex is isomorphic to $HF^{-}(Y,R)$ as $R[U]$-modules, where the
  $U$ action on the Floer homology is given by multiplication by any
  of the $U_i$'s.
\end{thm}

\section{Knot Floer homology}

The last time we talked about knots, we only talked about knots and
links inside the three-sphere $S^3$. This is because for the most
part in this thesis, we will not be needing the general case. However,
in general, a link is an embedding of a disjoint union of circles
inside a three-manifold $Y$, and a knot is a link with one
component. The following is a Heegaard diagram describing a link.

\begin{defn}
  A link Heegaard diagram $(\Sigma_g,\alpha_1,\ldots,\alpha_{g+k-1},
  \beta_1,\ldots,\beta_{g+k-1}, z_1,\cdots, z_k, w_1, \cdots, w_k)$ is
  genus-$g$ surface $\Sigma_g$ with two collections of $(g+k-1)$
  disjoint curves, called $\alpha$ curves and $\beta$ curves
  respectively, and two collections of $k$ basepoints called $z$
  points and $w$ points respectively, such that
  $(\Sigma\setminus\alpha)$ has $k$ components each with a
  $z$-basepoint and a $w$-basepoint, and $(\Sigma\setminus\beta)$ also
  has $k$ components each containing a $z$-basepoint and
  $w$-basepoint.
\end{defn}

Given a link Heegaard diagram, observe that if we forget about the
$w$-basepoints, we get an ordinary Heegaard diagram. The
three-manifold which that Heegaard diagram represents is the ambient
three-manifold $Y$. To recover the link $L\subset Y$, in each
component of $(\Sigma\setminus \alpha)$ join $z$ to $w$ by an embedded
oriented arc avoiding all the $\alpha$ curves, and then push the
interior of this arc towards the $\alpha$-handlebody $U_g$ (i.e. the
handlebody in which all the $\alpha$ curves bound disks). Similarly in
each component of $(\Sigma\setminus\beta)$ join $w$ to $z$ by an
embedded oriented arc avoiding all the $\beta$ curves, and then push
the interior of the arc towards the $\beta$-handlebody $V_g$. The
resulting one-dimensional oriented subspace of $Y$ is the link $L$.

More often than not, we work with knots inside $S^3$. In that case, we
usually choose $k=1$, although for the most part in this thesis, we
will not be doing that. Figure \ref{fig:trefoildiagram} shows a
Heegaard diagram for the trefoil inside $S^3$ with $k=1$.

\begin{figure}[ht] 
\psfrag{a}{$\alpha$}
\psfrag{b}{$\beta$}
\psfrag{z}{$z$}
\psfrag{w}{$w$}
\begin{center} 
\includegraphics[width=180pt]{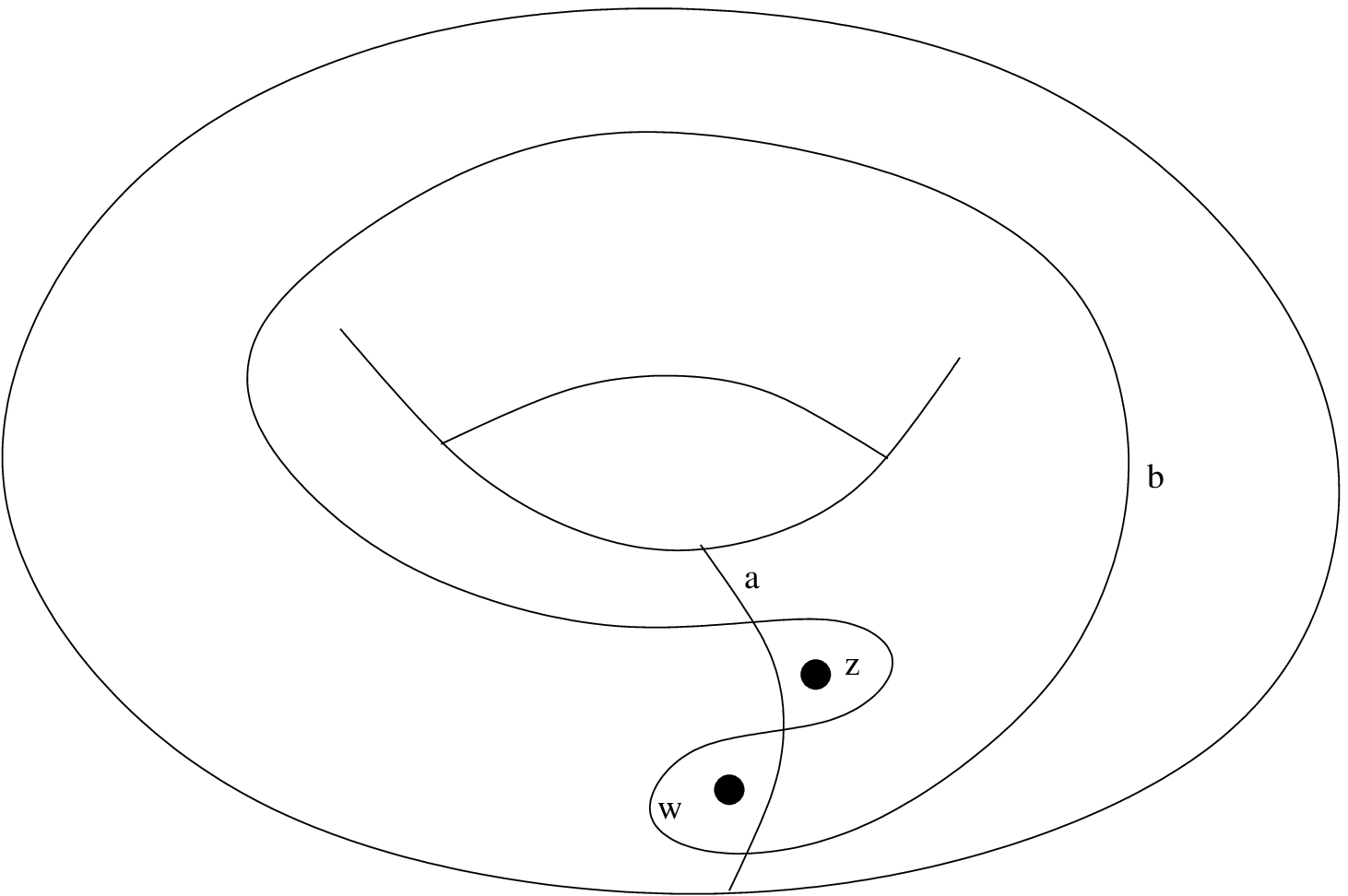}
\end{center}
\caption{A Heegaard diagram for the trefoil}\label{fig:trefoildiagram}
\end{figure}

Knot Floer homology was introduced by Peter Ozsv\'{a}th and Zolt\'{a}n
Szab\'{o} \cite{POZSzknotinvariants} and independently by Jacob
Rasmussen in his PhD thesis \cite{JR}. It was later generalized by
Ozsv\'{a}th and Szab\'{o} to include the case of links
\cite{POZSzlinkinvariants}, but for now, let us just present the
definition of knot Floer homology.

Given an oriented knot $K$ inside an oriented three-manifold $Y$, let
$(\Sigma,\alpha,\beta,z,w)$ be an admissible Heegaard diagram
representing the knot. It turns out that there is always such a
Heegaard diagram, and two such Heegaard diagrams with the same number
of basepoints are related by a sequence of isotopies, handleslides and
stabilizations and destabilizations of the first type in the
complement of both the $z$-basepoints and the $w$-basepoints. We once
more choose a complex structure on $\Sigma_g$ and then take a generic
perturbation of the induced complex structure on
$Sym^{g+k-1}(\Sigma_g)$. Let $\T_{\alpha}$ and $\T_{\beta}$ be the two
totally real half-dimensional tori, and let $Z_i=\{z_i\}\times
Sym^{g+k-2}(\Sigma_g)$ and $W_i=\{w_i\}\times Sym^{g+k-2}(\Sigma_g)$
be the codimension-two holomorphic subspaces. Fix a commutative ring
$R$ (once more, usually $\Z$ or $\F_2$). For the hat version, the
chain complex is the $R$-module freely generated by
$\T_{\alpha}\cap\T_{\beta}$, and for a generator $x$, the boundary map
is given by
$$\wh{\del} x=\sum_{y\in\T_{\alpha}\cap\T_{\beta}}\sum_{\begin{subarray}{c}
    \phi\in\pi_2(x,y)\\ \mu(\phi)=1\\ \phi\cdot Z_i=\phi\cdot W_i=0 \end{subarray}}
\#(\mc{M}(\phi)/\R)y$$

In the minus version, the chain complex is the
$R[U_1,U_2,\ldots,U_k]$-module freely generated by $\T_{\alpha}\cap
\T_{\beta}$, and for a generator $x$, the boundary map is given by
$$\del^{-} x=\sum_{y\in\T_{\alpha}\cap\T_{\beta}}\sum_{\begin{subarray}{c}
    \phi\in\pi_2(x,y)\\ \mu(\phi)=1\\ \phi\cdot Z_i=n_i\\ \phi\cdot
    W_i=0 \end{subarray}} \#(\mc{M}(\phi)/\R)U_i^{n_i}y$$

The natural analogues of Theorems \ref{thm:fthf1} and \ref{thm:fthf2}
hold, and thus in both the hat version and the minus version, Heegaard
Floer homology presents us with knot invariants called knot Floer
homology and denoted by $\wh{HFK}(K,Y)$ and $HFK^{-}(K,Y)$. However in
certain cases, especially for knots inside $S^3$, the invariant has
more structure than meets the eye, and hence from now on until
the end of this section, let us always choose the ambient three-manifold
to be $S^3$.

Given two generators $x,y\in\T_{\alpha}\cap\T_{\beta}$, the space of
Whitney disks joining them $\pi_2(x,y)$, is isomorphic to $\Z$ for
$(k+g)\geq 4$ (a minor restriction that can easily be ensured by
stabilization). In fact, in the next section, we will introduce a
slightly different definition of $\pi_2(x,y)$ and under the new
definition, the space of Whitney disks joining any two points will
always be isomorphic to $\Z$ for integral homology spheres. Choose a
Whitney disk $\phi\in\pi_2(x,y)$. For any point
$p\in\Sigma_g\setminus(\alpha\cup\beta)$, let
$n_p(\phi)=\phi\cdot(\{p\}\times Sym^{g+k-2}(\Sigma_g))$ (we are mostly
interested in the case when $p$ is one of the basepoints). Then define
the relative Maslov grading to be $M(x,y)=\mu(\phi)-\sum_i
n_{z_i}(\phi)$ and the relative Alexander grading to be $A(x,y)=\sum_i
(n_{w_i}(\phi)-n_{z_i}(\phi))$. It is relative easy to check that the
definition is independent of the choice of $\phi\in\pi_2(x,y)$, and
the only subtlety in showing that they are indeed relative gradings
(i.e. $M(x,y)+M(y,z)=M(x,z)$ and $A(x,y)+A(y,z)=A(x,z)$) lies in the
observation that the Maslov index $\mu$ is additive.

The definitions convert the hat version of the chain complex to a
relatively bigraded $R$-module (we declare all elements of $R$ to have
$(M,A)$ bigrading $(0,0)$). The minus version of the chain complex can
also be made a relatively bigraded $R[U_1,\ldots,U_k]$-module by
declaring each $U_i$ to have $(M,A)$ bigrading $(-2,-1)$. It is easy
to check that in both the hat and the minus version, the boundary map
reduces the Maslov grading by one and keeps the Alexander grading
constant. Thus in either case, the homology carries a relative
bigrading, where the relative Maslov grading is essentially the
homological grading. This induces a relative bigrading on
$\wh{HFK}(K,S^3)\otimes^{k-1} R^2$ and $HFK^{-}(K,S^3)$. For the hat
version, in each copy of $R^2$, the two generators are declared to
have $(M,A)$-bigradings of $(0,0)$ and $(-1,-1)$, and thus we get an
induced bigrading on $\wh{HFK}(K,S^3)$ too. Further note that the
definition of the relative Maslov grading did not use the
$w$-basepoints, and hence the relative Maslov grading is in fact a
relative grading on the Heegaard Floer homology of the ambient
three-manifold.

The three-sphere admits a Heegaard diagram with only one generator (in
fact it is the only three-manifold to admit such Heegaard diagrams)
and hence $\wh{HF}(S^3)=\Z$. For knots inside $S^3$, the relative
Maslov grading can be lifted to an absolute Maslov grading (also
denoted by $M$) by declaring the absolute Maslov grading of the
generator of $\wh{HF}(S^3)$ to be zero. There is a similar
well-defined lift of the relative Alexander grading to an absolute
one. It is defined to be the unique lift such that following property
holds. 
$$\#\{x\in\T_{\alpha}\cap\T_{\beta}|A(x)>0\} \equiv
\#\{x\in\T_{\alpha}\cap\T_{\beta}|A(x)<0\}\pmod{2}$$

The not so obvious fact that there is such a lift, and it is unique,
is a simple consequence of the following cute theorem by Ozsv\'{a}th
and Szab\'{o}. The proof uses Kauffman's definition of Alexander's
polynomial, and we leave it as something for the interested reader to
prove or look up.

\begin{thm}\cite{POZSzknotinvariants}
The Alexander polynomial is the Euler characteristic of the knot Floer
homology, or in other words the Alexander polynomial of a knot $K$ is
equal to $\pm(\sum_i \sum_j (-1)^i rk(\wh{HFK}_{i,j}(K,S^3))t^j)$,
where $\wh{HFK}_{i,j}(K,S^3)$ is the part of the hat version of knot
Floer homology in $(M,A)$ bigrading $(i,j)$.
\end{thm}

Recall that the Alexander polynomial provided some information about
the genus of a knot. It is only natural to expect that the knot Floer homology
will also provide some information about the genus. However it turns
out that, due to yet another amazing theorem by Ozsv\'{a}th and
Szab\'{o}, the knot Floer homology in fact determines the genus.

\begin{thm}\cite{POZSzgenusbounds}
  If $g(K)$ is the genus of a knot $K$, then $g(K)$ is the highest
  Alexander grading $j$ such that $\bigoplus_i \wh{HFK}_{i,j}(K,S^3)$ is
  non-trivial (with coefficients in $\Z$).
\end{thm} 

Thus, modulo an algorithm to calculate the knot Floer homology, the
above theorem provides a way to calculate a geometric invariant, the
genus of a knot. Another geometric property of knots that can be
computed using knot Floer homology is fiberedness. A knot is said to
be fibered if the knot complement is a fiber bundle fibering over the
meridian (a meridian is a simple closed curve on the boundary of a
tubular neighborhood of the knot, which bounds a disk inside the
neighborhood). The strength of knot Floer homology as a knot invariant
is further established by the following theorem proved by Yi Ni
\cite{YN} and later by Andr\'{a}s Juh\'{a}sz \cite{AJ}.

\begin{thm}\cite{YN,AJ}
  If $g(K)$ is the genus of a knot $K$, then $K$ is fibered if and
  only if, $\bigoplus_i \wh{HFK}_{i,g(K)}(K,S^3)$ (computed with
  coefficients in $\Z$) is isomorphic to $\Z$.
\end{thm}

\section{Cylindrical Reformulation}

The story of Floer homology that we have described so far involves maps
from a disk to high dimensional complex manifolds. Not only are such
maps incredibly hard to maneuver, they are also incredibly hard to
visualize. In a remarkable paper \cite{RL}, Robert Lipshitz presented
the cylindrical reformulation of Heegaard Floer homology, which made
certain aspects somewhat unnatural, but made almost all the aspects
easier to compute, and as a side product produced a combinatorial
formula for the Maslov index.

Let us restrict ourselves to the case of closed three-manifolds since
the case for knots inside three-manifolds is very similar. Let
$(\Sigma_g, \alpha_1,\ldots,\alpha_{g+k-1}, \beta_1,
\ldots,\beta_{g+k-1}, z_1,\ldots, z_k)$ be an admissible Heegaard
diagram for a three-manifold $Y$. A generator $x$ is a formal sum
$x_1+\cdots+ x_{g+k-1}$ of $(g+k-1)$ distinct points on $\Sigma_g$
such that each $\alpha$ circle contains one point and each $\beta$
circle contains one point. (It is easy to see that generators
correspond to points of $\T_{\alpha}\cap\T_{\beta}$.) Let $\mc{G}$ be
the set of all such generators. A domain $D$ joining $x$ to $y$ is a
$2$-chain generated by components of $\Sigma_g\setminus(\alpha
\cup\beta)$ such that $\del((\del D)_{|\alpha})=y-x$, and by a
(slight) misuse of notation, the set of all domains joining $x$ to $y$
is denoted by $\pi_2(x,y)$. It is not true that domains joining $x$ to
$y$ correspond to Whitney disks joining $x$ to $y$ in the symmetric
product, but however given a Whitney disk $\phi$, there is a domain
$D(\phi)$ associated to it, defined as follows. A region is defined
to be a component of $\Sigma\setminus (\alpha\cup\beta)$ and the
coefficient of the $2$-chain $D(\phi)$ at a region is defined to be
$n_p(\phi)$ where $p$ is any point in the region. In fact for
$(k+g)\geq 4$, this association is bijective.

If $p$ is a point of intersection between an $\alpha$ and a $\beta$
curve, and $D$ is some $2$-chain generated by regions, then $n_p(D)$
is defined to be the average of the coefficients of $D$ at the four
(possibly different) regions around $p$. Then for a generator
$x=\sum_i x_i$, the point measure $n_x(D)$ is defined as $\sum_i
n_{x_i}(D)$.

Fix a metric on the surface $\Sigma_g$ such that all the $\alpha$
curves and all the $\beta$ curves are geodesics and they intersect
each other at right angles. For any $2$-chain $D$ generated by the
regions, define the Euler measure $e(D)$ as $\frac{1}{2\pi}$ times the
integral of the curvature along the $2$-chain $D$. Being an integral,
the Euler measure is additive, which implies that if $D=\sum_i a_i
D_i$ where $a_i$'s are integers and $D_i$'s are regions, then
$e(D)=\sum_i a_i e(D_i)$. Also note that if a region $D_i$ is a
$2n$-gon (i.e. if it is homeomorphic to an open ball, and if it has
$n$ $\alpha$ arcs and $n$ $\beta$ arcs on its boundary), then
$e(D_i)=1-\frac{n}{2}$.

Given a domain $D\in\pi_2(x,y)$ (and after another minor abuse of
notation), the Maslov index is defined by the Lipshitz' formula as
$\mu(D)=e(D)+n_x(D)+n_y(D)$. The abuse of notation is quickly
justified by the following theorem by Lipshitz,

\begin{thm}\cite{RL}
Let $\phi$ be a Whitney disk joining $x$ to $y$ in the symmetric
product, and let $D(\phi)$ be domain associated to it. Then
$\mu(\phi)=\mu(D(\phi))$. 
\end{thm}

The relative Maslov grading is defined similarly. If $Y$ is an integer
homology sphere, then for two generators
$x$ and $y$, choose $D\in\pi_2(x,y)$, and define $M(x,y)=\mu(D)-\sum_i
n_{z_i}(D)$. This definition is again easily seen to be independent of
the choice of the domain $D$. However showing that this defines a
relative grading, i.e. $M(x,y)+M(y,z)=M(x,z)$ without resorting to the
previous theorem involves more work, and is hereby left as a
challenging exercise to the interested reader, see \cite{SS}.

As before, if $b_1(Y)>0$, then we require the Heegaard diagram to be
admissible. This time, we will actually state precisely what this
means. Let $\pi_2^0(x,y)$ be the subset of $\pi_2(x,y)$ consisting of
all the domains $D$ with $n_{z_i}(D)=0$ for all basepoints $z_i$. A
domain $D\in\pi_2^0(x,x)$ is called a periodic domain. For a
diagram to be (weakly) admissible, we require all non-trivial periodic
domains to have both positive and negative coefficients (as
$2$-chains).

Let us choose a complex structure on $\Sigma_g$, and consider the
induced complex structure on $Sym^{g+k-1}(\Sigma_g)$. In theory, we
should be working with a generic perturbation of the complex
structure, but that is where things get complicated, so for now, let
us just stick to the induced complex structure. Let us consider a
Whitney disk $\phi\in\pi_2(x,y)$ which has a holomorphic
representative, i.e. $\phi$ can be thought of as a holomorphic map
from the unit disk to the symmetric product, satisfying certain
boundary conditions. There is a $(g+k-1)$-sheeted holomorphic branched
covering map $\Sigma_g\times Sym^{g+k-2}(\Sigma_g)\rightarrow
Sym^{g+k-1}(\Sigma_g)$. Therefore there is a compact surface $F$ (with
boundary) which is $(g+k-1)$-sheeted covering of the unit disk $D^2$
and a map $F\rightarrow\Sigma_g\times Sym^{g+k-2}(\Sigma_g)$ such that
the following diagram commutes. 
$$\xymatrix{F\ar[r]\ar[d]&\Sigma_g\times Sym^{g+k-2}(\Sigma_g)\ar[d]\\
D^2\ar^-{\phi}[r] &Sym^{g+k-1}(\Sigma_g)}$$

We postcompose the map from $F$ to
$\Sigma_g\times Sym^{g+k-2}(\Sigma_g)$ with the projection map to the
first factor. Thus we get holomorphic maps from $F$ to $\Sigma_g$ and
$D^2$, and hence an induced holomorphic map $u:F\rightarrow
\Sigma_g\times D^2$ (where the target has the product complex
structure). Let $p_1$ and $p_2$ be the first projection and the second
projection respectively, which implies that the map from $F$ to
$\Sigma_g$ is $p_1\circ u$ and the $(k+g-1)$-sheeted branched cover of
$F$ over $D_2$ is $p_2\circ u$. It is easy to see that the image of
$p_1\circ u$ (as $2$-chains) is $D(\phi)$, the domain associated to $\phi$.

Let $\Delta\subset Sym^{g+k-1}(\Sigma_g)$ be the fat diagonal,
i.e. the set of all unordered $(g+k-1)$-tuples of points in $\Sigma_g$
where some two points are equal. It is a codimension-two holomorphic
subspace which is disjoint from the two tori $\T_{\alpha}$ and
$\T_{\beta}$. Thus for any Whitney disk $\phi$, the intersection
number $\phi\cdot\Delta$ is well-defined. The following relation was
proved by Rasmussen in his PhD thesis,

\begin{thm}\cite{JR}
For any Whitney disk $\phi$, $\phi\cdot\Delta=\mu(\phi)-2e(D(\phi))$.
\end{thm}

In light of Lipshitz' formula, the above simplifies to give
$\phi\cdot\Delta=\mu_x(D(\phi))+\mu_y(D(\phi))-e(D(\phi))$. Also note
that the $(g+k-1)$-sheeted branched cover of $\Sigma_g\times
Sym^{g+k-2}(\Sigma_g)\rightarrow Sym^{g+k-1}(\Sigma_g)$ is branched
precisely over the fat diagonal $\Delta$. Hence $F$ is branched over
the unit disk $D^2$ precisely at the points which map by $\phi$ to
$\Delta$. Hence the number of branch points of $p_2\circ u$ is
$\phi\cdot\Delta =\mu_x(D(\phi))+\mu_y(D(\phi))-e(D(\phi))$.

The above observation also has the following two important
implications. Firstly, all the branch points of $p_2\circ u$ lie in
the interior of the disk $D^2$, since the boundary of the disk maps to
$\T_{\alpha}\cup\T_{\beta}$ which is disjoint from the fat diagonal
$\Delta$. This implies that the map $p_2\circ u_{|\del F}$ is
$(g+k-1)$-sheeted covering map to the circle $\del D^2$. Let
$X_1,\ldots,X_{g+k-1}$ be the preimages of $-i$ and let
$Y_1,\ldots,Y_{g+k-1}$ be the preimages of $i$ (all numbered
arbitrarily). The complement of the $X$-points and $Y$-points in $\del
F$ can be divided into arcs of two types $A$-type and $B$-type, where
the $A$-type arcs cover $\del D^2\cap\{s\in\C|Re(s)>0\}$ and $B$-type
arcs cover $\del D^2\cap\{s\in\C|Re(s)<0\}$. Then on $\del F$, the
$X$-points and the $Y$-points alternate, and the $A$-type arcs and the
$B$-type arcs alternate. The boundary conditions on
$\phi$ induce certain boundary conditions on $p_1\circ u$, whereby the
formal sum of the images of the $X$-points is the generator $x$, the
formal sum of the images of the $Y$-points is the generator $y$, the
$A$-type arcs map to the $\alpha$ curves and the $B$-type arcs maps to
the $\beta$ curves.

Secondly, the holomorphic map $u:F\rightarrow \Sigma_g\times D^2$ is
an embedding. For otherwise, if two distinct points $p$ and $q$ map to
the same point $\{t\}\times\{s\}$, then $\phi(s)$ is a
$(g+k-1)$-unordered tuple of points where two of the points are $t$,
and hence $\phi(s)$ intersects the fat diagonal $\Delta$, and hence
the map $p_2\circ u$ should have been branched over $s$, implying
$p=q$.

All this says is that given a Whitney disk $\phi$ with holomorphic
representatives, we can construct a Riemann surface $F$ (with
boundary) and a holomorphic embedding $u:F\rightarrow \Sigma_g\times
D^2$, satisfying certain boundary conditions. Conversely, given a
holomorphic embedding $u$ satisfying the previously mentioned boundary
conditions, we can recover the Whitney disk $\phi$. Since $p_2\circ u$
is a $(g+k-1)$-sheeted branched cover, for each point $s\in D^2$, look
at its $(g+k-1)$-preimages in $F$ (possibly with multiplicities), and
then map them by $p_1\circ u$ to get an unordered $(g+k-1)$-tuple of
points in $\Sigma_g$ or in other words a point in
$Sym^{g+k-1}(\Sigma_g)$. This (holomorphic) map is the required
Whitney disk $\phi$.

The setting for the cylindrical reformulation of Floer homology is now
clear. For some commutative ring $R$, the hat version of the chain
complex is the $R$-module freely generated by $\mc{G}$.  Given a
domain $D$ joining a generator $x$ to a generator $y$, the moduli
space $\mc{M}(D)$ is the moduli space of all embeddings of compact
Riemann surfaces $u:F\hookrightarrow \Sigma_g\times D^2$, satisfying
the above-mentioned boundary conditions such that the image of
$p_1\circ u$ is $D$. Here the complex structure on $\Sigma_g\times
D^2$ is a generic perturbation of the product complex structure, but
throughout this section, we will keep things simple and assume that it
is in fact the product complex structure. There is a natural
$\R$-action on this moduli space given by postcomposing the map
$p_2\times u$ by the one-parameter family of diffeomorphisms of
$D^2\setminus\{\pm i\}$, and let the quotient be the reparametrized
moduli space $\wh{\mc{M}(D)}$. It turns out that the expected
dimension of $\mc{M}(D)$ is $\mu(D)$ (and hence the expected dimension
of $\wh{\mc{M}(D)}$ is $\mu(D)-1$), and for a generator $x$, the hat
version of the boundary map is given by
$$\wh{\del} x=\sum_{y\in\mc{G}}\sum_{\begin{subarray}{c}
    D\in\pi_2(x,y)\\ \mu(D)=1\\ n_{z_i}(D)=0 \end{subarray}}
\#(\mc{M}(D)/\R)y$$

In the minus version, we are allowed to pass through the
basepoints. The chain complex is the $R[U_1,\ldots,U_k]$-module freely
generated by $\mc{G}$, and for a generator $x$, the boundary map
is given by 
$$\del^{-} x=\sum_{y\in\mc{G}}\sum_{\begin{subarray}{c}
    D\in\pi_2(x,y)\\ \mu(D)=1\\ n_{z_i}(D)=n_i \end{subarray}}
\#(\mc{M}(D)/\R)U_i^{n_i}y$$

We have spent a long time without any figures, so this is an opportune
moment to introduce Figure \ref{fig:domain}, a genus two Heegaard
diagram for $S^3$ with one basepoint. The genus two surface $\Sigma_2$
is obtained by gluing the circles that lie in the same
horizontal level (which are also the $\alpha$ circles), in an
orientation reversing way such that the black dots on the circles
match up. The shaded region is a (positive) Maslov index one domain
joining the generator marked with white squares to the generator
marked with white dots.
Here is yet another tricky exercise for the
curious reader. Show that (with the product complex structure on
 $\Sigma_2\times D^2$ and coefficients in $\F_2$) $\#(\wh{\mc{M}(D)})=1$.

\begin{figure}[ht] 
\psfrag{a1}{$\alpha_1$}
\psfrag{a2}{$\alpha_2$}
\psfrag{b1}{$\beta_1$}
\psfrag{b2}{$\beta_2$}
\psfrag{z}{$z$}
\begin{center} 
\includegraphics[width=200pt]{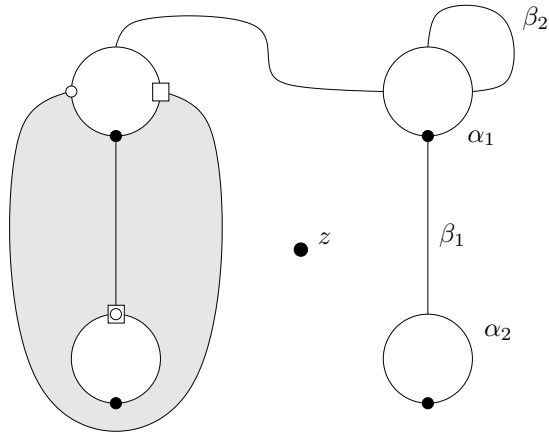}
\end{center}
\caption{A Maslov index one domain}\label{fig:domain}
\end{figure}

The following theorems bring relief, and justify the inclusion of the
word `reformulation' in the nomenclature of this section.

\begin{thm}\cite{RL}
  The homology of the hat version of the chain complex defined above is isomorphic
  to $\wh{HFK}(Y,R)\otimes^{k-1}R^2$.
\end{thm}

\begin{thm}\cite{RL}
  The homology of the minus version of the chain complex defined above
  is isomorphic
  to $HFK^{-}(Y,R)$ as $R[U]$-modules.
\end{thm}

Before we conclude this section, we should mention the following few
formulae, all of which are corollaries of the cylindrical
reformulation.

\begin{thm}\cite{RL}
  The Euler characteristic of the surface $F$ is given by
  $\chi(F)=2e(D)+g+k-1-\mu(D)$.
\end{thm}

A component of the surface $F$ is called a trivial disk if it is a
disk with only one $X$-marking and one $Y$-marking on its boundary,
and it maps to a single point in $\Sigma_g$ by $p_1\circ u$. If
$p\in\Sigma_g$ is the image a trivial disk, then clearly $p$ is both
an $x$-coordinate and a $y$-coordinate, and furthermore $n_p(D)=0$.
Recall that the number of branch points of $p_2\circ u$ is
$\mu(D)-2e(D)$. The number of branch points of $p_1\circ u$ can also
be computed as $\mu(D)-e(D)-\frac{1}{2}(g+k-1-t)$ where $t$ is the
number of trivial disks.

There is one last thread that we need to wrap up. It is a
simple observation, but an extremely important one.

\begin{thm}
  If $\mc{M}(D)\neq\varnothing$, then $D$ is a positive domain, i.e.
  $n_p(D)\geq 0$ for all points $p\in\Sigma
  \setminus(\alpha\cup\beta)$.
\end{thm}

\begin{proof}
  If $\mc{M}(D)\neq\varnothing$, then there is a holomorphic map
  $u:F\rightarrow \Sigma_g\times D^2$ representing the domain $D$.
  However $n_p(D)$ is the intersection between $u(F)$ and $\{p\}\times
  D^2$, and assuming the complex structure is the product complex
  structure (at least in a neighborhood of $\{p\}\times D^2$), both
  the subspaces are holomorphic objects and hence they intersect
  non-negatively, thus concluding the proof.
\end{proof}

\chapter{Letting bigons be bigons}
Heegaard Floer homology is a great invariant. Other than being
completely new and extremely powerful, it also enjoys a geometric
lineage which allows it to provide information about many geometric
properties of three-manifolds. (One prime example of this phenomenon
is knot Floer homology for knots in $S^3$ determining the knot genus.)
Hopefully our brief encounter with Heegaard Floer homology in the
previous chapter has already convinced the reader of this fact.

However there is one minor inconvenience in the whole theory. Till
date, it has not admitted any combinatorial reformulation. For
example, there is no algorithm to compute ${HFK}^{-}(Y,R)$ for any
ring $R$ (just a word caution, we are not claiming whether or not
there can be any algorithm, it is just that until now we have not
discovered any). In this chapter we present a partial solution to the
problem, mostly following the lines of the paper \cite{SSJW} by Jiajun
Wang and the present author.

\section{Consequences of being nice}

We restrict our attention to some special Heegaard diagrams called
nice pointed Heegaard diagrams. The terminology is perhaps a little
unfortunate, since the term `nice' is neither mathematical, nor an
accurate description of these types of Heegaard diagrams.

\begin{defn}\label{defn:nicediagram} 
  Let $\mc{H}=(\Sigma_g, \alpha_1,\ldots,\alpha_{g+k-1},
  \beta_1,\ldots,\beta_{g+k-1}, z_1,\ldots,z_k)$ be a Heegaard diagram
  for a three-manifold $Y$. The diagram $\mc{H}$ is called a nice
  pointed diagram if any region that does not contain any basepoint
  $z_i$ is either a bigon or a square.
\end{defn}

Let $Y$ be a closed oriented three-manifold. Suppose $Y$ has a nice
admissible Heegaard diagram $\mc{H}=(\Sigma, \alpha, \beta, z)$ (in
fact it is not so hard to see that a diagram being nice implies that
the diagram is admissible, see \cite{RLCMJW}). We choose a product
complex structure on $\Sigma_g\times D^2$.

\begin{defn}\label{defn:empty}
  A domain $D\in \pi_2(x,y)$ with coefficients $0$ and $1$ is called
  an empty embedded $2n$-gon, if it is topologically an embedded disk
  with $2n$ vertices (a vertex being a point of the form $x_i$ or
  $y_i$) on its boundary, such that at each vertex $v$,
  $n_v(D)=\frac{1}{4}$, and it does not contain any $x_i$ or $y_i$ in
  its interior.
\end{defn}

The following theorems show that, for a domain $D\in\pi_2^0(x,y)$ (or
in other words, a domain $D\in\pi_2(x,y)$ which avoids all the
basepoints, i.e. $n_{z_i}(D)=0$ for all $i$), the count function
$c(D)\neq 0$ if and only if $D$ is an empty embedded bigon or an empty
embedded square, and in that case $c(D)=1$. Thus $c(D)$ can be
computed combinatorially in a nice Heegaard diagram.

\begin{thm}\label{thm:holo=>embedded}\cite{SSJW}
  Let $D\in \pi_2^0(x,y)$ be a domain such that $\mu(D)=1$. If $D$ has
  a holomorphic representative, then $\phi$ is an empty embedded bigon
  or an empty embedded square.
\end{thm}

\begin{proof}

  We know that only positive domains can have holomorphic
  representatives. We also know that bigons and squares have
  non-negative Euler measure. We will use these facts to limit the
  number of possible cases.

  Suppose $D=\sum a_i D_i$, where $D_i$'s are regions (i.e. components
  of $\Sigma\setminus(\alpha\cup\beta)$) containing no basepoints.
  Since $D$ has a holomorphic representative, we have $a_i\geq 0$,
  for all $i$. Since each $D_i$ is a bigon or a square, we have
  $e(D_i)\geq 0$ and hence $e(D)\geq 0$. So, by Lipshitz' formula
  $\mu(D)=e(D)+n_x(D)+n_y(D)$, we get
  $0\le n_x(D)+n_y(D)\le 1$.

  Now let $x=x_1+\cdots+x_{g+k-1}$ and $y=y_1+\cdots+y_{g+k-1}$, with
  $x_i,y_i \in \alpha_i$. We say $D$ hits some $\alpha$ circle if
  $\partial D$ is non-zero on some part of that $\alpha$ circle.
  Since $D \neq n\Sigma$, it has to hit at least one $\alpha$ circle,
  say $\alpha_1$, and hence $n_{x_1}, n_{y_1} \ge \frac{1}{4}$ as
  $\partial(\partial D_{|\alpha})=y-x$. Also if $D$ does not hit
  $\alpha_i$, then $x_i=y_i$ and they must lie outside the domain $D$,
  since otherwise we have $n_{x_i}=n_{y_i}\geq \frac{1}{2}$ and hence
  $n_x+n_y$ becomes too large.

  We now note that $e(D)$ can only take half-integral values, and
  thus only the following cases might occur.

\begin{itemize}

\item  $D$ hits $\alpha_1$ and another $\alpha$
  circle, say $\alpha_2$, $D$ consists of squares,
  $n_{x_1}=n_{x_2}=n_{y_1}=n_{y_2}=\frac{1}{4}$, and there are
  $(g+k-3)$ trivial disks.

\item  $D$ hits $\alpha_1$, $D$ consists of
  squares and exactly one bigon, $n_{x_1}=n_{y_1}=\frac{1}{4}$,
  and there are $(g+k-2)$ trivial disks.

\item  $D$ hits $\alpha_1$, $D$ consists of
  squares, $n_{x_1}+n_{y_1}=1$, and there are $(g+k-2)$ trivial
  disks.

 \end{itemize}

Using the reformulation by Lipshitz, in each of these cases, we will
try to figure out the surface $S$ which maps to $\Sigma\times D^2$.
Recall that a trivial disk is a component of $S$ which maps to a point
in $\Sigma$ after post-composing with the projection $\Sigma\times
D^2\rightarrow\Sigma$.

The first case corresponds to a map from $S$ to $\Sigma$ with
$\chi(S)=(g+k-2)$, and $S$ has $(g+k-3)$ trivial disk components. If
the rest of $S$ is $F$, then $F$ is a double branched cover over $D^2$
with one branch point and with $\chi(F)=1$, i.e. $F$ is a disk with
$4$ marked points on its boundary.  Call the marked points corners,
and call $F$ a square.

In the other two cases, $S$ has $(g+k-2)$ trivial disk components, so
if $F$ denotes the rest of $S$, then $F$ is just a single cover over
$D^2$. Thus the number of branch points has to be $0$. But in the
third case the number of branch points is $1$, so the third case
cannot occur. In the second case, $F$ is a disk with $2$ marked points
on its boundary. Call the marked points corners, and call $F$ a bigon.

Thus in both the first and the second cases, $D$ is the image of $F$
and all the trivial disks map to the $x$-coordinates (which are also
the $y$-coordinates) which do not lie in $D$. Note that in both cases,
the map from $F$ to $D$ has no branch point, so it is a local
diffeomorphism, even at the boundary of $F$. Furthermore using the
condition that $n_{x_i}=n_{y_i}=\frac{1}{4}$, we can conclude that
there is exactly one preimage for the image of each corner of $F$.

All we need to show is that the map from $F$ to $\Sigma$ is an
embedding, or in other words, the local diffeomorphism $f:F\rightarrow
D$ is actually a diffeomorphism. First note that it is enough to show
that $f_{|\partial F}$ is an embedding. For then, the image of $\del
F$ under the map $f$ is an embedded circle in $\Sigma$, and it is also
nullhomologous since it bounds the $2$-chain $D$. Therefore, the
circle divides up $\Sigma$ into two components, and the coefficients
of $D$ are constant in each component. Since the coefficients $0$ and
$1$ appear in a neighborhood of $x_1$ and $y_1$, these are the only
two coefficients that appear in $D$, and hence $D$ is an empty
embedded square or an empty embedded bigon.

Now in $F$ (which is a square or a bigon) look at the preimages of all
the $\alpha$ and the $\beta$ circles. Using the fact that $f$ is a
local diffeomorphism, we see that each of the preimages of $\alpha$
and $\beta$ arcs are also $1$-manifolds, and by an abuse of notation,
we also call them $\alpha$ or $\beta$ arcs. Now since $f$ is a local
diffeomorphism, it is easy to see that when $F$ is a square, all the
components of $F\setminus(\alpha\cup\beta)$ are squares, and when $F$
is a bigon, all but one component of $F\setminus(\alpha\cup\beta)$ are
squares, and that component is a bigon. Figure
\ref{fig:squarebigon} shows the induced tiling on $F$ in each of the
cases. Throughout  this chapter, in all  the figures, we  will use the
convention that the thick lines denote the $\alpha$ arcs, and the thin
lines denote the $\beta$ arcs.

\begin{figure}[ht]
\psfrag{p1}{$p$}
\psfrag{p2}{$p$}
\psfrag{q1}{$q$}
\psfrag{q2}{$q$}
\begin{center}
\includegraphics[width=300pt]{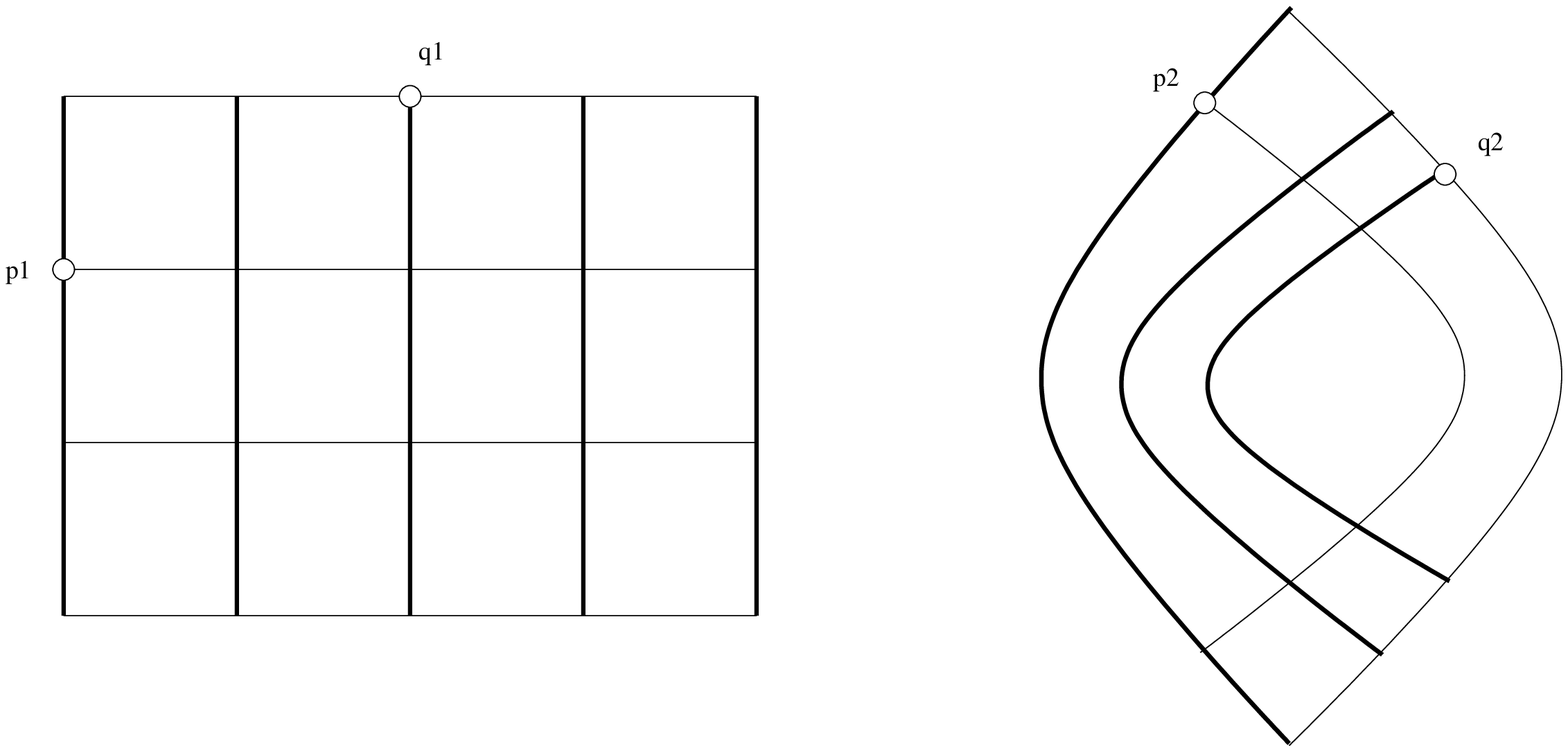}
\end{center}
\caption{Induced tiling on the surface $F$}\label{fig:squarebigon}
\end{figure}

Let the vertices be the intersection points between the $\alpha$ arcs
and the $\beta$ arcs in $F$. Recall that in $F$, some of the vertices
(namely the ones at the corners) are called corners. Also recall that
there is exactly one preimage for the image of each corner. Now assume
if possible, $f_{|\del F}$ is not an embedding. This immediately
implies that there are two distinct vertices $p$ and $q$ lying on
$\del F$ such that $f(p)=f(q)$. Furthermore, if $F$ is a square, then
the opposite sides of $\del F$ either map to different $\alpha$
circles or map to different $\beta$ circles, and hence we can assume
that $p$ and $q$ lie on adjacent sides on the boundary of $F$. (The
proof of admissibility in \cite{RLCMJW} actually shows that the image
of each side of $F$ is embedded and hence $p$ and $q$ can not lie on
the same side). In Figure \ref{fig:squarebigon}, we have marked
certain vertices on $\del F$ as $p$ and $q$. We assume that the
$\alpha$ curve passing through $p$ lies on $\del F$, and similarly
(since $p$ and $q$ lie on adjacent sides) the $\beta$ curve passing
through $q$ lies on $\del F$.

The rest of the proof is fairly straightforward. We will move the
points $p$ and $q$ on $F$ such that $p$ and $q$ remain disjoint and
the condition $f(p)=f(q)$ is still satisfied. Eventually the point $p$
will hit a corner, and that will be a contradiction since the image of
each corner has exactly one preimage. So now move the point $q$ in a
single direction along an $\alpha$ curve (since $q$ started as a point
on a $\beta$ curve lying in $\del F$, the direction of motion is
fixed). To ensure that $f(p)=f(q)$, the point $p$ also starts to move
along an $\alpha$ curve. Note that since $p$ started off as a point on
an $\alpha$ curve lying in $\del F$, $p$ continues to lie on the same
$\alpha$ curve in $\del F$ and it approaches one of the corners of
$F$. Also observe that $p$ encounters a vertex on its way exactly when
$q$ encounters a vertex on its way. Thus it is clear (at least from
Figure \ref{fig:squarebigon}) that irrespective of which direction $p$
is moving, the point $p$ hits a corner no later than when the point
$q$ hits the boundary of $F$ again. This is the required
contradiction.
\end{proof}

\begin{thm}\label{thm:embedded=>holo}\cite{SSJW}
  If $D\in \pi_2^0(x,y)$ is an empty embedded bigon or an
  empty embedded square, then the product complex structure on
  $\Sigma\times D^2$ achieves transversality for $D$ under a
  generic perturbation of the $\alpha$ and the $\beta$ circles, and
  $\mu(D)=c(D)=1$ (with coefficients in $\F_2$).
\end{thm}

\begin{proof}

  Let $D$ be an empty embedded $2n$-gon. Each of the corners of $D$
  must be an $x$-coordinate or a $y$-coordinate, and at every other
  $x$ or $y$-coordinate the point measures $n_{x_i}$ and $n_{y_i}$
  must be zero. Therefore $n_x(D)+n_y(D)= 2n\cdot\frac{1}{4}=
  \frac{n}{2}$. Also $D$ is topologically a disk, so it has Euler
  characteristic $1$. Since it has $2n$ corners each with an angle of
  $\frac{\pi}{4}$, the Euler measure $e(D)$ is equal to $1-\frac{2n}{4}=
  1-\frac{n}{2}$. Thus the Maslov index $\mu(D)=e(D)+n_x(D)+n_y(D)=1$.

  By \cite[Lemma 3.10]{RL}, we see that $D$ satisfies the
  boundary injective condition, and hence under a generic perturbation
  of the $\alpha$ and the $\beta$ circles, the product complex
  structure achieves transversality for $D$.

  When $D$ is an empty embedded square, we can choose the surface $F$ to be a
  disk with $4$ marked points on its boundary, which is mapped to
  $D$ diffeomorphically. Given a complex structure on $\Sigma$, the
  holomorphic structure on $F$ is determined by the cross-ratio of the
  four points on its boundary, and there is an one-parameter family of
  positions of the branch point in $D^2$ which gives that cross-ratio.
  Thus there is a holomorphic branched cover $F\rightarrow D^2$
  satisfying the boundary conditions, unique up to reparametrization.
  Hence the domain of $D$ has a holomorphic representative, and from the proof of
  Theorem \ref{thm:holo=>embedded} we see that this determines the
  topological type of $F$, and hence it is the unique holomorphic
  representative.

  When $D$ is an empty embedded bigon, we can choose $F$ to be a
  disk with $2$ marked points on its boundary, which is mapped to
  $D$ diffeomorphically. A complex structure on $\Sigma$ induces a
  complex structure on $F$, and there is a unique holomorphic map from
  $F$ to the standard $D^2$ after reparametrization. Thus again $D$
  has a holomorphic representative, and similarly it must be the
  unique one.
 \end{proof}

 The upshot of Theorems \ref{thm:holo=>embedded} and
 \ref{thm:embedded=>holo} is the following. With coefficients in
 $\F_2$, the hat version of Heegaard Floer homology of a
 three-manifold $Y$ can be computed combinatorially in a nice pointed
 Heegaard diagram representing $Y$. The story for knot Floer homology
 is similar. A nice pointed Heegaard diagram for a knot $K\subset Y$
 is a Heegaard diagram for the knot, which when viewed as a Heegaard
 diagram for $Y$ (by forgetting the $w$-basepoints) is a nice pointed
 diagram. It is immediate that given a nice pointed diagram for a
 knot, the hat version of knot Floer homology can be computed
 combinatorially with coefficients in $\F_2$. We will return to the
 case for knots in $S^3$ in the next chapter.

 However, despite having these theorems, we are actually quite far
 from having an algorithm to compute the hat version of the invariant
 (with coefficients in $\F_2$). To be able to achieve that, we need an
 algorithm which inputs a Heegaard diagram for $Y$, and outputs a nice
 pointed Heegaard diagram for the same three-manifold. We would also
 require the output to be an admissible Heegaard diagram, but as we
 have already mentioned, for nice diagrams, (weak) admissibility comes
 for free, so we will not bother with admissibility issues in the
 future. With this in mind, we head on to the next section, which does
 exactly what it claims to do, which happens to be precisely what we
 need.

\section{An algorithm for being nice}

Let $(\Sigma_g,\alpha_1,\ldots,\alpha_{g+k-1},\beta_1,\ldots,
\beta_{g+k-1}, z_1,\ldots,z_k)$ be a Heegaard diagram for $Y$. Before
describing the algorithm, let us recall a few notations. A region is a
component of $\Sigma\setminus(\alpha\cup\beta)$. A $2n$-gon is a
region which is topologically an open disk, and has $2n$ vertices on
its boundary, where a vertex is an intersection between an $\alpha$
curve and a $\beta$ curve, both lying on the boundary of the region.

We will gradually modify the Heegaard diagram, such that it throughout
remains a Heegaard diagram for $Y$, and eventually it becomes a nice
pointed Heegaard diagram. The only modifications that we will do are
isotopies and handleslides (in fact mostly it will be isotopies), and
this ensures that all through the process the Heegaard diagram
represents $Y$.

Let $D_1,D_2,\ldots,D_m$ be the regions that do not contain the
basepoints, and let $B_i$ be the region that contains the basepoint $z_i$.
For a region $D_i$, let $\chi(D_i)$ be its Euler characteristic, and
let $e(D_i)$ be its Euler measure (which is simply
$\chi(D_i)-\frac{v}{4}$ where $v$ is the number of vertices on the
boundary of $D_i$). Define the reduced Euler measure $\wt{e}(D_i)=
\mathrm{min}\{0,e(D_i)\}$. Let 
$\chi(\mc{H})=\sum_i(1-\chi(D_i))$ and let
$e(\mc{H})=-2\sum_i\wt{e}(D_i)$.

Recall that each component of
$\Sigma\setminus\alpha$ contains some basepoint $z_j$. Let the
distance $d(D_i)$ be the smallest number of times we have to cross the
$\alpha$ arcs to get from a point in the interior of $D_i$ to a
basepoint. If $\mc{H}$ is nice, define $d(\mc{H})=0$, otherwise define
it to be the smallest distance of a region which does not contain a
basepoint and is not a bigon or a square.

Given a Heegaard diagram $\mc{H}$, we define its complexity
$c(\mc{H})=(\chi(\mc{H}),e(\mc{H}), d(\mc{H}))$. After observing that
each of the entries are non-negative integers, we order all the
complexities lexicographically (the crucial thing is that the ordering
is a well-ordering). A careful reader will also notice that this
definition of complexity is slightly different from the one in
\cite{SSJW}.

We make one final observation before we state the main result of this
section and embark upon its proof. A Heegaard diagram $\mc{H}$ is nice
if and only if $c(\mc{H})=(0,0,0)$ (which happens if and only if the
first two entries are zero). It is clear now that the following
theorem provides the required algorithm to make a Heegaard diagram
nice and achieves the purpose of this section.

\begin{thm}
If $\mc{H}$ is a Heegaard diagram which is not a nice diagram, then we
can modify $\mc{H}$ by isotopies and handleslides to get a new
Heegaard diagram $\mc{H}'$ with a smaller complexity.
\end{thm}

\begin{proof}
In fact we will prove something stronger than what we stated in the
theorem. We will actually explicity produce a sequence of isotopies
and handleslides that decreases the complexity.

We start with a Heegaard diagram $\mc{H}$ which is not nice. Let the
complexity $c(\mc{H})=(\chi(\mc{H}), e(\mc{H}), d(\mc{H}))$. Since
$\mc{H}$ is not nice, we know that $(\chi(\mc{H}),e(\mc{H}))\neq
(0,0)$.

We will break up the proof into lots of cases, and we will try to keep
the cases as organized as possible. Keeping that in mind, let us
proceed by immediately dividing up the proof into two cases. 

\subsection*{Case 1: $\chi(\mc{H})\neq 0$} In this case, there is a
region (say $D_1$) which is topologically not a disk. However since
the $\alpha$ circles (and also the $\beta$ circles) span a
half-dimensional subspace of $H_1(\Sigma)$, the region $D_1$ has genus
zero. Thus $D_1$ has more than one boundary component. We will do an
isotopy, after which the number of boundary components of $D_1$ will
decrease by one, and Euler characteristic of all the other regions
will remain unchanged.

Note that not all the boundary components of $D_1$ can consist of just
$\alpha$ circles. For then $D_1$ will be a component of
$\Sigma\setminus\alpha$, but each such component contains a basepoint
and $D_1$ does not contain a basepoint. Similarly not all boundary
components of $D_1$ can be $\beta$ circles. Therefore, somewhere on
$\del D_1$ there is an $\alpha$ arc, and on some other component of
$\del D_1$ there is a $\beta$ arc. Join such an $\alpha$ arc to such a
$\beta$ arc by an embedded path in $D_1$, and then do an isotopy of
the $\alpha$ curve along this path until it hits the $\beta$ curve.
Such an isotopy is called a finger move, and we will constantly be
using such moves. Figure \ref{fig:fingermove} illustrates the relevant
finger move. It is clear that this isotopy reduces $\chi(\mc{H})$ and
thus decreases the complexity.

\begin{figure}[ht]
\psfrag{d}{$D_1$}
\begin{center}
\includegraphics[width=300pt]{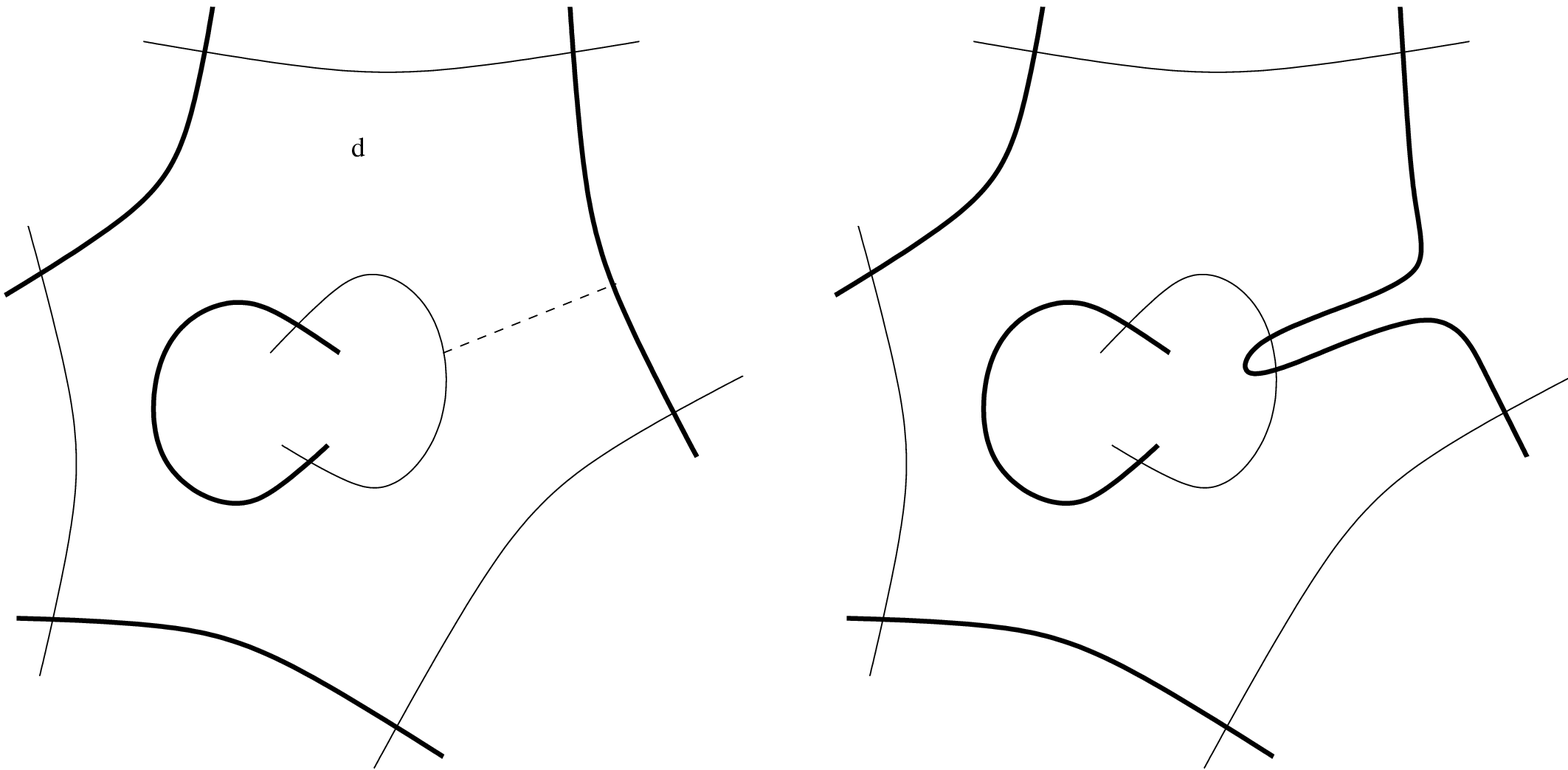}
\end{center}
\caption{Reducing $\chi(D_1)$}\label{fig:fingermove}
\end{figure}

\subsection*{Case 2: $\chi(\mc{H})=0$} In this case, all the regions
not containing any basepoints are topological disks. Since $\mc{H}$ is
not a nice diagram, we must have $e(\mc{H})\neq 0$. We will do a
sequence of moves such that at the end, all the regions not containing
any basepoints are still disks (i.e. $\chi(\mc{H})=0$), and we have
either reduced $e(\mc{H})$ or we have kept $e(\mc{H})$ constant and
reduced $d(\mc{H})$ (or in other words, we have decreased the value of
the pair $(e,d)$).

Let $D_1$ be a $2n$-gon region not containing any basepoint such that
$n>2$ and $d(D_1)=d(\mc{H})$. We call a path $\gamma:[0,1]\rightarrow
\Sigma$ to be $\alpha$-avoiding if it is an embedded path and it lies
in the complement of all the $\alpha$ circles, and we call such an
$\alpha$-avoiding path to be tight if it never enters a region and
then immediately proceeds to leave the region through the same $\beta$
arc. We similarly define $\beta$-avoiding paths and tight
$\beta$-avoiding paths.

Let $\gamma:[0,1]\rightarrow \Sigma$ be a $\beta$-avoiding path
joining a point in the interior of $D_1$ to a basepoint which
intersects the $\alpha$ arcs exactly $d(D_1)$ ($=d(\mc{H})$)
times. Clearly $\gamma$ is a tight $\beta$-avoiding curve. Let the
sides of the region $D_1$ be named (going counterclockwise)
$a_1,b_1,a_2,b_2,\cdots,a_n,b_n$, where the $a_i$'s are the $\alpha$
arcs, the $b_i$'s are the $\beta$ arcs, and $\gamma$ enters $D_1$
through $a_1$.

Let $\delta:[0,1]\rightarrow \Sigma$ be a tight $\alpha$-avoiding path
which joins $\gamma(1)$ to either a basepoint or a point
inside a bigon, and whose only intersection with $\gamma$ inside $D_1$
is at $\gamma(1)$. We know that there is at least one such path, since
$D_1$ can be joined to a basepoint by a tight $\alpha$-avoiding path. 

The notations are already set up, so it is about time to divide this
case into two further subcases. The main idea of the proof is already
present in the first subcase, the second subcase is just there to
round up a few other situations.

\textbf{Subcase 2a:} The path $\delta$ can be chosen such that it does
not leave $D_1$ through either $b_1$ or $b_n$.

We choose such a path $\delta$ which does not leave $D_1$ through
either $b_1$ or $b_n$. Recall that the paths $\gamma$ and $\delta$ are
embedded, but there could be intersections between $\gamma$ and
$\delta$ (but no such intersections inside $D_1$ except at
$\gamma(1)=\delta(0)$). Let $S\subset(0,1]$ be the set of all points $s$,
such that $\delta(s)$ lies in either the image of $\gamma$ or in a
bigon region or in a region $B_i$ containing some basepoint. Let $p$
be the image under $\delta$ of a point lying in the leftmost connected
component of $S$.

We choose the last $\alpha$ arc on the path $\gamma$, and we do a
finger move of that along the remaining part of $\gamma$ and then
along $\delta$ and stop just before $p$. After doing this isotopy,
note that $\gamma$ intersects the $\alpha$ arcs one fewer time. Now
again choose the last $\alpha$ arc on $\gamma$ (which was
originally the second to last $\alpha$ arc on $\gamma$), then do a finger
move along the remaining part of $\gamma$ and then along $\delta$ for
as long as we can (we have to stop short of $p$). We repeat this
process until we have done finger moves on all the $\alpha$ arcs that
originally hit $\gamma$. This is a total of $d(\mc{H})$ finger moves,
and we can view this process as a single multi-finer move. In the
future also, we will be using such multi-finger moves. 

Figure \ref{fig:multifinger} illustrates the case when $p$ lies on
$\gamma$. It is fairly clear (at least from the figure) that the pair
$(\chi,e)$ is unchanged after such a move, but however after the move
the region $D'$ is a hexagon with $d(D')<d(D_1)=d(\mc{H})$. Thus after
this multi-finger move isotopy, the new Heegaard diagram has a smaller
complexity.

\begin{figure}[ht]
\psfrag{a1}{$a_1$}
\psfrag{a2}{$a_2$}
\psfrag{a3}{$a_3$}
\psfrag{b1}{$b_1$}
\psfrag{b2}{$b_2$}
\psfrag{b3}{$b_3$}
\psfrag{d}{$D_1$}
\psfrag{d1}{$D'$}
\psfrag{p}{$p$}
\psfrag{de}{$\delta$}
\psfrag{ga}{$\gamma$}
\begin{center}
\includegraphics[width=330pt]{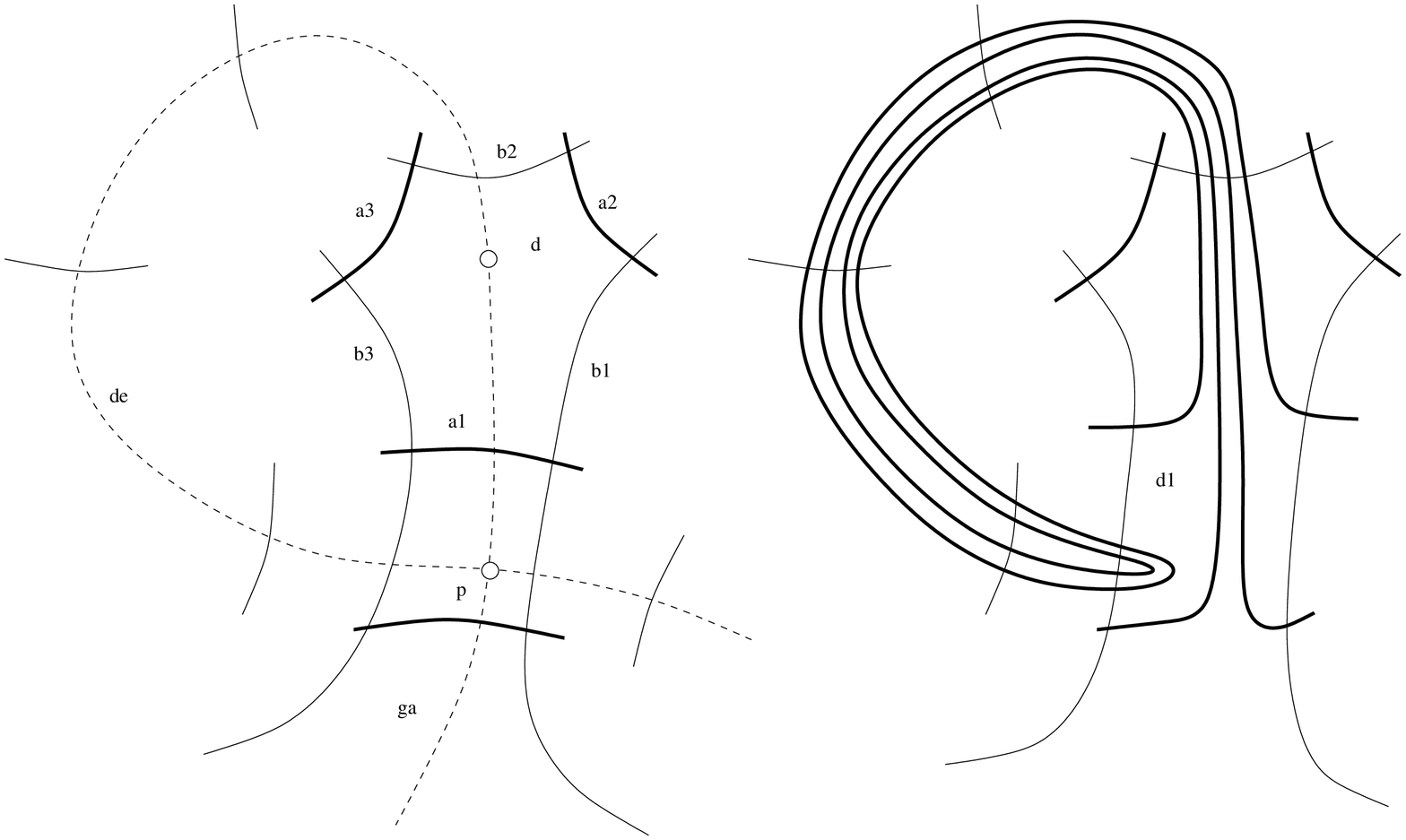}
\end{center}
\caption{Reducing complexity when $p\in\gamma$}\label{fig:multifinger}
\end{figure}

Figure \ref{fig:multifinger1} illustrates the case when $p$ lies in a
region $B$, which is either a bigon or a region with a basepoint. If
$B$ is a bigon, then after this multi-finger move $B$ becomes a
square, and hence $e$ decreases ($\chi$ however remains constant).
Similarly, if $B$ is a region containing a basepoint, then $\chi$
remains constant and $e$ decreases. Therefore in either case, the
complexity of the Heegaard diagram decreases after the multi-finger
move.

\begin{figure}[ht]
\psfrag{a1}{$a_1$}
\psfrag{a2}{$a_2$}
\psfrag{a3}{$a_3$}
\psfrag{b1}{$b_1$}
\psfrag{b2}{$b_2$}
\psfrag{b3}{$b_3$}
\psfrag{d}{$D_1$}
\psfrag{b}{$B$}
\psfrag{p}{$p$}
\psfrag{de}{$\delta$}
\psfrag{ga}{$\gamma$}
\begin{center}
\includegraphics[width=330pt]{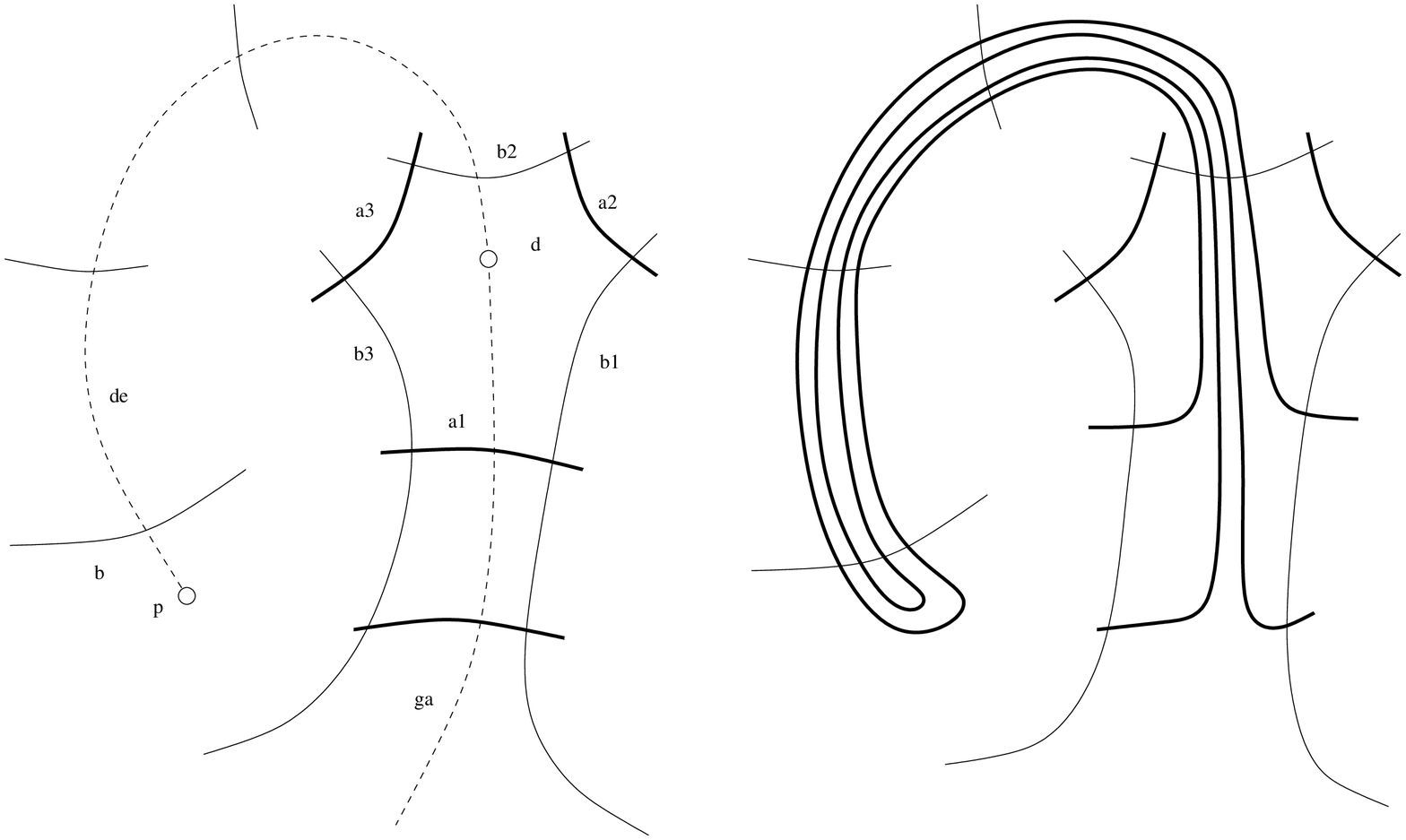}
\end{center}
\caption{Reducing complexity when $p\in B$}\label{fig:multifinger1}
\end{figure}
 
\textbf{Subcase 2b:} All such paths $\delta$ leave $D_1$ either
through $b_1$ or $b_n$.
 
We again choose a path $\delta$ which leaves $D_1$ through either
$b_1$ or $b_n$, and without loss of generality, let us assume it
leaves $D_1$ through $b_1$. Even though it is not necessary, we choose
$\delta$ such that $\delta$ enters each region at most once.

We now try to create another tight $\alpha$-avoiding path
$\varepsilon$ which starts at $\gamma(0)$ and does not leave $D_1$
though either $b_1$ or $b_n$. We also assume that inside $D_1$, the
path $\varepsilon$ intersects $\gamma$ and $\delta$ only at
$\gamma(1)=\delta(0)$.  Starting the construction of $\varepsilon$ is
easy. We start at $\gamma(1)$ and immediately leave $D_1$ through some
$b_i$ with $1<i<n$ (since $n>2$, such an $i$ always exists). The way
we keep constructing the path $\varepsilon$ is the following. After we
enter a region through some $\beta$ arc, we leave that region through
a different $\beta$ arc (this ensures tightness). We can clearly keep
doing this unless we hit a bigon. So to construct $\varepsilon$, we
basically continue this process, until we either enter a bigon, or
enter a region $B_i$ containing some basepoint, or enter a region hit
by $\delta$ or enter a region previously visited by $\varepsilon$
(this includes re-entering $D_1$). Since there are only finitely many
regions in $\mc{H}$, at least one of the above must happen, and we
stop immediately after one of these things happens. While making the
proof more cumbersome, this naturally leads to further subcases.

\begin{itemize}
\item \emph{2b.i:} $\varepsilon$ enters a bigon or a region containing
  a basepoint.
\end{itemize} 

This is a direct contradiction to the assumption of Subcase 2b, because
$\varepsilon$ is a perfectly good candidate for the required tight
$\alpha$-avoiding path.

\begin{itemize}
\item \emph{2b.ii:} $\varepsilon$ enters a region hit by $\delta$ other
  than $D_1$.
\end{itemize} 

Let $D_2$ be the region (other than $D_1$) that is visited by both
$\varepsilon$ and $\delta$. Let $b$ be the $\beta$ arc through which
$\delta$ leaves $D_2$, and let $b'$ be the $\beta$ arc through which
$\varepsilon$ enters $D_2$. Since $D_2$ is the only region other than
$D_1$ which is visited by both $\varepsilon$ and $\delta$, the arc $b$
is different from the arc $b'$. Thus we can construct a new tight
$\alpha$-avoiding path $\delta'$ joining $\gamma(1)$ to either a bigon
or a basepoint, as shown in Figure \ref{fig:subcase2bii}. The new path
$\delta'$ basically follows $\varepsilon$ until $D_2$, then leaves
$D'$ through $b$, and follows $\delta$ for the rest of the way (the
observation that $b\neq b'$ ensures tightness at $D_2$). This new path
$\delta'$ again provides a contradiction to the assumption of Subcase
2b.

\begin{figure}[ht]
\psfrag{a1}{$a_1$}
\psfrag{a2}{$a_2$}
\psfrag{b1}{$b_1$}
\psfrag{b2}{$b_2$}
\psfrag{b}{$b$}
\psfrag{bb}{$b'$}
\psfrag{d}{$\delta$}
\psfrag{g}{$\gamma$}
\psfrag{e}{$\varepsilon$}
\psfrag{dd}{$\delta'$}
\psfrag{d1}{$D_1$}
\psfrag{d2}{$D_2$}
\begin{center}
\includegraphics[width=330pt]{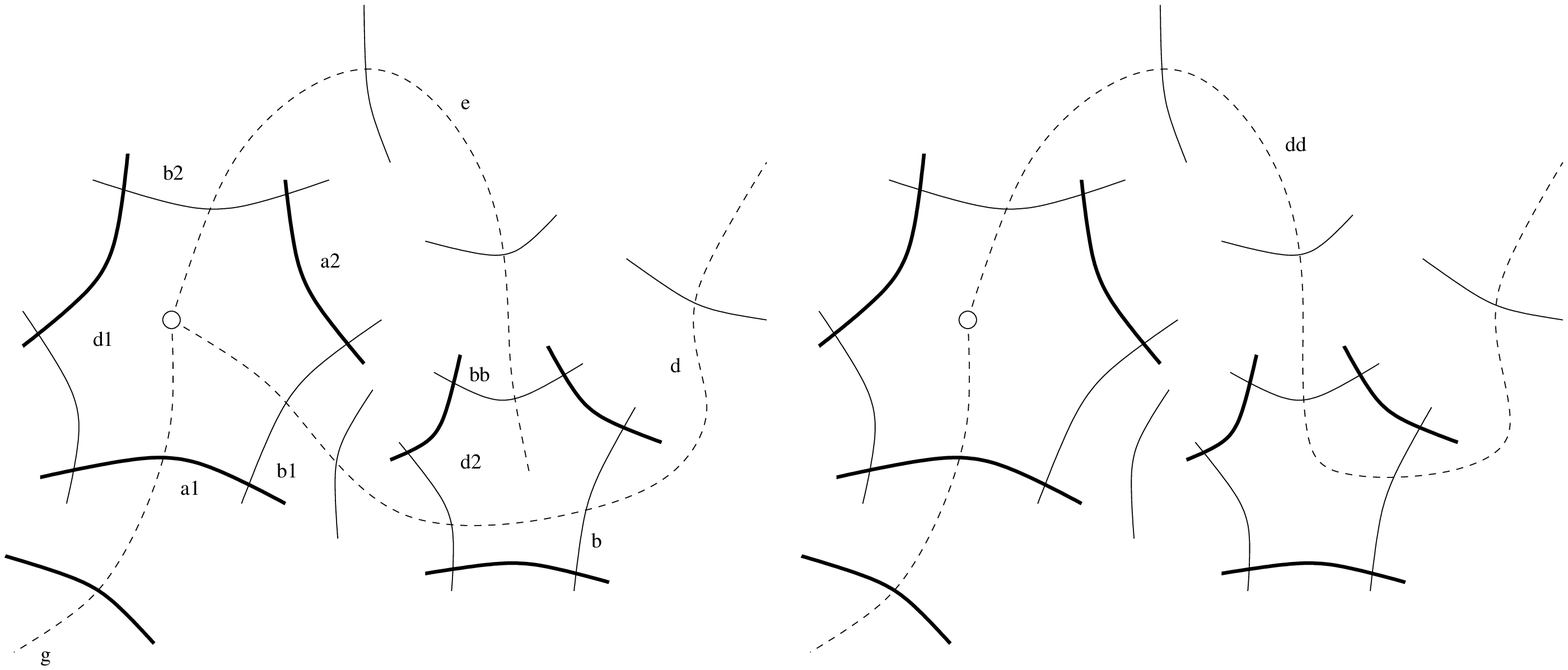}
\end{center}
\caption{Constructing $\delta'$ in Subcase 2b.ii}\label{fig:subcase2bii}
\end{figure}

\begin{itemize}
\item \emph{2b.iii:} $\varepsilon$ re-enters $D_1$ through an arc other than $b_n$.
\end{itemize} 

Once more we try to construct a path $\delta'$ which will provide a
contradiction to the assumption of Subcase 2b. If $\varepsilon$
re-enters $D_1$ on the same side of $\gamma(1)$ as $b_1$, the we can
basically construct $\delta'$ in the same way as in Subcase 2b.ii
(even though the counterexample path $\delta'$ will hit the region
$D_1$ twice). This is shown in Figure \ref{fig:subcase2biiia}. Note that
to ensure tightness of $\delta'$, $\varepsilon$ should not re-enter
$D_1$ through $b_1$, but if that happens, we are actually in Subcase 2b.ii.

\begin{figure}[ht]
\psfrag{a1}{$a_1$}
\psfrag{b1}{$b_1$}
\psfrag{d}{$\delta$}
\psfrag{g}{$\gamma$}
\psfrag{e}{$\varepsilon$}
\psfrag{dd}{$\delta'$}
\psfrag{d1}{$D_1$}
\begin{center}
\includegraphics[width=300pt]{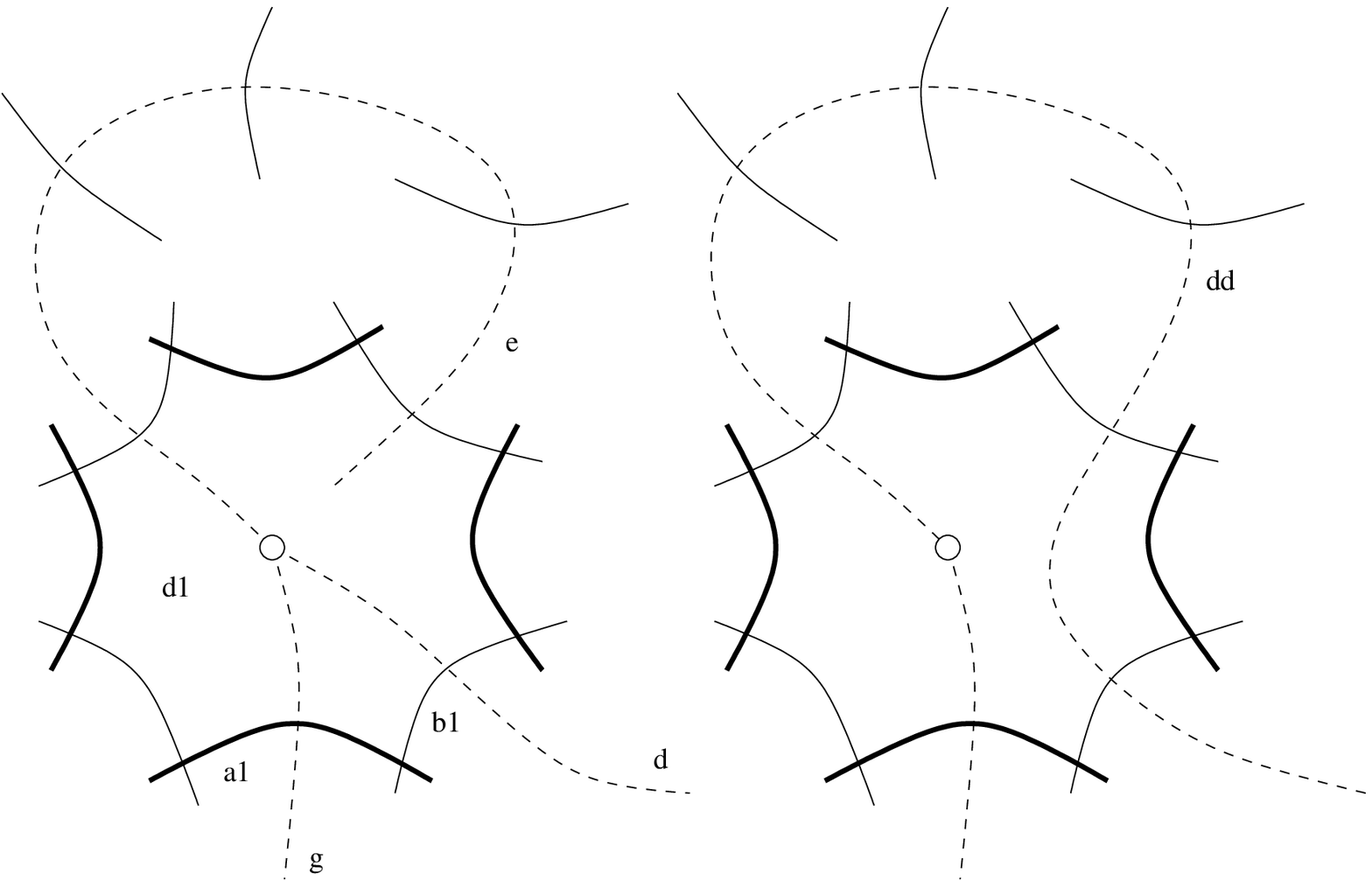}
\end{center}
\caption{Constructing $\delta'$ in Subcase 2b.iii}\label{fig:subcase2biiia}
\end{figure}

On the other hand, if $\varepsilon$ returns to $D_1$ on the other side
of $\gamma(1)$ as $b_1$, then the above construction will yield a path
$\delta'$ which either intersects itself inside $D_1$ or intersects
$\gamma$ inside $D_1$ (and neither is allowed). The way to fix this is
very simple. We just reverse the orientation on $\varepsilon$ to get a
new path $\varepsilon'$, which then returns to $D_1$ on the same side
of $\gamma(0)$ as $b_1$, and we carry out the above construction with
$\varepsilon'$ and $\delta$ to get the required counterexample
$\delta'$. The reversal of orientation on $\varepsilon$ to get
$\varepsilon'$ is shown in Figure \ref{fig:subcase2biiib}. Note how
the assumption that $\varepsilon$ does not return through $b_n$ is
crucial in this argument, for if it did, then after the orientation
reversal, $\varepsilon'$ will leave $D_1$ through $b_n$, a situation
that is not desirable.

\begin{figure}[ht]
\psfrag{a1}{$a_1$}
\psfrag{b1}{$b_1$}
\psfrag{d}{$\delta$}
\psfrag{g}{$\gamma$}
\psfrag{e}{$\varepsilon$}
\psfrag{ee}{$\varepsilon'$}
\psfrag{d1}{$D_1$}
\begin{center}
\includegraphics[width=300pt]{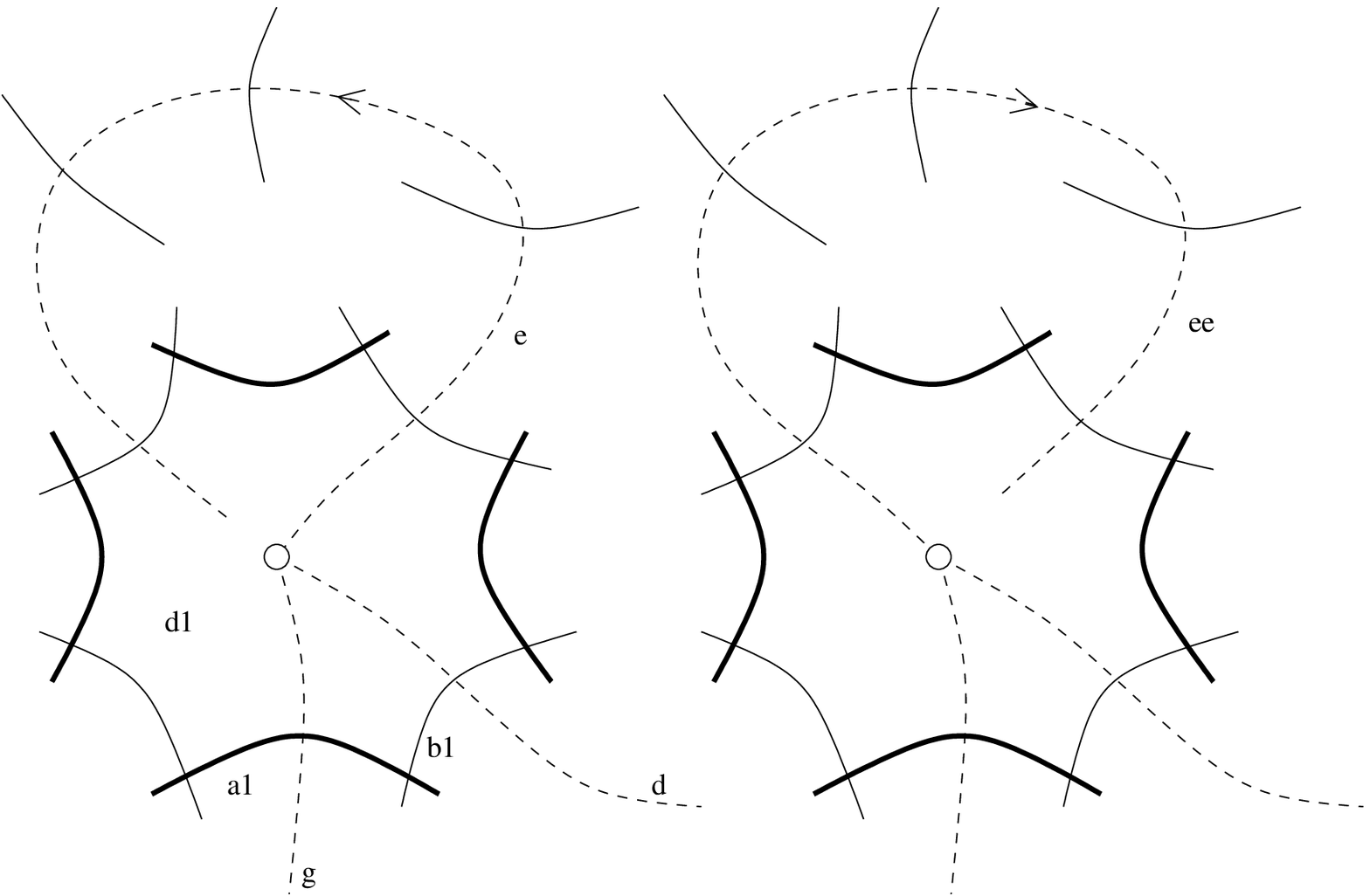}
\end{center}
\caption{Constructing $\varepsilon'$ in Subcase 2b.iii}\label{fig:subcase2biiib}
\end{figure}

\begin{itemize}
\item \emph{2b.iv:} $\varepsilon$ hits a region already visited by
  $\varepsilon$ other than $D_1$.
\end{itemize} 

This subcase can easily be reduced to Subcase 2b.iii. Let $D_2$ be the
region where $\varepsilon$ enters a region (other than $D_1$) that it
has already visited. Let $b$ be the arc through which $\varepsilon$
entered $D_2$ for the first time, and let $b'$ be the arc through
which $\varepsilon$ enters $D_2$ for the second time. Since $D_2$ is
the only region that $\varepsilon$ has visited twice, it is immediate
that $b\neq b'$. We construct a new tight $\alpha$-avoiding path
$\varepsilon'$ in the following way. The initial part of the path
$\varepsilon'$ is same as $\varepsilon$, and then from $D_2$ onwards,
$\varepsilon'$ just follows $\varepsilon$ back to $D_1$. (Tightness of
$\varepsilon'$ is ensured by the observation that $b\neq b'$). Clearly
$\varepsilon'$ re-enters $D_1$ through the same arc that it left
(which is neither $b_1$ nor $b_n$), and thus the tight
$\alpha$-avoiding path $\varepsilon'$ provides a reduction to Subcase
2b.iii. The construction of $\varepsilon'$ is shown in Figure
\ref{fig:subcase2biv}.

\begin{figure}[ht]
\psfrag{d}{$\delta$}
\psfrag{g}{$\gamma$}
\psfrag{e}{$\varepsilon$}
\psfrag{ee}{$\varepsilon'$}
\psfrag{d1}{$D_1$}
\psfrag{d2}{$D_2$}
\psfrag{b}{$b$}
\psfrag{bb}{$b'$}
\begin{center}
\includegraphics[width=300pt]{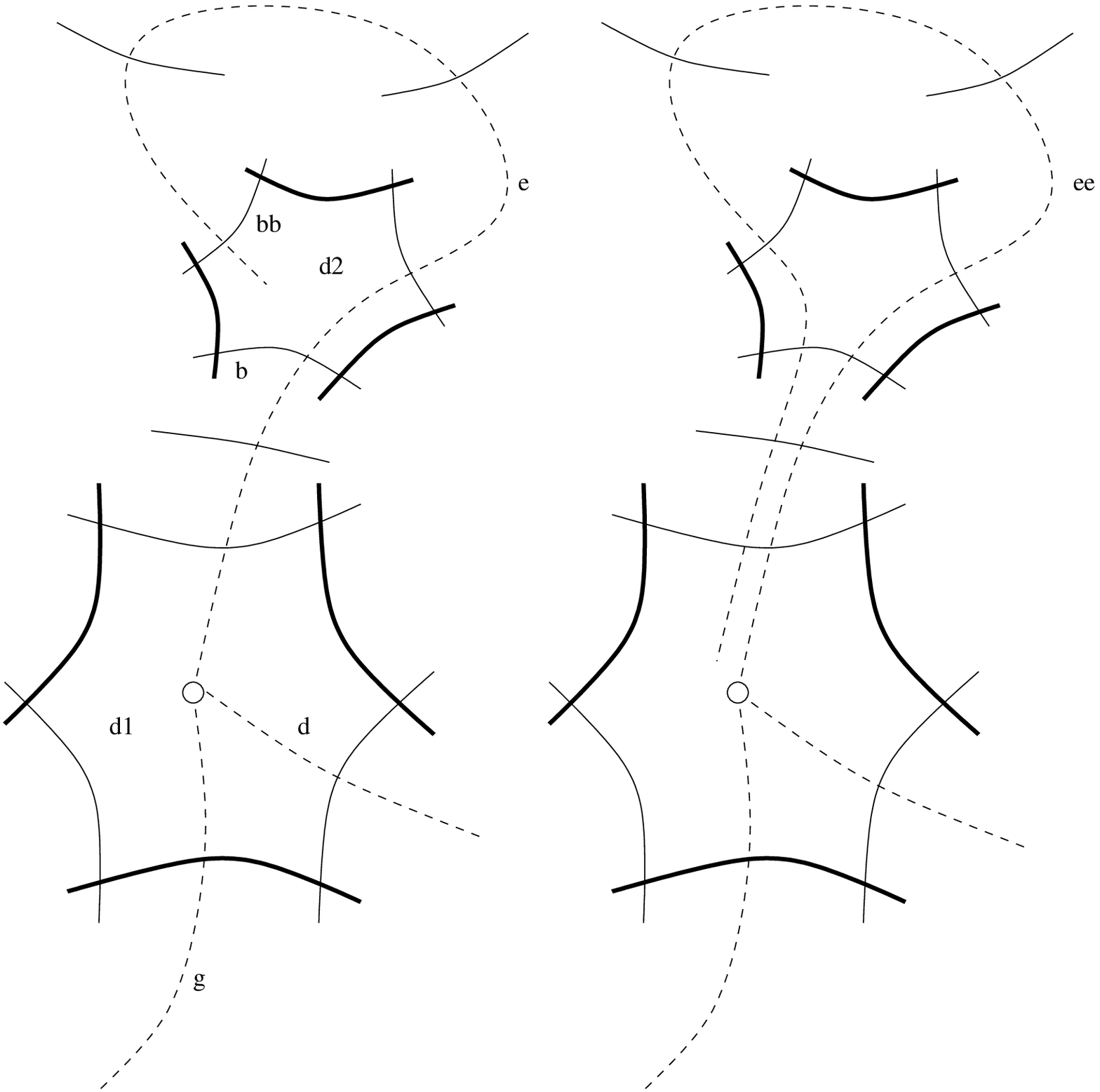}
\end{center}
\caption{Constructing $\varepsilon'$ in Subcase 2b.iv}\label{fig:subcase2biv}
\end{figure}

\begin{itemize}
\item \emph{2b.v:} $\varepsilon$ re-enters $D_1$ through $b_n$.
\end{itemize} 

We are almost done with the proof. This is the very final subcase to
consider. Let the regions that $\varepsilon$ visits be (in order)
$D_1,D_2,\ldots,D_l,D_1$. If $D_1$ is a hexagon, and all the regions
$D_2,\ldots,D_l$ are squares, then we had no choice in the
construction of $\varepsilon$. However if $D_1$ is not a hexagon or if
any of the regions $D_2,\ldots,D_l$ is not a square, then somewhere
along the way, we had a choice while constructing $\varepsilon$. As
unfortunate as it might be, this breaks up Subcase 2b.v into two
further subcases.

\emph{2b.v':} Either $D_1$ is not a hexagon, or one of $D_2,\ldots,D_l$ is
not a square.

We have already constructed one tight $\alpha$-avoiding path
$\varepsilon$ which left $D_1$ through an arc other than $b_1$ or
$b_n$, and returned to $D_1$ through $b_n$. However, by the
hypothesis, we had a choice while constructing $\varepsilon$. Let us
now choose another such tight $\alpha$-avoiding path $\varepsilon'$
which also leaves $D_1$ through an arc other than $b_1$ or $b_n$. If
$\varepsilon'$ satisfies the conditions of one of the subcases from
2b.i to 2b.iv, we are done after having reduced this case to an
earlier one. Therefore assume that $\varepsilon'$ also returns to
$D_1$ through $b_n$.

Let $D$ be the first region such that $\varepsilon$ and $\varepsilon'$
agree upto $D$ and start to disagree immediately after that, or in
other words, they leave $D$ through different $\beta$ arcs ($D$ could
very well be $D_1$). Let $D'$ be the next region after $D$ along
$\varepsilon$, which is also visited by $\varepsilon'$. Since both
$\varepsilon$ and $\varepsilon'$ return to $D_1$ through $b_n$, there
is such a region $D'$, and it is not $D_1$. Let $b$ and $b'$ be the arcs
through which $\varepsilon$ and $\varepsilon'$ enter $D'$ (it is
fairly clear that $b\neq b'$). We now construct a new path
$\varepsilon''$ as follows. Travel along $\varepsilon$ all the way
upto $D'$, and then travel back along $\varepsilon'$. Note that since
$b\neq b'$,  $\varepsilon''$ is tight, and since $\varepsilon''$
returns to $D_1$ through the arc through which $\varepsilon'$ left
$D_1$ (which is neither $b_1$ nor $b_n$), the path $\varepsilon''$
reduces this subcase to either Subcase 2b.iii or 2b.iv. The path
$\varepsilon''$ is shown in Figure \ref{fig:subcase2bva}

\begin{figure}[ht]
\psfrag{de}{$\delta$}
\psfrag{g}{$\gamma$}
\psfrag{e}{$\varepsilon$}
\psfrag{ee}{$\varepsilon'$}
\psfrag{eee}{$\varepsilon''$}
\psfrag{d1}{$D_1$}
\psfrag{d}{$D$}
\psfrag{dd}{$D'$}
\psfrag{b}{$b$}
\psfrag{bb}{$b'$}
\begin{center}
\includegraphics[width=330pt]{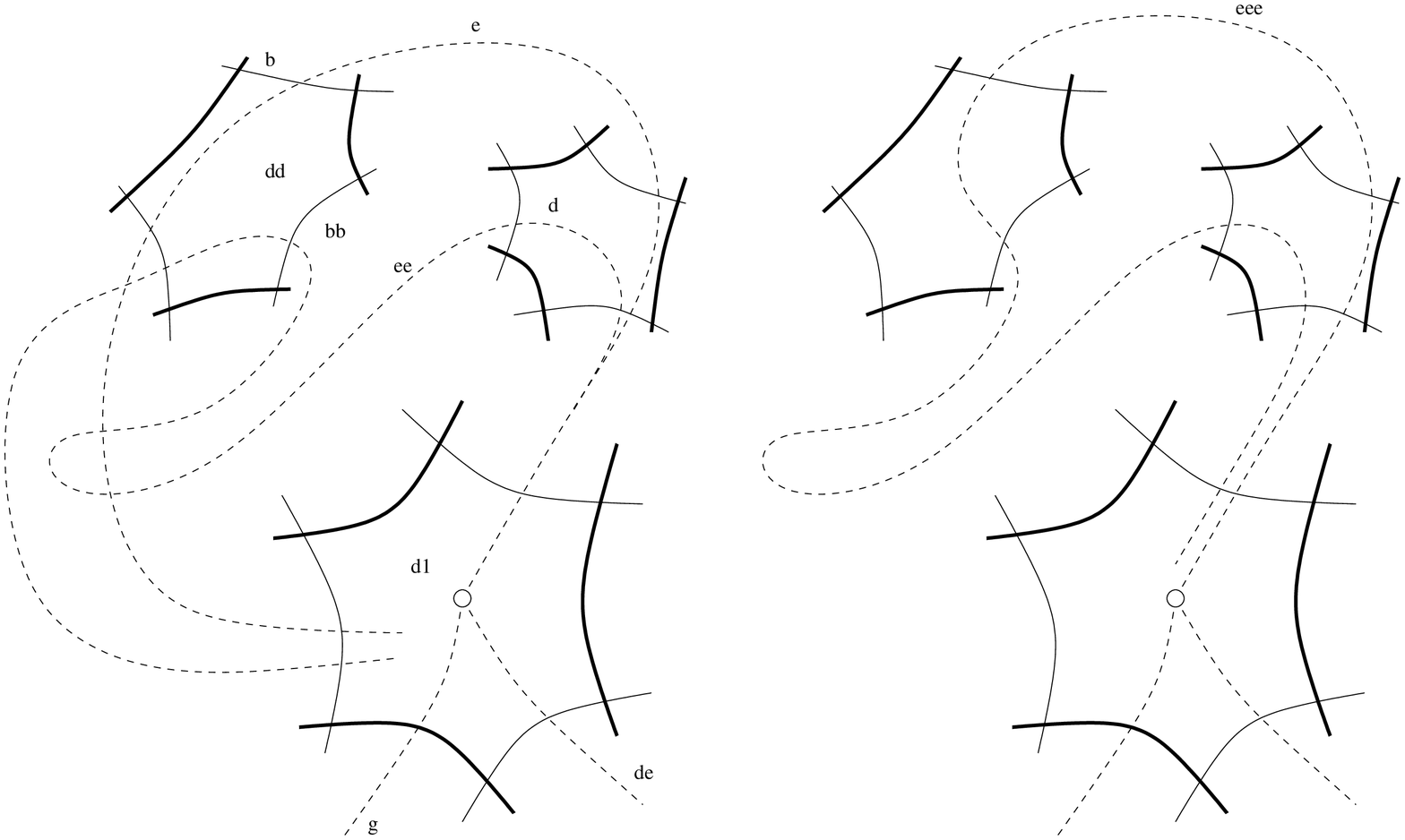}
\end{center}
\caption{Constructing $\varepsilon''$ in Subcase 2b.v'}\label{fig:subcase2bva}
\end{figure}

\emph{2b.v'':}$D_1$ is a hexagon, and each of $D_2,\ldots,D_l$ is a square.

This, we promise beforehand, is the final subcase. Since $D_1$ is a
hexagon, and each of $D_2,\ldots,D_l$ is a square, locally the
Heegaard diagram $\mc{H}$ looks like the first picture in
Figure \ref{fig:subcase2bvb}, with the paths $\gamma$ and $\varepsilon$
shown. Observe that the $\alpha$ arcs to the left of $\varepsilon$ at
each region join to form a whole $\alpha$ circle.  Without loss of
generality, let us call this $\alpha$ circle $\alpha_1$.

\begin{figure}[ht]
\psfrag{ga}{$\gamma$}
\psfrag{e}{$\varepsilon$}
\psfrag{d1}{$D_1$}
\psfrag{d2}{$D_2$}
\psfrag{d3}{$D_3$}
\psfrag{d4}{$D_4$}
\psfrag{d5}{$D_5$}
\psfrag{a1}{$a_1$}
\psfrag{b1}{$b_1$}
\psfrag{a2}{$a_2$}
\psfrag{b2}{$b_2$}
\psfrag{a3}{$a_3$}
\psfrag{b3}{$b_3$}
\psfrag{a}{$\alpha_1$}
\begin{center}
\includegraphics[width=330pt]{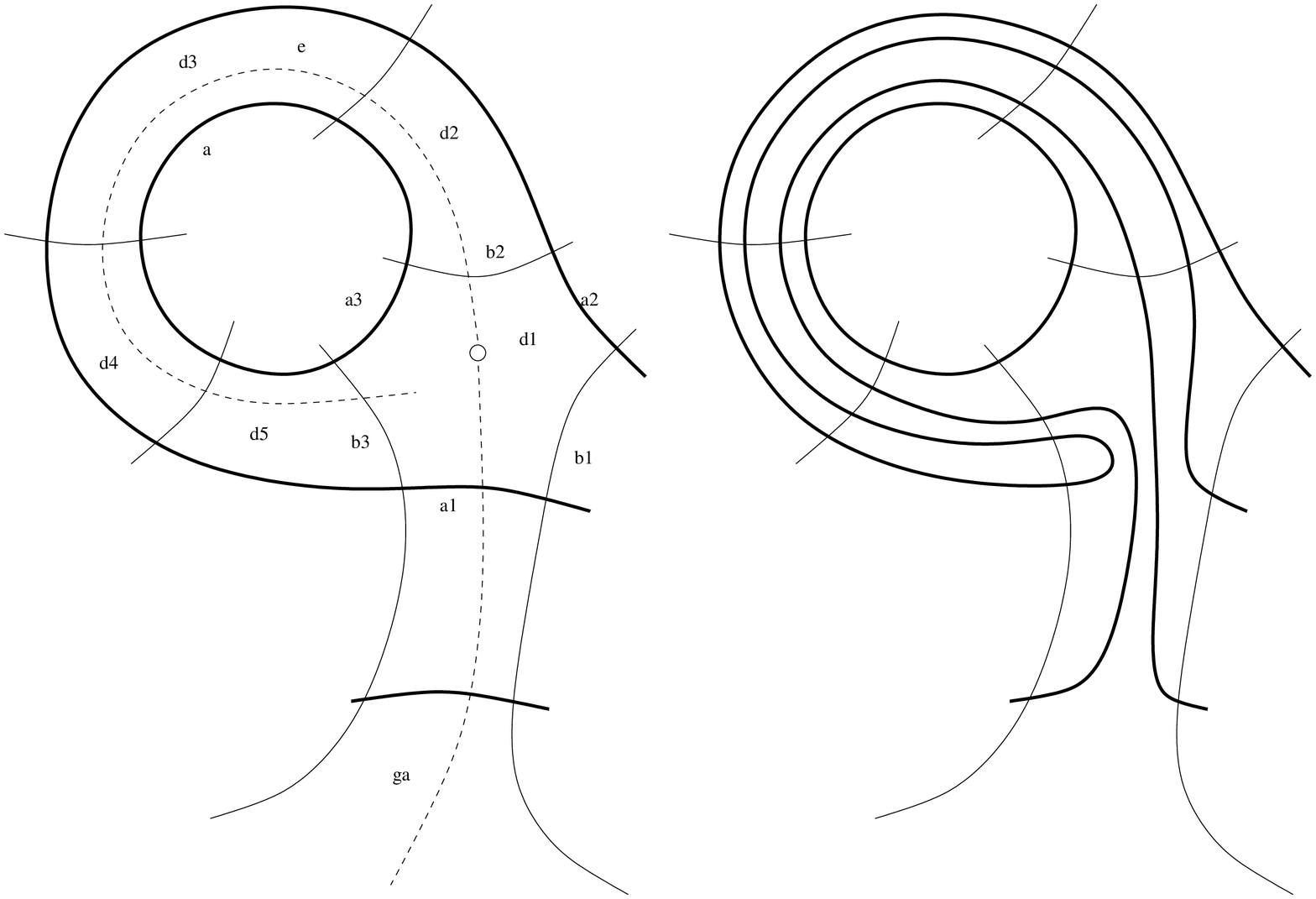}
\end{center}
\caption{Handleslides to reduce complexity}\label{fig:subcase2bvb}
\end{figure}

Our method is very similar to that in Subcase 2a. Take the last
$\alpha$ arc on $\gamma$, and then do a finger move on it along
$\gamma$ upto $\gamma(1)$, and then handleslide it along $\alpha_1$.
Then consider the $\alpha$ arc which is now the last one on $\gamma$
(which was previously the second to last one), and do the same thing.
Repeat this process until we are done with all the $\alpha$ arcs on
$\gamma$. We can view this process as a multi-handleslide move, as
shown in Figure \ref{fig:subcase2bvb}. If $\mc{H}'$ is the new
Heegaard diagram, it is clear that $\chi(\mc{H}')=\chi(\mc{H})$ and
$e(\mc{H}')<e(\mc{H})$, and hence complexity decreases, thus finishing
the proof.
\end{proof}

The proof was fairly long, but it was straightforward. The idea was to
push all the negative Euler measure to the regions containing the
basepoints. We did it essentially by using the fact that each
component of $\Sigma\setminus\alpha$ has a basepoint, and each
component of $\Sigma\setminus\beta$ has a basepoint. There were
certain situations where the method did not work, but then we took
cases, and applied new methods to tackle those cases. This led to
more and more cases, and newer and newer methods, until we were stuck
with a very special case, and in that situation, simple handleslides
did the trick. The reader is advised to also read the algorithm
presented in \cite{SSJW}, which is very similar to (if not the same
as) the algorithm presented above. Algorithms to make a Heegaard
diagrams nice are usually messy (as are the final Heegaard diagrams),
and more often than not, the algorithm itself is far less important
than the final nice Heegaard diagram, where the computation of
$\wh{HF}$ with coefficients in $\F_2$ can be carried out
combinatorially. Purely for amusement, we present a nice Heegaard
diagram for the Poincar\'{e} homology sphere $\Sigma(2,3,5)$ in Figure
\ref{fig:nicepoincare}. The figure uses the same conventions as in
Figure \ref{fig:domain}.

\begin{figure}[ht]
\psfrag{z}{$z$}
\psfrag{a1}{$\alpha_1$}
\psfrag{a2}{$\alpha_2$}
\psfrag{b1}{$\beta_1$}
\psfrag{b2}{$\beta_2$}
\begin{center}
\includegraphics[totalheight=0.9\textheight]{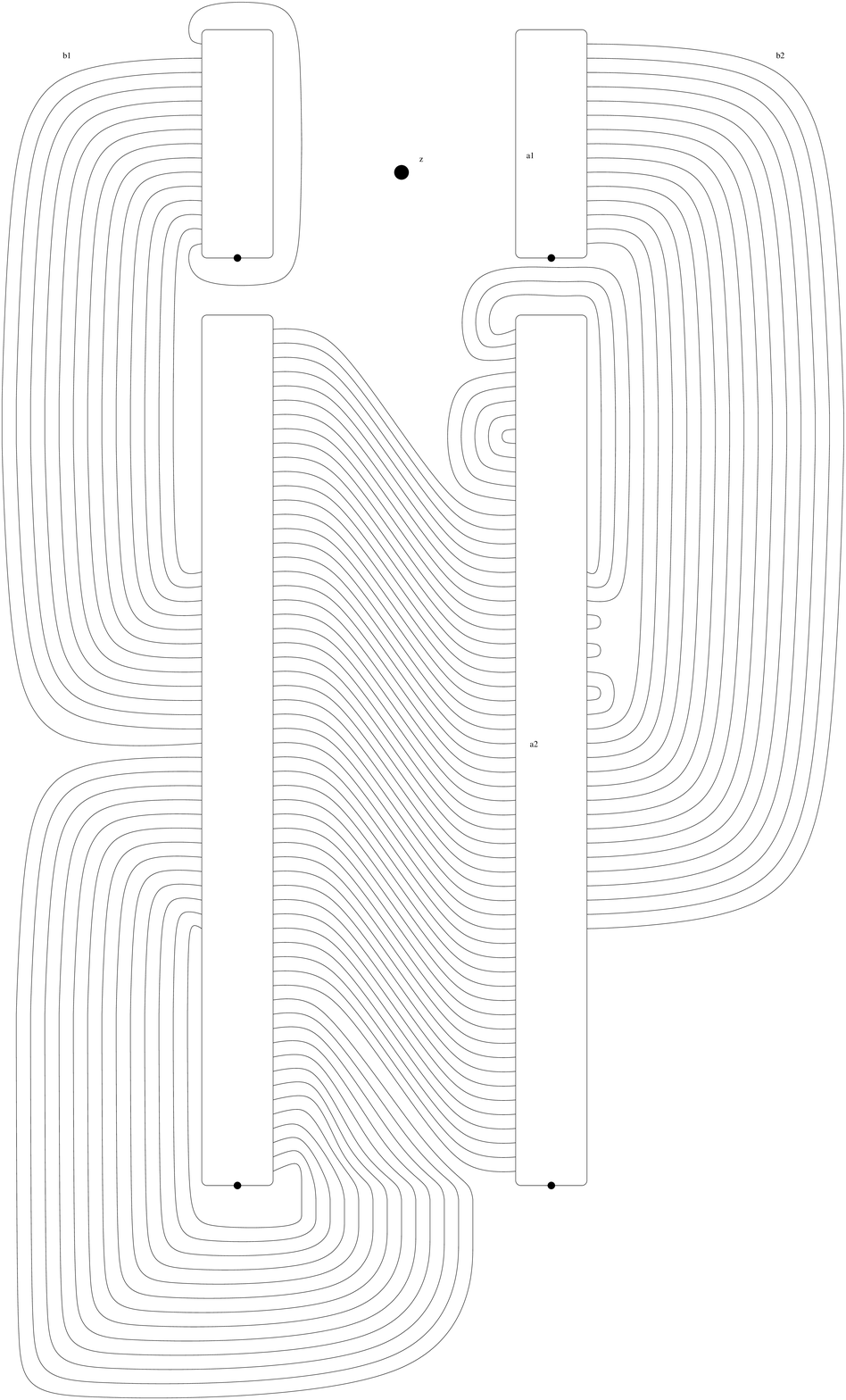}
\end{center}
\caption{A nice pointed diagram for the Poincar\'{e} homology
  sphere}\label{fig:nicepoincare}
\end{figure}

\chapter{The griddy algorithm}
In this chapter, we concentrate on the case of links inside $S^3$, and
indeed for the most part, we will be dealing with knots. Recall that
the two versions of knot Floer homology that we work with, are the hat
version and the minus version denoted by $\widehat{HFK}$ and
${HFK}^{-}$ respectively. They are bigraded modules over $\mathbb{Z}$
and $\mathbb{Z}[U]$ respectively, although, we will often ignore the
$U$ action on $HFK^{-}$ and treat them simply as bigraded abelian
groups. The two gradings $M$ and $A$ are the Maslov grading and the
Alexander grading respectively, and they both assume integer values
for knots in $S^3$. 

In \cite{CMPOSS}, based on a grid presentation of the knot, chain
complexes over $\mathbb{F}_2$ are constructed, whose homologies agree
with knot Floer homologies with coefficients in $\mathbb{F}_2$. A sign
refined version of the grid chain complexes was constructed by Ciprian
Manolescu, Peter Ozsv\'{a}th, Zolt\'{a}n Szab\'{o} and Dylan Thurston
in \cite{CMPOZSzDT}, where they also gave a combinatorial proof of the
invariance of the homology of the chain complex. In this chapter we try
to construct CW complexes corresponding to those grid chain complexes,
and mimic the proof of invariance from \cite{CMPOZSzDT} to show that
the stable homotopy type of these CW complexes is also a knot
invariant. We first review some of the basic definitions about posets.

\section{Partially ordered sets}\label{sec:posets}

A set $P$ with a binary  relation $\preceq$ is a partially ordered set
if  $a\preceq  b,b\preceq   c\Rightarrow  a\preceq  c$  and  $a\preceq
b,b\preceq a\Leftrightarrow  a=b$. If $a\preceq  b, a\neq b$,  then we
often say $a$ is less than  $b$ and write $a\prec b$.  If $\nexists z,
b\prec z$, we say $b$ is  a maximal element. Minimal elements are also
defined similarly. We also  often abbreviate partially ordered sets as
posets.

We say  $b$ covers $a$,  and write $a\leftarrow  b$ if $a\prec  b$ and
$\nexists z,  a\prec z\prec b$. Any  subset of a poset  has an induced
partial  order. A  subset  $C\subseteq P$  is  called a  chain if  the
induced  order is  a  total order.   Chains  themselves are  partially
ordered by  inclusion. Maximal chains  are the maximal  elements under
this order. Submaximal chains are  chains which are covered by maximal
chains under this  order. The length of a chain  is the cardinality of
the chain considered just as a set.

The 
Cartesian product of two posets $P$ and $Q$ is defined as the poset
$P\times Q$, whose elements are pairs $(p,q)$ with $p\in P$ and $q\in
Q$, and we declare $(p',q')\preceq (p,q)$ if $p'\preceq p$ in $P$ and
$q'\preceq q$ in $Q$.

The  order  complex  of  a   poset  is  a  simplicial  complex,  whose
$k$-simplexes  are chains  of length  $(k+1)$. The  boundary  maps are
defined naturally.

We define a closed interval $[a,b]=\{z\in P|a\preceq z\preceq b\}$. Open intervals,
or half-closed intervals are defined similarly. We 
also define $(-\infty,b]$ as  $\{z\in P|z\preceq b\}$ and $[a,\infty)$
as $\{z\in P|a\preceq z\}$.

A poset is said to be  graded if in every interval, all maximal chains
have the same length, in which  case the common length is known as the
length of  the interval. A graded poset  is said to be  thin, if every
submaximal chain is  covered by exactly $2$ maximal  chains.  A graded
poset is  subthin if  it is  not thin, and  every submaximal  chain is
covered by at most $2$ maximal chains.

A graded  poset is said to be  shellable if the maximal  chains have a
total    ordering     $\leq$,    such    that     $\mathfrak{m}_i    <
\mathfrak{m}_j\Rightarrow \exists \mathfrak{m}_k < \mathfrak{m}_j$ and
$\exists                    x\in\mathfrak{m}_j$                   with
$\mathfrak{m}_i\cap\mathfrak{m}_j\subseteq\mathfrak{m}_k
\cap\mathfrak{m}_j=\mathfrak{m}_j\setminus\{x\}$.

\begin{lem} Any  interval (closed,  half-closed, open) of  a shellable
poset is itself shellable.
\end{lem}

\begin{proof} We  just prove for the  case of an interval  of the form
$(a,b]$. The other cases follow  similarly. Take a maximal chain $c_1$
in $(-\infty, a]$, and take a maximal chain $c_2$ in
  $(b,\infty)$. Using the chosen maximal chains, the
maximal chains  in $(a,b]$ can be put in an one-one  correspondence with maximal
chains  of the  original poset  which start  with $c_1$  and  end with
$c_2$. But such maximal chains  have a total ordering induced from the
shellable structure, and it  is routine  to check  that such  an ordering
suffices.
\end{proof}

\begin{lem} \label{lem:shell1} Let $P$ be  a shellable poset with a unique
minimum $z$.  If we construct a  new poset $P'$ by  adjoining a single
element $z'$ which  covers nothing and is itself  covered by precisely
the elements that cover $z$, then $P'$ is shellable.
\end{lem}

\begin{proof} Note that the maximal chains in $[z',\infty)$ correspond
to maximal chains in $[z,\infty)$, and thus a shellable total ordering
of maximal chains in $[z,\infty)$  gives us a shellable total ordering
of maximal chains in $[z',\infty)$. We put a total ordering on maximal
chains in  $P'$ by declaring any  maximal chain in  $[z,\infty)$ to be
smaller than any  maximal chain in $[z',\infty)$. It  is again easy to
check that this ordering satisfies all the required properties.
\end{proof}

\begin{lem}  \label{lem:shell2} Let  $P$  be a  shellable  poset with  two
minimums $z$ and  $z'$ which are covered by the  same elements.  If we
construct a new poset $P'$ by  adjoining a single element $w$ which is
covered by $z$ and $z'$, then $P'$ is shellable.
\end{lem}

\begin{proof} Note  that maximal chains of $P'$  correspond to maximal
chains of  $P$. Thus a shellable  total ordering of  maximal chains in
$P$  induces a  total ordering  of maximal  chains in  $P'$,  which is
easily checked to be shellable.
\end{proof}

A  graded poset  is  said to  be  edge-lexicographically shellable  or
EL-shellable if there is a map  $f$ from the set of covering relations
(alternatively closed  intervals of length  $2$) to a  totally ordered
set,  such that  for any  interval $[x_1,x_n]$  of length  $n$,  if we
associate the  $(n-1)$-tuple $(f([x_1,x_2]),\ldots,f([x_{n-1},x_n]))$ to a
maximal  chain $x_1\leftarrow  x_2\cdots  \leftarrow x_{n-1}\leftarrow
x_n$, then there is a unique  maximal chain for which the $(n-1)$-tuple is
increasing,  and   under  lexicographic  ordering,   the  corresponding
$(n-1)$-tuple  is smaller than  any  $(n-1)$-tuple coming  from any  other
maximal chain between $x_1$ and $x_n$.

We shall mainly use the following theorems.

\begin{thm} \cite{AB} \label{thm:el}
EL-shellable $\Rightarrow$ every closed interval is shellable.
\end{thm}

\begin{proof}
Choose an interval $[x_1,x_n]$ with length $n$. There is a map from the
set of covering relations to a totally ordered set, and the
lexicographic ordering induces an ordering of the maximal chains. This
is almost a total order, except two different maximal chains might
have the same labeling. So for each $(n-1)$-tuple of elements from the
totally ordered set, look at all the maximal chains which have that
$(n-1)$-tuple as its label, and totally order them in any way. This gives
us a total ordering of all maximal chains in $[x_1,x_n]$.

Let $\mathfrak{m}_1$ and $\mathfrak{m}_2$ be two maximal chains with
$\mathfrak{m}_1 < \mathfrak{m}_2$. Each maximal chain is a sequence of
$n$ elements from the poset, starting at $x_1$ and ending at
$x_n$. Thus $\mathfrak{m}_1$ and $\mathfrak{m}_2$ agree up to some
$x_k$, and start being different, and then agree again at $x_l$ (and
maybe disagree again later). In other words, $\mathfrak{m}_1$ starts
as $x_1\leftarrow\cdots\leftarrow x_k\leftarrow
y_{k+1}\leftarrow\cdots\leftarrow 
y_{l-1}\leftarrow x_l\leftarrow\cdots$, and $\mathfrak{m}_2$ starts
as $x_1\leftarrow\cdots\leftarrow x_k\leftarrow z_{k+1}\leftarrow
\cdots\leftarrow z_{l-1}\leftarrow x_l\leftarrow\cdots$, and the set
$\{y_{k+1},\ldots,
y_{l-1}\}$ is disjoint from the set $\{z_{k+1},\ldots,z_{l-1}\}$. Look
at the interval $[x_k,x_l]$, and let $\mathfrak{n}_i=\mathfrak{m}_i
\cap[x_k,x_l]$. Since the interval $[x_k, x_l]$ has a unique maximal
chain whose labeling is increasing, which in addition happens to the
minimum one, the labeling in $\mathfrak{n}_2$ cannot be
increasing. Hence there is a first place $z_{t-1}\leftarrow
z_t\leftarrow z_{t+1}$, where the labeling decreases. However there
must be an increasing chain $z_{t-1}\leftarrow z_t^{\prime}\leftarrow
z_{t+1}$ in the interval $[z_{t-1},z_{t+1}]$. Thus if
$\mathfrak{m}_3=\mathfrak{m}_2\cup \{z_t^{\prime}\}\setminus\{z_t\}$,
then $\mathfrak{m}_3 < \mathfrak{m}_2$, and
$\mathfrak{m}_1\cap\mathfrak{m}_2 \subseteq
\mathfrak{m}_3\cap\mathfrak{m}_2=\mathfrak{m}_2\setminus\{z_t\}$. This
shows $[x_1,x_n]$ is shellable. 
\end{proof}

\begin{thm} \cite{GDVK} Finite,  shellable and thin  (resp. subthin) $\Rightarrow$
Order complex is PL-homeomorphic to a sphere (resp. ball).
\end{thm}

\begin{proof}
  Let $P$ be a finite, shellable poset which is also either thin or
  subthin. Choose some shellable total ordering on the the maximal
  chains, and under that ordering let the maximal chains be
  $\mf{m}_1<\mf{m}_2<\cdots< \mf{m}_k$. Let $n$ be the length of each
  maximal chain. The order complex of $P$ is the union of the order
  complexes of the maximal chains $\mf{m}_i$, each of which is an
  $(n-1)$-simplex $\Delta^{n-1}$. 

Let us construct the order complex of $P$ in the following
manner. Let $X_i$ be the order complex of the union of the elements in
$\mf{m}_1,\ldots,\mf{m}_i$. We glue to it the order complex of
$\mf{m}_{i+1}$ to get $X_{i+1}$.

We start with $X_1$ which is an $(n-1)$-simplex $\Delta^{n-1}$ (and hence
PL-homeomorphic to a ball). By induction each of the $X_i$'s (except
possibly $X_k$ when $P$ is thin) is PL-homeomorphic to an
$(n-1)$-dimensional ball. A careful consideration reveals that while
gluing $\Delta^{n-1}$, the order complex of $\mf{m}_{i+1}$, to $X_i$
(which by induction is an $(n-1)$-ball), thinness or subthinness along
with shellability implies that the gluing is done along a union
of $(n-2)$-simplices on the boundaries. The proof finishes after the (slightly
non-trivial) observation that the union of a non-empty collection of
$(n-2)$-simplices on $\del \Delta^{n-1}$ is either an $(n-1)$-ball or
an $(n-1)$-sphere.
\end{proof}

In case of a finite subthin  shellable poset, the boundary of the ball
corresponds to  those submaximal chains, which are  covered by exactly
$1$ maximal chain.

We will often encounter posets with the following properties. A sign
assignment is a map from the set of
covering relations to $\{\pm1\}$, such that every length $3$
closed interval has exactly two maximal chains and the product of the
signs for all the four covering
relations is $(-1)$. Two such sign assignments are said to be
equivalent if one can be obtained from another by a sequence of moves,
where at each move we choose an element of the poset and change the
signs of all the covering relations involving that element. A
grading assignment is a map $g$ from the elements of the
poset to $\mathbb{Z}$, such that whenever $a\leftarrow b$,
$g(b)=g(a)+1$. Having a grading assignment is
weaker than being graded, but is stronger than each closed interval being
graded. 

\begin{defn} 
A poset equipped with a sign assignment and a grading assignment,
whose every closed interval of the form $[a,b]$ is shellable, is
called a graded signed shellable poset, or in other words, a GSS
poset.
\end{defn}

For most of the time, we will be working with GSS posets. Given a GSS
poset, it is very easy to associate a chain complex to it. The
generators of the chain complex are the elements of the poset with
gradings determined by the grading assignment, and the
boundary map is given by

\begin{eqnarray*}
\partial x=\sum_{y,y\leftarrow x} s(y,x) y
\end{eqnarray*}

where $s(y,x)$ is sign assigned to the covering relation $y\leftarrow
x$. It is easy to see that this indeed is a chain complex, and the
chain homotopy type of the chain complex remains unchanged if the sign
assignment is replaced by an equivalent one. We call this complex to
be the chain complex associated to the GSS poset.

\section{Grid diagrams}\label{sec:griddiagrams}

In this section we will introduce three types of diagrams, grid
diagrams, commutation diagrams and stabilization diagrams. They are
all pictures on the standard torus, and we will associate certain
posets to each one of them. We often think of diagrams on the torus as
diagrams on the unit square in the plane. There are certain
transformations that we can work with. We can rotate the diagrams by
an angle of $\theta$, where $\theta\in\{\frac{\pi}{2},\pi,
\frac{3\pi}{2}\}$, and we call it the rotation $R(\theta)$. We can
reflect the whole diagram along a horizontal line or a vertical line,
and we call them the reflections $R(h)$ and $R(v)$ respectively. The
transformations $R(\frac{\pi}{2}), R(\frac{3\pi}{2}), R(h)$ and $R(v)$
keep the elements of the posets unchanged but reverse the partial
order. But if a poset is a GSS poset, it stays a GSS poset even after
reversing its partial order, so as far as being GSS is concerned, it
does not matter. 

A grid diagram with grid number $N$, is a picture on the standard
torus $T$. There are $N$ $\alpha$ (resp. $\beta$) circles, which are
pairwise disjoint and parallel to the meridian (resp. longitude) and
cut the torus into $N$ horizontal (resp. vertical) annuli. Clearly
$T\setminus (\alpha\cup\beta)$ has $N^2$ components. There are $2N$
markings on $T\setminus (\alpha\cup\beta)$, $N$ of them marked $X$,
$N$ of them marked $O$, such that each component contains at most one
marking, and each horizontal (resp. vertical) annulus contains one $X$
and one $O$. At this point, a careful reader will observe that
$(T,\alpha,\beta,X,O)$ is a genus one Heegaard diagram for a link
inside $S^3$, where the $X$-points are the $z$-basepoints, and the
$O$-points are the $w$-basepoints.  If $T$ is embedded in
$\mathbb{R}^3$ in the standard way, with the meridian bounding a disc
inside the torus, and the longitude bounding a disc outside, then the
link is obtained by joining $O$ to $X$ (resp.  $X$ to $O$) in the same
horizontal (resp. vertical) annulus, inside (resp. outside) the torus
$T$. Thus at every crossing, the vertical strands are the
overpasses. Furthermore note that a grid diagram when viewed as a
Heegaard diagram is a nice Heegaard diagram as defined in the previous
chapter. Therefore if a knot or a link is presented in a grid diagram,
every version of the link Floer homology can be computed using the
grid diagram.

In the other direction, given a link $L\subset \mathbb{R}^3$, it is
not difficult to get a grid diagram for $L$.

\begin{lem}  Given  a link  $L\subset\mathbb{R}^3$,  there  is a  grid
diagram that represents $L$.
\end{lem}

\begin{proof} Let $L$ be represented by a PL-link diagram in the
  $xy$-plane. That means that there are a bunch of vertices and a
  bunch of straight edges joining some of the vertices, such that each
  vertex has exactly two edges coming into it. By moving the vertices
  slightly, we can ensure no two vertices lie in the same horizontal
  line or the vertical line. We then replace each edge by a pair of
  horizontal and vertical edges, in one of two possible ways, as shown
  in Figure \ref{fig:verthor}. Thus $L$ is now represented by
  horizontal edges (with no two on the same horizontal line) and
  vertical edges (with no two on the same vertical line).

\begin{figure}[ht]
\begin{center} \includegraphics[width=80pt]{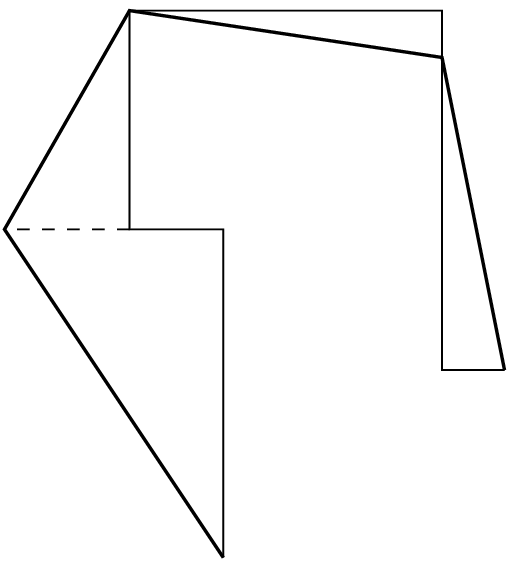}
\end{center}
\caption{Converting all edges to horizontal and vertical ones}\label{fig:verthor}
\end{figure}

  If in any crossing, the horizontal edge is the overpass, then we
  change the local picture as shown in Figure \ref{fig:verttohor} to
  ensure that the vertical edge is the overpass. Such a diagram then
  easily corresponds to a grid diagram.

\begin{figure}[ht]
\begin{center}
\includegraphics[width=200pt]{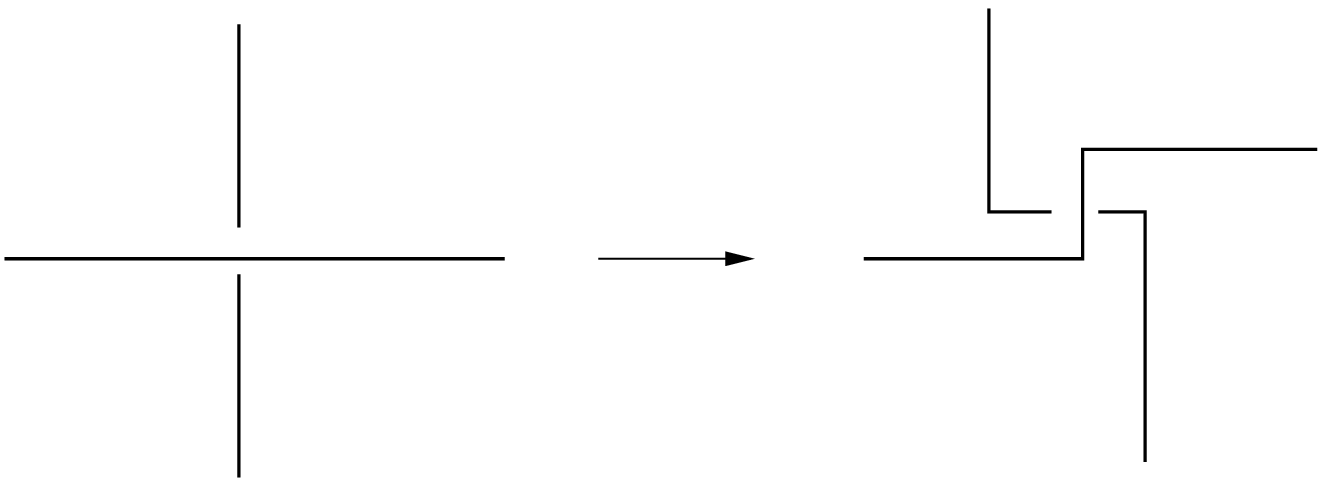}
\end{center}
\caption{Changing the horizontal overpasses to vertical ones}\label{fig:verttohor}
\end{figure}
\end{proof}

There are two processes on the grid diagram, namely commutation and
stabilization, which do not
change the isotopy class of the underlying link. We view markings in a
particular horizontal (resp. vertical) annulus as an embedded
$0$-sphere in one of the bounding $\alpha$ (resp. $\beta$) circles. 

In a horizontal (resp. vertical) commutation, we choose two adjacent
horizontal (resp. vertical) annuli, such that the markings in one of
them is unlinked with the markings in the other. Then we interchange
the markings for the two annuli. This process can also be viewed as
changing the $\alpha$ (resp. $\beta$) circle that lies between the two
adjacent horizontal (resp. vertical) annuli. Note that commutation
does not change the grid number, and it also keeps the isotopy class
of the link unchanged.

\begin{figure}[ht]
\psfrag{O}{O}
\psfrag{O1}{O}
\psfrag{O2}{O}
\psfrag{O3}{O}
\psfrag{O4}{O}
\psfrag{O5}{O}
\psfrag{X}{X}
\psfrag{X1}{X}
\psfrag{X2}{X}
\psfrag{X3}{X}
\psfrag{X4}{X}
\psfrag{X5}{X}
\psfrag{e}{$=$}
\begin{center} \includegraphics[width=300pt]{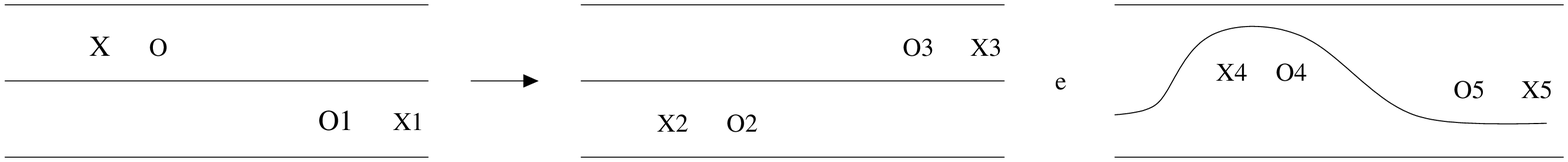}
\end{center}
\caption{Commutation}
\end{figure}

We can represent the process of commutation by a single grid like
diagram on the torus. 
Let $G$ and $G'$ be two grid diagrams drawn on the same torus $T$ with
grid number $N$, which differ from one another by a horizontal
commutation. (The case of a vertical commutation can be obtained from
the horizontal commutation by the rotation $R(\frac{\pi}{2})$). Thus $G'$ looks exactly
like $G$, except it has a circle $\alpha_c^{\prime}$ instead of
$\alpha_c$. We can represent the whole commutation by a single diagram $G_c$,
which is basically the grid diagram $G$ with an extra circle
$\alpha_c^{\prime}$. The circles $\alpha_c$ and $\alpha_c^{\prime}$
intersect in exactly two points, and we ensure that none of the
$\beta$ circles pass through either of those two points. Thus the
diagram has $(N^2+N+2)$ regions, of which $4$ are triangles, $4$ are
pentagons, and the rest are squares. There are two triangles and two
pentagons around each point of $\alpha_c\cap\alpha_c^{\prime}$, and we
can ensure that for each of those points, either the triangle to the right
or the triangle to the left has an $X$ marking. Of the two points of
intersection between $\alpha$ and $\alpha'$, let $\rho$ be the one
with $\alpha$ on its top-left. We call the pair $(G_c,\rho)$ a
commutation diagram. Note that due to presence of the point $\rho$,
the definition is not symmetric regarding the roles of $G$ and $G'$.

In a stabilization, we
choose a marking $X$, and change the vertical annulus
through the marking into two parallel vertical annuli by adding a $\beta$
circle, and change the horizontal annulus through the marking into two
parallel horizontal annuli by adding an $\alpha$ circle. The component
containing the original $X$ marking has now become $4$
components, and
we put two $X$ markings in two diagonally opposite
components, and put one $O$ marking in one of the other
two components. The original horizontal and vertical annuli through
our $X$ marking contained two $O$ markings,
and their position in the new diagram gets fixed by the condition that
each horizontal and each vertical annulus must contain exactly one $X$
and exactly one $O$ marking. Again note that stabilization keeps the
isotopy class of the link unchanged, but increases the grid number by
$1$. The roles of $X$ and $O$ seem asymmetric in this definition,
but the other type of stabilization, where the roles of $X$ and $O$
are reversed, can be obtained as a composition of stabilization of
this type and a few commutations.

\begin{figure}[ht]
\psfrag{O}{O}
\psfrag{O1}{O}
\psfrag{O2}{O}
\psfrag{O3}{O}
\psfrag{O4}{O}
\psfrag{X}{X}
\psfrag{X1}{X}
\psfrag{X2}{X}
\begin{center}
\includegraphics[width=200pt]{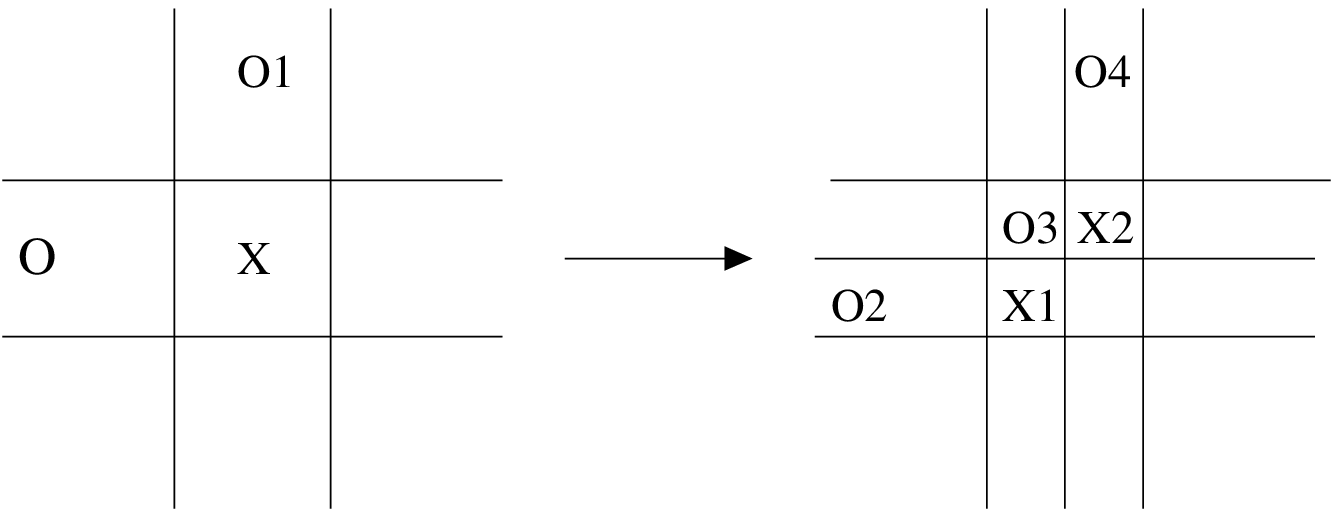}
\end{center}
\caption{Stabilization}
\end{figure}

Note that after stabilization, in the new grid diagram, a neighborhood
of the original $X$ marking looks like Figure \ref{fig:stablecases}.
The new $\alpha$ and $\beta$ circles are denoted by thick lines. The
cases $(c)$ and $(d)$ can be obtained from cases $(a)$ and $(b)$
respectively after the rotation $R(\pi)$. Thus we will only be
concentrating on the cases $(a)$ and $(b)$. (Indeed the case $(b)$ can
be obtained from the case $(a)$ by a rotation $R(\frac{\pi}{2})$, but
the reversal of the partial order presents some problems). We call the
new $\alpha$ circle and the new $\beta$ circle, $\alpha_s$ and
$\beta_s$, and call their intersection $\rho$. If the new grid diagram
is $G$, we call the pair $(G,\rho)$ a stabilization diagram. Thus a
stabilization diagram is basically just a grid diagram with a
distinguished $\alpha$ and a distinguished $\beta$ circle such that the
neighborhood of their intersection looks like Figure
\ref{fig:stablecases}.

\begin{figure}[ht]
\psfrag{O}{O}
\psfrag{O1}{O}
\psfrag{O2}{O}
\psfrag{O3}{O}
\psfrag{X}{X}
\psfrag{X1}{X}
\psfrag{X2}{X}
\psfrag{X3}{X}
\psfrag{X4}{X}
\psfrag{X5}{X}
\psfrag{X6}{X}
\psfrag{X7}{X}
\psfrag{a}{$(a)$}
\psfrag{b}{$(b)$}
\psfrag{c}{$(c)$}
\psfrag{d}{$(d)$}
\begin{center} \includegraphics[width=350pt]{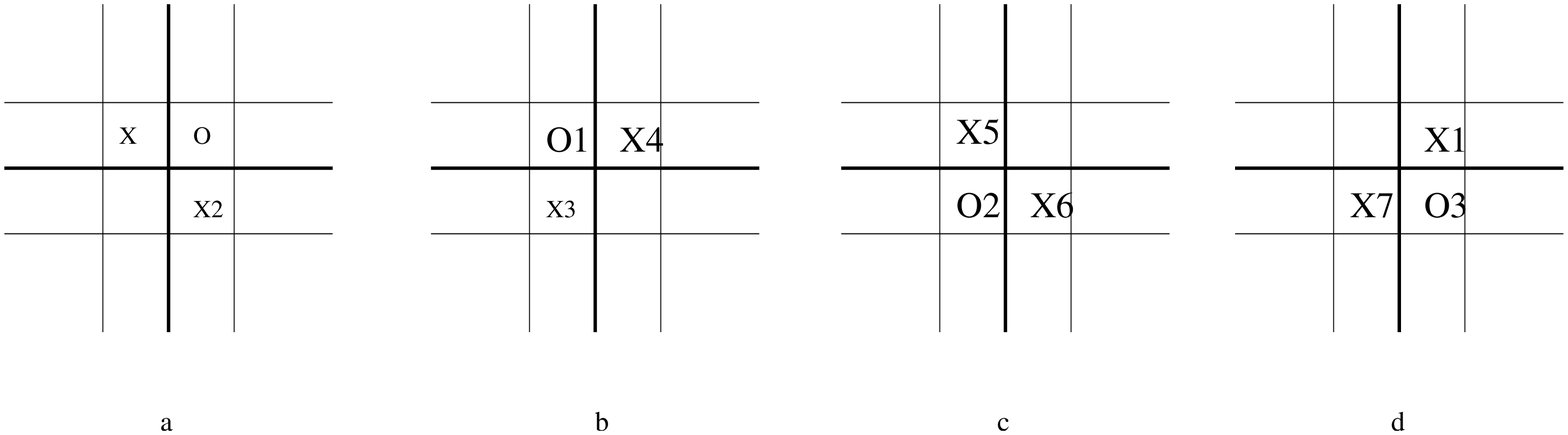}
\end{center}
\caption{Different types of stabilization}\label{fig:stablecases}
\end{figure}

\begin{thm}\cite{PC}
If two grid diagrams represent the same link, then we can apply
sequences of commutations and stabilizations on each of them, such
that the final two grid diagrams are the same.
\end{thm}

\subsection{Grid diagram}
Given a grid diagram with grid number $N$ representing a link $L$, we
can define two GSS posets $\widehat{\mathcal{G}}$ and
$\mathcal{G}^{-}$ such that the homology of the associated chain
complexes in the first case depends only on $L$ and $N$, and in the
second case depends only on $L$. The homologies are closely related to
the hat version and the minus version of the knot Floer homologies. To
help the reader and to keep this chapter mostly independent of the
first chapter, we will redefine all the relevant objects now.
The elements of the poset $\widehat{\mathcal{G}}$ are indexed by
formal sums $\widehat{x}=x_1+x_2+\cdots+x_N$ of $N$ points, such that
each $\alpha$ circle (resp.  each $\beta$ circle) contains one point.
The elements of $\mathcal{G}^{-}$ are indexed elements of the form
$x=\widehat{x} \prod_{i=0}^N U_i^{k_i}$ where
$\widehat{x}\in\widehat{\mathcal{G}}$ and $k_i\in\mathbb{N}\cup\{0\}$.
We need the following few definitions to understand the partial order
in the poset.

First number the $O$ (resp. $X$) markings as $O_1,O_2,\ldots,O_N$
(resp. $X_1,X_2,\ldots,X_N$). Let $\mathbb{O}$ (resp. $\mathbb{X}$) be
the formal sums $\sum_i O_i$ (resp. $\sum_i X_i$).  A domain $D$
connecting a generator $\widehat{x}$ to another generator
$\widehat{y}$, is a $2$-chain generated by components of
$T\setminus(\alpha\cup\beta)$ with $\partial(\partial
D_{|\alpha})=\widehat{y}-\widehat{x}$.  The set of all domains
connecting $\widehat{x}$ to $\widehat{y}$ is denoted by
$\mathcal{D}(\widehat{x},\widehat{y})$. For a point $p\in
T\setminus(\alpha\cup\beta)$ and a domain
$D\in\mathcal{D}(\widehat{x},\widehat{y})$, we define $n_p(D)$ to be
the coefficient of the $2$-chain $D$ at the point $p$. We define
$\mathcal{D}^0(\widehat{x},\widehat{y})$ (resp.
$\mathcal{D}^{0,0}(\widehat{x},\widehat{y})$) as a subset of
$\mathcal{D}(\widehat{x},\widehat{y})$ consisting of domains $D$ with
$n_p(D)=0$ whenever $p$ is any of the $N$ $X$ markings (resp. $2N$ $X$
or $O$ markings). For $x=\widehat{x}\prod_i U_i^{k_i}$ and
$y=\widehat{y} \prod_i U_i^{l_i}$ in $\mathcal{G}^{-}$, we define
$\mathcal{D}^0(x,y)$ as the subset of
$\mathcal{D}^0(\widehat{x},\widehat{y})$ consisting of all domains
with $n_{O_i}=l_i-k_i$. A domain $D$ is positive if $n_p(D)\geq
0\forall p$. For $v$ a point of intersection between an $\alpha$ curve
and a $\beta$ curve, and $D\in\mathcal{D}(\widehat{x},\widehat{y})$,
we define $n_v(D)$ as the average of the coefficients of $D$ in the
four components of $T\setminus(\alpha\cup\beta)$ around $v$.  Domains
in $\mathcal{D}(\widehat{x},\widehat{x})$ are said to be periodic
domains.

\begin{lem}  All  periodic  domains  are  generated  by  vertical  and
horizontal annuli.
\end{lem}

\begin{proof} Let $D$ be a periodic domain. Let $\partial D=\sum_i n_i
\alpha_i +\sum_j m_j \beta_j$.  Since any $\alpha_i$ (resp. $\beta_j$)
is homologous  to the meridian (resp. longitude),  this means $(\sum_i
n_i)\alpha_1+(\sum_j  m_j)\beta_1$  is  null-homologous in  the  torus
$T$. This implies $\sum_i n_i=\sum_j  m_j=0$. It is pretty easy to see
that we  can construct  a periodic domain  $D_v$ (resp. $D_h$)  out of
only vertical (resp. horizontal) annuli such that $\partial D_v=\sum_j
m_j  \beta_j$   (resp.  $\partial  D_h=\sum_i   n_i  \alpha_i$).  Thus
$D-D_v-D_h$ is a periodic domain  without boundary, and thus has to be
$kT$ for some $k$. We finish the proof by observing that the torus $T$
is also generated by vertical annuli.
\end{proof}

For two generators $\widehat{x}=\sum_i x_i$ and $\widehat{y}=\sum_i
y_i$, and a domain $D\in\mathcal{D}(\widehat{x},\widehat{y})$, the
Maslov index is defined to be $\mu (D)=\sum_i
(n_{x_i}(D)+n_{y_i}(D))$.  Notice that this is Lipshitz' formula for
the Maslov index suited to the case of grid diagrams. The relative
Maslov grading is defined to be $M(\widehat{x},\widehat{y})=\mu
(D)-2(\sum_i n_{O_i}(D))$. The relative Alexander grading is defined
to be $A(\widehat{x},\widehat{y})= \sum_i (n_{X_i}(D)-n_{O_i}(D))$.

 The following lemma shows that
the gradings are well defined.

\begin{lem} The relative gradings $A(\widehat{x},\widehat{y})$ and
  $M(\widehat{x},\widehat{y})$ are independent of the choice of domain
  $D\in\mathcal{D}(\widehat{x},\widehat{y})$.
\end{lem}

\begin{proof}  Any two domains  joining $\widehat{x}$  to $\widehat{y}$  are related  by a
periodic domain which is generated  by annuli. Adding any annulus to a
domain  increases  the Maslov  index  by  $2$,  increases  $\sum_i
n_{O_i}$ by $1$ and increases $\sum_i n_{X_i}$ by $1$, thus completing the proof.
\end{proof}

\begin{lem}\cite{RL}        For        generators
  $\widehat{x},\widehat{y},\widehat{z} \in\widehat{\mathcal{G}}$,
  $A(\widehat{x},\widehat{y})+A(\widehat{y},\widehat{z})=A(\widehat{x},\widehat{z})$
  and
  $M(\widehat{x},\widehat{y})+M(\widehat{y},\widehat{z})=M(\widehat{x},\widehat{z})$.
\end{lem}

\begin{proof}  The proof for the relative Alexander grading is
  trivial. We only present the slightly trickier case of the relative
  Maslov grading. The proof is immediate if we assume that $\mu$ is
  indeed an index, and hence is additive over Whitney disks. However a
  combinatorial setting deserves a combinatorial proof, and we give a
  proof without assuming that fact.

For a domain  $D\in\mathcal{D}(\widehat{x},\widehat{y})$ and  any $2$-chain
$D'$,  we  have $n_{\widehat{x}}(D')=n_{\widehat{y}}(D')+(\partial  D_{|\alpha})\cdot(\partial
D^\prime_{|\beta})$.    Here  the   dot  product   is   defined  after
translating the  $\alpha$ arcs in  four possible directions,  and then
taking the average of the four dot products, as shown in Figure \ref{fig:dotproduct}

\begin{figure}[ht]
\psfrag{a}{$\partial D$ is denoted by thick solid lines}
\psfrag{b}{The four translates of $\partial D_{|\alpha}$ are shown}
\psfrag{c}{$\partial D'$ is denoted by thick dotted lines}
\psfrag{d}{$(\partial  D_{|\alpha})\cdot(\partial
D^\prime_{|\beta})=\frac{5}{4}$}
\begin{center} \includegraphics[width=180pt]{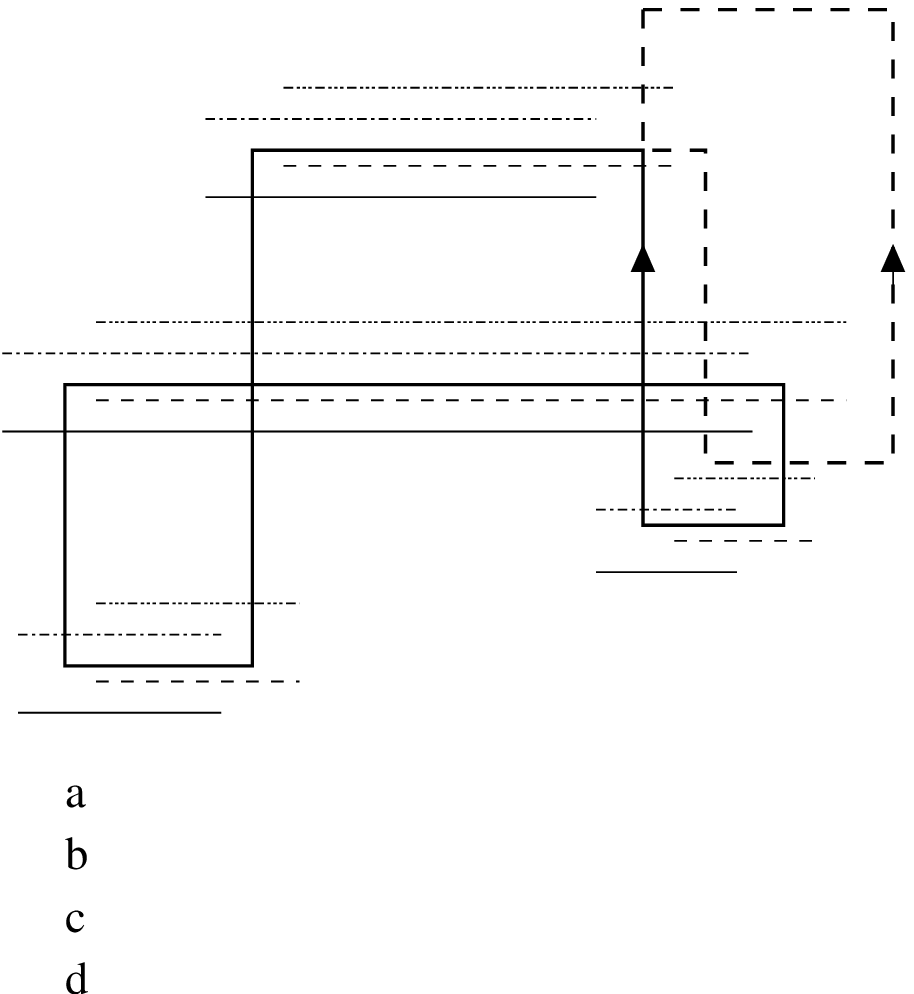}
\end{center}
\caption{Defining dot product of arcs}\label{fig:dotproduct}
\end{figure}

Now  take  $D_1\in\mathcal{D}(\widehat{x},\widehat{y})$  and $D_2\in\mathcal{D}(\widehat{y},\widehat{z})$.  We
have      to      show      $n_{\widehat{y}}(D_1+D_2)=n_{\widehat{x}}(D_2)+n_{\widehat{z}}(D_1)$.      But
$n_{\widehat{x}}(D_2)=n_{\widehat{y}}(D_2)+(\partial D_{1|\alpha})\cdot(\partial D_{2|\beta})$
and      $n_{\widehat{z}}(D_2)=n_{\widehat{y}}(D_2)-(\partial      D_{2|\alpha})\cdot(\partial
D_{1|\beta})$. Note $(\partial  D_1)\cdot(\partial D_2)=0$, and expand
to finish the proof.
\end{proof}

Indeed there is a different way to see this. For
$\widehat{x}\in\widehat{\mathcal{G}}$, we can define absolute Maslov
grading $M(\widehat{x})$ and absolute Alexander grading
$A(\widehat{x})$ such that
$M(\widehat{x},\widehat{y})=M(\widehat{x})-M(\widehat{y})$ and
$A(\widehat{x},\widehat{y})=A(\widehat{x})-A(\widehat{y})$.

We choose an $\alpha$ circle and a $\beta$ circle on the grid diagram
$G$ and cut open the torus $T$ along those circles to obtain a diagram
in $[0,N)\times[0,N)\subset \mathbb{R}^2$. In this planar diagram, 
the $\alpha$ circles become the lines $y=i$ and the $\beta$ circles
become the lines $x=i$ for $0\le i<N$. Let the $X$ marking and $O$
markings occupy half-integral lattice points. Now for two points
$a=(a_1, a_2)$ and $b=(b_1,b_2)$ in $\mathbb{R}^2$, we define
$J(a,b)=\frac{1}{2}$ if $(a_1-b_1)(a_2-b_2)>0$ and $0$ otherwise. We
extend $J$ bilinearly for formal sums and differences of points. For
$\widehat{x}\in\widehat{\mathcal{G}}$, we define
$M(\widehat{x})=J(\widehat{x}-\mathbb{O},\widehat{x}-\mathbb{O})+1$ and
$A(\widehat{x})=J(\widehat{x}-\frac{\mathbb{X}+\mathbb{O}}{2},\mathbb{X}-\mathbb{O})-\frac{N-1}{2}$. The
following is mere verification.

\begin{lem}\cite{CMPOZSzDT}
$A(\widehat{x})$ and $M(\widehat{x})$ are independent of choice of $\alpha$ and $\beta$
  circles along which the torus is cut open. $M(\widehat{x})$
  always takes integral values and $A(\widehat{x})$ takes integral values for a
  knot. Furthermore $M(\widehat{x},\widehat{y})=M(\widehat{x})-M(\widehat{y})$ and $A(\widehat{x},\widehat{y})=A(\widehat{x})-A(\widehat{y})$.
\end{lem}

We extend the assignment of Maslov and Alexander gradings from
$\widehat{\mathcal{G}}$ to $\mathcal{G}^{-}$. We define $M(\widehat{x}
\prod_i U_i^{k_i})=M(\widehat{x})-2\sum_i k_i$ and $A(\widehat{x} \prod_i
U_i^{k_i})=A(\widehat{x})-\sum_i k_i$. (In other words, we assign an $(M,A)$
bigrading of $(-2,-1)$ to each $U_i$). We define $\widehat{\mathcal{G}_m}$
(resp. $\mathcal{G}^{-}_m$) to be the the subset of
$\widehat{\mathcal{G}}$ (resp. $\mathcal{G}^{-}$) which has Alexander
grading $m$. Note that even though $\mathcal{G}^{-}$ is an infinite
set, for each $m$, $\widehat{\mathcal{G}_m}$ and $\mathcal{G}^{-}_m$ are
finite sets. In either case, we define $M_c=M+c$, and call it the
Maslov grading shifted by $c$.

If the reader is following the analogies from the Floer homology
picture, it should be pretty clear by this point that positive domains
of index one are of special importance to us. The following theorem
characterizes them. The theorem is in fact a consequence of the
results from the previous chapter, but we reprove it in these settings
so as to not disrupt the flow of the text.

\begin{lem}\cite{CMPOSS}  Let  $D\in\mathcal{D}(\widehat{x},\widehat{y})$  be  a positive  domain  with
$\mu(D)=1$.     Then   $\widehat{x}$    and    $\widehat{y}$   differ    in   exactly    two
coordinates. Furthermore, $D$ has coefficients $0$ and $1$ everywhere,
and the closure  of the regions where $D$ has  coefficients $1$ form a
rectangle   which  does   not  contain   any  $x$-coordinate   or  any
$y$-coordinate in its interior.
\end{lem}

\begin{proof} The domain $D$ cannot be copies of the torus, since each
copy of  the torus has  index $2N$. Thus  $D$ must have  boundary, and
without  loss of  generality, let  $\partial  D$ be  non-zero on  some
$\alpha$ circle, say  $\alpha_1$. It is easy to  see that $\partial D$
then  also  must  be  non-zero  on some  other  $\alpha$  circle,  say
$\alpha_2$.  Let $x_i$ and  $y_i$ be  the $x$  and $y$  coordinates on
$\alpha_i$.  Thus  $n_p(D)\ne  0$ for  $p\in\{x_1,x_2,y_1,y_2\}$,  and
since  each   is  at  least   $\frac{1}{4}$,  they  are   all  exactly
$\frac{1}{4}$.  Thus $\partial  D$ must  look like  the boundary  of a
rectangle, and $D$ itself must  be a rectangle. Furthermore it is also
clear   that  $D$   can  not   contain  any   $x$-coordinate   or  any
$y$-coordinate in its interior.
\end{proof}

We call positive  index one domains in $\mathcal{D}(\widehat{x},\widehat{y})$  to be empty
rectangles   and  denote  them   by  $\mathcal{R}(\widehat{x},\widehat{y})$.    Note  that
$\mathcal{R}(\widehat{x},\widehat{y})=\varnothing$  unless $\widehat{x}$ and  $\widehat{y}$ differ  in exactly
two  coordinates,  and  even  then $\#|\mathcal{R}(\widehat{x},\widehat{y})|\leq  2$.   We
define     $\mathcal{R}^0(\widehat{x},\widehat{y})=\mathcal{R}(\widehat{x},\widehat{y})\cap\mathcal{D}^0(\widehat{x},\widehat{y})$
and
$\mathcal{R}^{0,0}(\widehat{x},\widehat{y})=\mathcal{R}(\widehat{x},\widehat{y})\cap\mathcal{D}^{0,0}(\widehat{x},\widehat{y})$.
For $x=\widehat{x}\prod_i U_i^{k_i}$ and $y=\widehat{y}\prod_i
U_i^{l_i}$ in $\mathcal{G}^{-}$, we define $\mathcal{R}^0(x,y)=\mathcal{R}(\widehat{x},\widehat{y})\cap\mathcal{D}^0(x,y)$.
The following characterizes positive index $k$ domains.

\begin{lem}\label{lem:grading}  Let $D\in\mathcal{D}(\widehat{x},\widehat{y})$  be  a positive
domain.  Then    there    exists    generators
$\widehat{u_0},\widehat{u_1},\cdots,\widehat{u_k}\in\widehat{\mathcal{G}}$  with   $\widehat{u_0}=\widehat{x}$  and  $\widehat{u_k}=\widehat{y}$,  and
domains $D_i\in\mathcal{R}(\widehat{u_{i-1}},\widehat{u_i})$ such that $D=\sum_i D_i$.
\end{lem}

\begin{proof} Since $D$ is not a trivial domain, assume $n_{x_1}(D)\ne
0$.   Furthermore  since   $\partial   (\partial  D_{|\alpha})=\widehat{y}-\widehat{x}$,   the
coefficient of $D$  at either the top-right square  or the bottom-left
square of $x_1$  must be non-zero. Assume after a rotation $R(\pi)$ if
necessary,
it  is the top-right one. Now if $D$ contains the width one horizontal or
vertical annulus through this top-right square, then let $x_2$ be the
$x$-coordinate at the other boundary of the annulus. Then $D$ contains
the rectangle $r$ with $x_1$ and $x_2$ as the bottom-left and
top-right corners, and $D\setminus r$ has index $1$ less, and we are
done.

So now assume $D$ does not contain any such annulus. Consider all  $p$, points  of intersection
between $\alpha$  and $\beta$  circles, such that  $p\ne x_1$  and the
rectangle  with  $x_1$  as  the  bottom-left corner  and  $p$  as  the
top-right  corner is  contained  in $D$.  The  set of  such points  is
non-empty  by  assumption. Put  a  partial  order  on such  points  by
declaring a point $p$  to be smaller than or equal to  a point $q$, if
the  rectangle corresponding  to  $q$  contains $p$.  Let  $p_0$ be  a
maximal element under this order. Such a maximal element exists since
$D$ does not contain any of the above described annuli.

Now consider the rectangle $r$ with $x_1$ and $p_0$ as the bottom-left
and the top-right corners respectively. We first want to show that $r$
must  contain  an $x$-coordinate  other  than  $x_1$.  Assume $D$  has
non-zero coefficient  at the  square to the  top-left of  $p_0$. Since
$p_0$ is  a maximal  element, $D$ must  have zero coefficient  at some
square above the top horizontal line  of $r$. So we start at $p_0$ and
proceed left along this horizontal line until we reach the first point
$p_1$, such  $D$ has  non-zero coefficient at  the top-right  square of
$p_1$, but has zero coefficient  at the top-left square of $p_1$. Then
it is easy to see that  $p_1$ must be an $x$-coordinate. Similarly, if
$D$ has non-zero coefficient at the bottom-right square of $p_0$, then
also $r$ contains  an $x$-coordinate other than $x_1$.  Finally if the
coefficient of $D$  is zero at both the  top-left and the bottom-right
square of $p_0$, then $p_0$ itself is an $x$-coordinate.

Thus $D$ contains a rectangle, with two $x$-coordinates, say $x_1$ and
$x_2$  being   the  bottom-left   corner  and  the   top-right  corner
respectively.  Now consider  the partial  order on  points  other than
$x_1$,  that   we  defined  earlier,  but  restrict only to  the
$x$-coordinates.  Again  the poset  is  non-empty,  since it  contains
$x_2$. Take a minimal element, say $x_3$. Then the rectangle $r'$ with
$x_1$  and  $x_3$ being  the  bottom-left  and  the top-right  corners
respectively, is an index $1$  domain connecting $\widehat{x}$ to some generator
$\widehat{u_1}$. The  positive domain  $D\setminus r'$ has  index $1$  less
(alternatively has a smaller sum of coefficients as $2$-chains), and
hence an induction finishes the proof.
\end{proof}

From now on, until the rest of the section, we only consider the case
for knots.
There       is       a       combinatorial       sign       assignment
$s:\{(\widehat{x},\widehat{y},D)|\widehat{x},\widehat{y}\in\widehat{\mathcal{G}},
D\in\mathcal{R}(\widehat{x},\widehat{y})\}\rightarrow\{-1,1\}$,
satisfying the following properties. If $D_1+D_2$ is a horizontal
(resp. vertical) annulus and all is well-defined, then
$s(\widehat{x},\widehat{y},D_1)  s(\widehat{y},\widehat{x},D_2)$ is $1$
(resp. $-1$). Otherwise, if $D_1+D_2=D_3+D_4$, $\widehat{y}\ne\widehat{w}$ and all is 
well-defined,
$s(\widehat{x},\widehat{y},D_1)s(\widehat{y},\widehat{z},D_2)=
-s(\widehat{x},\widehat{w},D_3)s(\widehat{w},\widehat{z},D_4)$.  

Two such sign assignments  are said to
be equivalent  if one  can be obtained  from another by a sequence of
moves, such that at each move we fix a generator $\widehat{x}$ and we switch the
sign of  every triple of the form $(\widehat{x},\widehat{y},D)$ and $(\widehat{y},\widehat{x},D)$. 

The partial order in $\widehat{\mathcal{G}}$ (resp. $\mathcal{G}^{-}$)
is defined as $\widehat{y}\preceq \widehat{x}$ (resp. $y\preceq x$) if there exists a positive domain in
$\mathcal{D}^{0,0}(\widehat{x},\widehat{y})$
(resp. $\mathcal{D}^0(x,y)$). It is clear in both cases that the elements
in different 
Alexander gradings are not comparable. Also the covering relations are
indexed by elements of $\mathcal{R}^{0,0}(\widehat{x}, \widehat{y})$ and $\mathcal{R}^0(x,y)$. It is
routine to prove the following.

\begin{lem}\cite{CMPOZSzDT}
For knots, with sign assignment as defined above, and the grading assignment
being the Maslov grading, for each $m$, $\widehat{\mathcal{G}_m}$ and
$\mathcal{G}^{-}_m$ are well-defined, finite, graded and signed posets.
\end{lem}

In Section \ref{sec:gss}, we will see that the closed intervals in each of these posets are
also shellable, and hence they will be GSS posets. However just being
graded and signed is enough for us to associate a chain complex to
each of them. Let $\mathcal{C}^{-}$ and $\widehat{\mathcal{C}}$ be the
associated chain complexes. Their homology is bigraded, with the
Maslov grading being the homological grading, and the Alexander
grading being an extra grading.

\begin{thm} \cite{CMPOZSzDT}There is a bigraded abelian group $HFK^{-}(L)$ which
  depends only on the knot $L$, which is isomorphic (as bigraded abelian
  groups) to the homology of $\mathcal{C}^{-}$.
\end{thm}

\begin{thm} \cite{CMPOZSzDT}There is
a  bigraded  abelian group  $\widehat{HFK}(L)$  which depends
only  on  the knot $L$,  such  that  the  homology of  $\widehat{\mathcal{C}}$ is
isomorphic    (as    bigraded    abelian    groups)    to
$\widehat{HFK}(L)\otimes^{N-1} \mathbb{Z}^2$, where the $(M,A)$
bigrading of the two generators in $\mathbb{Z}^2$ are $(0,0)$ and $(-1,-1)$.
\end{thm}

If everything is computed with coefficients in $\F_2$, then these
groups have to the hat version and the minus version of the knot Floer
homology respectively.  However with coefficients in $\Z$, the groups
$\widehat{HFK}(L)$ and $HFK^{-}(L)$ do not have to be the hat and the
minus version of the link Floer homology. This is because there could
be a different sign convention on the grid poset whose homology is the
knot Floer homology. (The sign convention is unique only after
assuming that the product of signs corresponding to each width one
vertical annulus is the same).

The following is a crucial piece of observation.

\begin{lem}   If   the   grid   diagram  represents   a   knot,   then
$\mathcal{D}^{0,0}(\widehat{x},\widehat{x})$  consists
  of  only  the   trivial  domain.   In
particular,      for      any      pair
$\widehat{x},\widehat{y}\in\widehat{\mathcal{G}}$ (resp. $x,y\in\mathcal{G}^{-}$),
$\#|\mathcal{D}^{0,0}(\widehat{x},\widehat{y})|\leq 1$
(resp. $\#|\mathcal{D}^0(x,y)|\leq 1)$.
\end{lem}

\begin{proof}  Number  the  $O$  points  (modulo $N$)  such  that  the
horizontal  annulus through  $O_i$  and the  vertical annulus  through
$O_{i+1}$ intersect in an $X$ point. Since the grid diagram represents
a knot, such a numbering can be done.

Now let $A_i$ (resp. $B_i$) be the horizontal (resp. vertical) annulus
through  $O_i$. Let  $D\in\mathcal{D}^0(\widehat{x},\widehat{x})$ with  $D=\sum_i  n_i A_i
+\sum_j  m_j   B_j$.  Since  $n_{O_i}(D)=n_{O_{i+1}}(D)=0$,   we  have
$m_i=-n_i=m_{i+1}$. This  implies all the  $m_i$'s are equal,  and all
the $n_j$'s are equal, and they  are opposite of one another. Thus $D$
is the trivial domain.
\end{proof}

\subsection{Commutation diagram}
Many of the above results are true if we work with a commutation
diagram instead of a grid diagram. We define new posets
$\widehat{\mathcal{G}_c}$ and $\mathcal{G}_c^{-}$ corresponding
to the commutation. If
$\widehat{\mathcal{G}}$ and $\widehat{\mathcal{G}^{\prime}}$ are the
generators of $G$ and $G'$, then
$\widehat{\mathcal{G}_c}
=\widehat{\mathcal{G}}\cup\widehat{\mathcal{G}^{\prime}}$ and
$\mathcal{G}_c^{-} = \mathcal{G}^{-}\cup (\mathcal{G}^{\prime})^{-}$.
For $\widehat{x},\widehat{y}\in\widehat{\mathcal{G}_c}$, let $x_c$ and $y_c$
be the coordinates of $\widehat{x}$ and $\widehat{y}$ on $\alpha_c$ or
$\alpha_c^{\prime}$. If both $\widehat{x}$ and $\widehat{y}$ are in $\widehat{\mathcal{G}}$
(resp. $\widehat{\mathcal{G}^{\prime}}$) a domain joining $\widehat{x}$
to $\widehat{y}$ is a $2$-chain $D$ generated by components of
$T\setminus(\alpha\cup\alpha^{\prime}\cup\beta)$,  such that
$\partial(\partial D_{|(\alpha\cup\alpha')})=\widehat{y}-\widehat{x}$ and $\partial
D_{|\alpha_c^{\prime}}=0$ (resp. $\partial D_{|\alpha_c}=0$). For
$\widehat{x}\in\widehat{\mathcal{G}}$ and
$\widehat{y}\in\widehat{\mathcal{G}^{\prime}}$, a domain joining
joining $\widehat{x}$
to $\widehat{y}$ is a $2$-chain $D$ with $\partial(\partial
D_{|(\alpha\cup\alpha')})=\widehat{y}-\widehat{x}$ and $\partial(\partial
D_{|\alpha_c})=\rho-x_i$ and $\partial(\partial
D_{|\alpha_c^{\prime}})=y_i-\rho$. (We are not interested in domains
that join points in $\widehat{\mathcal{G}^{\prime}}$ to points in
$\widehat{\mathcal{G}}$). The set of all such domains is
denoted by $\mathcal{D}(\widehat{x},\widehat{y})$, and $\mathcal{D}^0(\widehat{x},\widehat{y})$
(resp. $\mathcal{D}^{0,0}(\widehat{x},\widehat{y})$) is the subset
which has coefficients $0$ at every $X$ marking (resp. every $X$ or
$O$ marking). For $x=\widehat{x}\prod_i U_i^{k_i}$ and
$y=\widehat{y}\prod_i U_i^{l_i}$ in $\mathcal{G}_c^{-}$, we define
$\mathcal{D}^0(x,y)$ as the subset of
$\mathcal{D}^0(\widehat{x},\widehat{y})$ with $n_{O_i}=l_i-k_i$. We call a domain to be
positive if it has non-negative coefficients everywhere. A $2$-chain
$D$ is said to be periodic if $\partial D$ is a collection of whole
copies of $\alpha$ and $\beta$ circles. Note that this is different
from $\mathcal{D}(\widehat{x},\widehat{x})$.

The Alexander gradings of points in $\widehat{\mathcal{G}_c}$ are the
ones induced from the Alexander gradings in $\widehat{\mathcal{G}}$
and $\widehat{\mathcal{G}^{\prime}}$. The Maslov grading for points in
$\widehat{\mathcal{G}_c}$ is defined using the the Maslov grading
induced from $\widehat{\mathcal{G}}$ and the Maslov grading induced
from $\widehat{\mathcal{G}^{\prime}}$ shifted by $-1$. The Maslov
grading shifted by $c$, $M_c$ is defined similarly as $M+c$. The
$(M,A)$ bigrading of each $U_i$ is still $(-2,-1)$. Given a domain
$D\in\mathcal{D}(\widehat{x},\widehat{y})$, we define the Maslov index
$\mu(D)=M(\widehat{x})-M(\widehat{y})+2\sum_i n_{O_i}(D)$. Note that
this is different from the standard way of defining Maslov index. We
will soon encounter objects called empty pentagons, and according to
our definition they have index $1$, but according to the standard
definition they have index $0$. There is actually an alternative way
to define our version of the Maslov index, analogous to the case for
grid diagrams, as follows. For $\widehat{x},\widehat{y}$ both in
$\widehat{\mathcal{G}}$ or $\widehat{\mathcal{G}^{\prime}}$, we define
$\mu(D)=n_{\widehat{x}}(D)+n_{\widehat{y}}(D)$. For
$\widehat{x}\in\widehat{\mathcal{G}}$ and
$\widehat{y}\in\widehat{\mathcal{G}^{\prime}}$, we define
$\mu(D)=\frac{1}{4}
+n_{\widehat{x}}(D)+n_{\widehat{y}}(D)-(\partial{D}_{|\alpha})\cdot
(\partial{D}_{|\alpha^{\prime}})$. However we will stick to our first
definition for the time being and leave the proof of equivalence of
the two definitions to the interested reader.

The partial orders are defined similarly. In
$\widehat{\mathcal{G}_c}$ (resp. $(\mathcal{G}^{\prime})^{-}$), we
define $\widehat{y}\preceq \widehat{x}$ (resp. $y\preceq x$) if there is a
positive domain in $\mathcal{D}^{0,0}(\widehat{x},\widehat{y})$
(resp. $\mathcal{D}^0(x,y)$). There exists a sign
assignment for covering relations with properties analogous to the
case for the grid diagrams. We define $\widehat{(\mathcal{G}_c)_m}$
(resp. $(\mathcal{G}_c^{-})_m$) to be the subset of $\widehat{\mathcal{G}_c}$
(resp. $\mathcal{G}_c^{-}$) with Alexander grading $m$.

The following is a list of lemmas, analogous to the case for grid
diagrams. Most of the following are mere verifications. We provide
details of the proofs for some of the trickier cases.

\begin{lem}
Periodic domains are generated by annuli. For horizontal annuli, we
consider both the annuli coming from $G$ and the annuli coming from
$G'$. Thus periodic domains are generated by annuli in $G$ and the
special domain $D_c$ as shown in Figure \ref{fig:specialdomain}.
\end{lem} 

\begin{figure}[ht]
\psfrag{a}{$\alpha_c$}
\psfrag{b}{$\alpha_c^{\prime}$}
\psfrag{r}{$\rho$}
\psfrag{on}{$1$}
\psfrag{mon}{$-1$}
\psfrag{z1}{$0$}
\psfrag{z2}{$0$}
\psfrag{z3}{$0$}
\psfrag{z4}{$0$}
\psfrag{z5}{$0$}
\begin{center} \includegraphics[width=200pt]{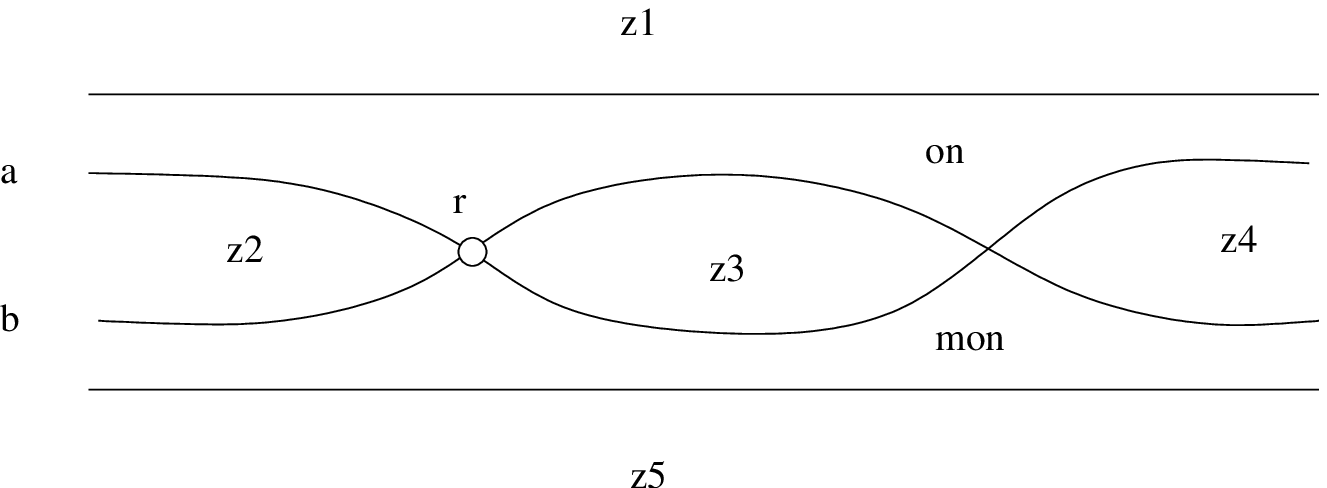}
\end{center}
\caption{Coefficients of the special domain}\label{fig:specialdomain}
\end{figure}

\begin{lem}\label{lem:rho1}
For any positive domain $D$ in $\mathcal{D}^0(\widehat{x},\widehat{y})$, each of the
  coefficients of $D$ in the four regions around $\rho$ is at most $1$.
\end{lem}

\begin{proof}
  Recall one of the $4$ regions around $\rho$ is an $X$ marking, and
  hence the coefficient of $D$ at that region is $0$. After the
  rotation $R(\pi)$ if necessary, we can assume that region is to the
  right of $\rho$. If $D$ is a domain in either $G$ or $G'$, then it
  is easy to see that $n_{\rho}(D)$ is either $0$ or $\frac{1}{2}$. So
  let us assume $\widehat{x}\in\widehat{\mathcal{G}}$ and
  $\widehat{y}\in\widehat{\mathcal{G}^{\prime}}$. If $x_c$ (resp.
  $y_c$) is the coordinate of $\widehat{x}$ (resp $\widehat{y}$) on
  $\alpha_c$ (resp. $\alpha_c'$), then $\partial(\partial
  D_{|\alpha_c})=\rho-x_c$ and $\partial(\partial
  D_{|\alpha^{\prime}_c})=y_c-\rho$. Thus there is a path which goes
  from $x_c$ to $\rho$ along $\alpha_c$ and then from $\rho$ to $y_c$
  along $\alpha_c'$ which coincides with $\partial
  D_{|(\alpha_c\cup\alpha_c')}$ as $1$-chains. Furthermore we can also
  ensure that the path does not enter $\rho$ through top-left and then
  leave through bottom-left. The way we construct this path is by
  starting at $x_c$ and then proceeding so as to keep the above
  conditions satisfied. It is easy to check that any such attempt
  always leads to a path with the required properties. We can also
  easily ensure that we never have to make an $180^\circ$ turn along
  our path. Now we will prove that such a path hits $\rho$ exactly
  once. Note that will be enough to prove the lemma.

Assume if possible the curve hits $\rho$ at least twice. Then look at
the part of the path between the first hit and the second hit. This
part has to one copy of either $\alpha_c$ or $\alpha_c'$, and neither
is allowed since both have some $X$ marking immediately on their left.
\end{proof}

\begin{lem}
For a knot, a periodic domain $D$ with $n_{X_i}(D)=n_{O_i}(D)=0$ for all $i$,
is generated by the special domain $D_c$. That implies that, given
$x,y\in\mathcal{G}_c^{-}$ there can
be at most $2$ positive domains in $\mathcal{D}^0(x,y)$.
\end{lem}

\begin{lem}\cite{CMPOZSzDT}
$A(x)$ and $M(x)$ are well-defined and they take integral values for a knot.
\end{lem}

\begin{lem}
Let $D\in\mathcal{D}(\widehat{x},\widehat{y})$ be a positive domain with $\mu(D)=1$. Then
either $D$ is an empty rectangle in $G$ or $G'$, or
$\widehat{x}\in\widehat{\mathcal{G}}$ and $\widehat{y}\in\widehat{\mathcal{G}^{\prime}}$,
and they differ in exactly two coordinates. Furthermore, $D$ has
coefficients $0$ or $1$ everywhere, and the closure of the regions
where $D$ has coefficient $1$ forms a pentagon which does not contain
any $x$-coordinate or any $y$-coordinate in its interior. 
\end{lem}

\begin{proof}
The proof is actually a direct corollary of Lemma \ref{lem:classification}, the proof of
which does not in any way require this theorem. 
\end{proof}

Such positive index $1$ domains are called empty rectangles or empty
pentagons depending on their shape, and their sets are denoted by
$\mathcal{R}(\widehat{x},\widehat{y})$ and
$\mathcal{P}(\widehat{x},\widehat{y})$. For
$\widehat{x},\widehat{y}\in\widehat{\mathcal{G}_c}$, (resp. $x,y\in\mathcal{G}_c^{-}$),
$\mathcal{R}^0$, $\mathcal{R}^{0,0}$, $\mathcal{P}^0$ and
$\mathcal{P}^{0,0}$ (resp. $\mathcal{R}^0$ and $\mathcal{P}^0$) are
defined naturally.

\begin{lem}
Empty rectangles and empty pentagons have index $1$.
\end{lem}

\begin{lem}\label{lem:classification}
Let $D\in\mathcal{D}(\widehat{x},\widehat{y})$ be a positive domain. Then    there    exists    generators
$\widehat{u_0},\widehat{u_1},\cdots,\widehat{u_k}\in\widehat{\mathcal{G}_c}$
with   $\widehat{u_0}=\widehat{x}$  and  $\widehat{u_k}=\widehat{y}$,
and domains $D_i\in(\mathcal{R}(\widehat{u_{i-1}},\widehat{u_i})\cup\mathcal{P}(\widehat{u_{i-1}},\widehat{u_i}))$
such that $D=\sum_i D_i$. This implies positive domains have non-negative index.
\end{lem}

\begin{proof}
We only prove the first part of the lemma. The second part is a
trivial implication.

If $D$ is non-trivial, choose an $x$-coordinate $x_1$ with $n_{x_1}\ne
0$ such that $x_1$ does not lie on $\alpha_c$. Either the top-right
square (or pentagon) or the bottom-left square to $x_1$
must have non-zero coefficient in $D$. Assume after a rotation
$R(\pi)$ if necessary that it is the top right square (or pentagon). Very similar to
the case for the grid diagram, we can assume that $D$ does not contain
any width one horizontal or vertical annulus through this top-right
square (or pentagon). There is a special case which needs extra
attention. If $x_1$ is on the $\alpha$ circle just below $\alpha_c$,
it is possible for $D$ to contain a width one annulus, which is tiled
by squares and $2$ pentagons. Even in this case, it is easy to see
that $D$ contains $r$, an empty rectangle or an empty pentagon
joining $\widehat{x}$ to some $\widehat{u_1}$, and
thus $D\setminus r$ has smaller sum of coefficients than $D$, and we
can proceed by induction.

So now assume there are no such annuli. We consider all the intersection points
$p$ between $\alpha$ and $\beta$ circles other than $x_1$, such that
$D$ contains the rectangle or the pentagon with $x_1$ as the bottom
left corner and $p$ as the top-right corner. Since we are only dealing
with horizontal commutations, the only type of pentagons that can
appear, will have $\rho$ in the top part of the pentagon. We call such
a rectangle or a pentagon to be the domain corresponding to $p$ (if
there is both a rectangle and a pentagon corresponding to $p$, we let
domain be the rectangle). With the
partial order being defined by inclusion of corresponding domains, let $p_0$ be
a maximal element. Now, we are looking for an $x$-coordinate other
than $x_1$ in $r_0$, the domain corresponding to $p_0$.

We proceed case by case as in the proof for Lemma
\ref{lem:grading}. All the cases are similar, except the following
one. The domain corresponding to $p_0$ is a pentagon, and the top
left-left square to $p_0$ has non-zero coefficient in $D$. Then we
start at $p_0$ and walk left towards $\rho$ along $\alpha_c'$, until
we first encounter a point $p_1$ such that the top-right square to
$p_1$ has non-zero coefficient, but the top-left square (or triangle)
to $p_1$ has zero coefficient in $D$. Such a point $p_1$ exists since
$p_0$ was a maximal element. It is easy to see that such a point $p_1$
is an $x$-coordinate.

So now we take the partial order, and restrict it to only
$x$-coordinates. If $x_2$ is a minimal element, then the domain
corresponding to $x_2$ is an index $1$ domain $r$, and $D\setminus r$
has smaller sum of coefficients than $D$, thus completing the induction.  
\end{proof}

\begin{lem}\cite{CMPOZSzDT}
For knots, with sign assignment described in the beginning of this
subsection, and with the grading being the Maslov grading, each of the
posets $\widehat{(\mathcal{G}_c)_m}$ and $(\mathcal{G}_c)^{-}_m$ are
well-defined, finite, signed and graded.
\end{lem}

In Section \ref{sec:gss}, we will see that the closed  intervals  in commutation posets
are also shellable, and so like grid posets they will also be GSS posets.

\subsection{Stabilization diagram}

Now we repeat the whole process for the stabilization diagram. We only
consider the case $(a)$ of Figure \ref{fig:stablecases}. The other
cases are obtained by different rotations. The reversal of partial
order that might happen does not pose a problem here or in Section
\ref{sec:gss}. However in Section \ref{sec:gridhomotopy}, it deserves
some special attention, and hence we will also deal with case $(b)$ there.  

Let $H$ be the grid diagram before stabilization and let $G$ be the
diagram after. Let $\alpha_s$ and $\beta_s$ be the extra circles, and
let $\rho$ be their intersection point. Let $G_s=(G,\rho)$ be the
stabilization diagram. Let $\widehat{\mathcal{I}}$
(resp. $\mathcal{I}_s^{-}$) be the set
of all intersection points in $\widehat{\mathcal{G}}$
(resp. $\mathcal{G}^{-}$) which contain
$\rho$ as one of its coordinates. Let
$\widehat{\mathcal{NI}}=\widehat{\mathcal{G}}\setminus
\widehat{\mathcal{I}}$ and let
$\mathcal{NI}^{-}=\mathcal{G}^{-}\setminus \mathcal{I}^{-}$.

Let us number the $X$ and $O$ marking in $G$ as $X_0,X_1,\ldots,X_N$
and $O_0,O_1,\ldots,O_N$ such that the neighborhood of $\rho$ contains
the points $O_0,X_0,X_1$ with $O_0$ directly above $X_0$, and $O_1$
lies in the same horizontal annulus as $X_0$. Thus $H$ is
obtained from $G$ by deleting $\alpha_s,\beta_s,O_0$ and $X_0$, and
the rest of the points being numbered the same.

Note that there is a natural bijection $\widehat{f}$ from $\widehat{\mathcal{I}}$ to
$\widehat{\mathcal{H}}$, and we will always identify them in this
subsection using this bijection. This bijection actually induces a map $f^{-}$
from $\mathcal{I}^{-}$ to $\mathcal{H}^{-}$ given by
$f^{-}(\widehat{x}\prod_{i=0}^N U_i^{n_i})=f(\widehat{x})U_1^{n_0+n_1}\prod_{i=2}^N U_i^{n_i}$.

We define $\widehat{\mathcal{G}_s}$ (resp. $\mathcal{G}_s^{-}$) as a disjoint union of
$\widehat{\mathcal{G}}$ (resp. $\mathcal{G}^{-}$) and two copies $\widehat{\mathcal{H}}$
and $\widehat{\mathcal{H}'}$ of $\widehat{\mathcal{H}}$ (resp. one
copy of $\mathcal{H}^{-}$). In $\widehat{\mathcal{G}_s}$, the $(M,A)$
grading is obtained from the one induced from $\widehat{\mathcal{G}}$,
the one induced from $\widehat{\mathcal{H}}$ shifted by $(-1,0)$ and
the one induced from $\widehat{\mathcal{H}'}$ shifted by $(-2,-1)$. In
$\mathcal{G}_s^{-}$, the Alexander grading is the one induced from
$\mathcal{G}^{-}$ and $\mathcal{H}^{-}$, and the Maslov grading is
obtained from the one induced from $\mathcal{G}^{-}$ and the one
induced from $\mathcal{H}^{-}$ shifted by $-1$.

For $\widehat{x},\widehat{y}$ both in $\widehat{\mathcal{G}}$ or $\widehat{\mathcal{H}}$
or $\widehat{\mathcal{H}'}$, $\mathcal{D}(\widehat{x},\widehat{y})$ is defined like in the
subsection for grid diagrams. For $\widehat{x}\in\widehat{\mathcal{G}}$ and
$\widehat{y}\in\widehat{\mathcal{H}}$ or $\widehat{y}\in\widehat{\mathcal{H}'}$, we define
$\mathcal{D}(\widehat{x},\widehat{y})=\mathcal{D}(\widehat{x},\widehat{f}^{-1}(\widehat{y}))$. Domains
$\mathcal{D}^0$, $\mathcal{D}^{0,0}$, Maslov index $\mu$, empty
rectangles $\mathcal{R}$, $\mathcal{R}^0$ and $\mathcal{R}^{0,0}$ are
all defined analogously. However there is one minor change. For
$\widehat{x}\in\widehat{\mathcal{G}}$ and $\widehat{y}\in\widehat{\mathcal{H}}$ (but not
$\widehat{\mathcal{H}'}$), while defining
$\mathcal{D}^0(\widehat{x},\widehat{y}),\mathcal{D}^{0,0}(\widehat{x},\widehat{y}),
\mathcal{R}^0(\widehat{x},\widehat{y})$ and 
$\mathcal{R}^{0,0}(\widehat{x},\widehat{y})$, we require all the domains to have
$n_{X_0}=1$ (instead of the usual $0$).

In $\widehat{\mathcal{G}_s}$ the partial order is given by $\widehat{y}\preceq
\widehat{x}$ if there is a positive domain in $\mathcal{D}^{0,0}(\widehat{x},\widehat{y})$. Note,
for $\widehat{x}\in\widehat{\mathcal{G}}$ and $f(\widehat{x})\in\widehat{\mathcal{H}'}$,
the trivial domain is a positive domain in $\mathcal{D}^{0,0}(\widehat{x},f(\widehat{x}))$
and hence $f(\widehat{x})\prec \widehat{x}$ (indeed $f(\widehat{x})\leftarrow \widehat{x}$). However for
$\widehat{x}\in\widehat{\mathcal{G}}$ and $\widehat{y}\in\widehat{\mathcal{H}}$, partial
order does not come from trivial domains, due to the $n_{X_0}=1$
condition. 

For $\mathcal{G}_s^{-}$, if $x,y$ both in $\mathcal{G}^{-}$ or
$\mathcal{H}^{-}$, the partial order is the usual one given by
positive domains in $\mathcal{D}^0(x,y)$. For $\widehat{x}\in\widehat{\mathcal{G}}$
and $\widehat{y}\in\widehat{\mathcal{H}}$, we declare the partial order to be given by
$\widehat{y} U_1^{n_0+n_1+k_0+k_1}\prod_{i>1} U_i^{n_i+k_i}\preceq \widehat{x}
U_0^{n_0}U_1^{n_1}\prod_{i>1}U_i^{n_i}$ if there is a positive domain
in $\mathcal{D}^0(\widehat{x},\widehat{y})$ which has
$n_{O_i}=k_i$. Again note that we require all such domains to have
$n_{X_0}=1$ and hence
trivial domains do not contribute to the partial order.

The sign assignment is the one for the grid diagram $G$, and the
grading assignment is the Maslov grading. Most of the results proved in
the subsection for grid diagrams are true here with some minor
modifications. We just mention the few results that are slightly
different.

\begin{lem}
For $\widehat{x},\widehat{y}\in\widehat{\mathcal{G}_s}$ (resp. $x,y\in\mathcal{G}_s^{-}$), and $D$
a positive domain in $\mathcal{D}^0(\widehat{x},\widehat{y})$ (resp. $\mathcal{D}^0(x,y)$), at most two regions around
$\rho$ have non-zero coefficients, and each coefficient is at most $1$.
\end{lem}

\begin{lem}
A periodic domain $D$ for the grid $G$, with $n_{X_i}=0\forall i$,
$n_{O_i}=0\forall i>1$, and $n_{O_0}+n_{O_1}=0$ is generated by the
special domain $D_s$ which is the vertical annulus through $X_0$ minus the
horizontal annulus through $X_0$.
\end{lem}

\begin{lem}
For $x\in\mathcal{G}^{-}$ and $y\in\mathcal{H}^{-}$, there are at most
two positive domains in $\mathcal{D}^0(x,y)$. For any
other combination of $x,y$ in $\mathcal{G}_s^{-}$ or
$\widehat{x},\widehat{y}$ in $\widehat{\mathcal{G}_s}$, there is at most
one such positive domain.
\end{lem}

Like before, in each Alexander grading $m$, the stabilization posets
turn out to be well-defined, finite, graded and signed. In the next
section, we will prove that the closed intervals in these posets are
also shellable.

\section{GSS shellability}\label{sec:gss}

In this section we will use the posets defined in Section \ref{sec:griddiagrams} and show that they are GSS posets.

Let $G$ be
a grid with grid number $n$  drawn on a torus $T$, representing a knot
$K$. Recall that for $\widehat{x},\widehat{y}\in\widehat{\mathcal{G}}$
(resp. $x,y\in\mathcal{G}^{-}$), we have $\widehat{y}\preceq
\widehat{x}$ (resp. $y\preceq x$) if  there is a  positive domain in
$\mathcal{D}^{0,0}(\widehat{x},\widehat{y})$
(resp. $\mathcal{D}^0(x,y)$).
We now show that each closed interval  in either of these posets is
EL-shellable. For that, first note that it is enough to do it for the
case of $\mathcal{G}^{-}$. Fix a
point  $P$ in a  connected component  of $T\setminus(\alpha\cup\beta)$
containing some  marking, say  $X_1$.  Draw a  circle $l$  through $P$
which  is parallel  to  the longitude  and  is disjoint  from all  the
$\beta$ circles. We only require that our domains do not contain a
horizontal annulus through $P$.

Let $r\in\mathcal{R}^0(x,y)$ be an  empty rectangle not containing any
$X$ marking. By definition, $r$ cannot  contain the point $P$. To each such
domain $r$, we associate  a triple $(s(r),i(r),t(r))$, where $s(r)$ is
$0$  is  $D$  intersects  $l$  and is  $1$  otherwise.   If  $s(r)=0$
(resp. $s(r)=1$), $i(r)$  is the minimum number of  $\beta$ circles we
have to intersect  to reach the leftmost arc of  $r$, starting at $l$
and going  left (resp.  right)  throughout.  We always have  $t(r)$ to
denote the thickness  of a rectangle $r$.  The set  of such triples is
ordered  lexicographically, and thus  we have  a map  from the  set of
covering relations to a totally ordered set.

\begin{thm} Let  $x,y\in\mathcal{G}^{-}$. The  map which sends  a covering
relation represented  by an empty rectangle  $r$ to $(s(r),i(r),t(r))$
induces an EL-shelling on the interval $[y,x]$.
\end{thm}

Note that the  interval $[y,x]$ is non-empty if  and only if $y\preceq
x$. From now on, we only consider that case. Also note that given a
generator $z\in\mathcal{G}^{-}$, and  a triple $(s,i,t)$, there is  at most one generator
$z'$ covering $z$, such that the covering relation corresponds to that
triple.    Thus  each   maximal  chain   in  $[y,x]$   has   a  unique
labeling. Thus there is a unique minimum chain $c$. The following two
lemmas will prove the above theorem.

\begin{lem} The unique minimum chain $c$ is increasing.
\end{lem}

\begin{proof} Assume not. Let $m\leftarrow n\leftarrow p$ be the first
place in $c$ where the labeling decreases. Let $r_1$ and $r_2$ be the
two rectangles  involved for the  two covering relations.   Since each
vertical  and each  horizontal annulus  has at  least one  $X$ marking, so
$\partial  (r_1+r_2)$  must be  non-zero  on  at  least three  $\beta$
circles (and clearly on at most four $\beta$ circles).

If it is  non-zero on exactly four $\beta$  circles, then switch $r_1$
and $r_2$,  and thus  we have  produced a new  maximal chain  which is
smaller than  $c$ and thus  contradicting the assumption that  $c$ was
the minimum. If on the other hand, $\partial (r_1+r_2)$ is non-zero on
exactly   three  $\beta$   circles,  then   $r_1+r_2$  looks   like  a
hexagon. Depending on the shape of the hexagon and the position of the
line $l$ only the  cases as shown in  Figure \ref{fig:cuthexagon} can occur. In each of the cases, the
lexicographically best way to divide the hexagon is shown, and in each
case, that  happens to  be the increasing  one.  This proves  that the
minimum chain $c$ is increasing.

\begin{figure}[ht] \center{\includegraphics[width=330pt]{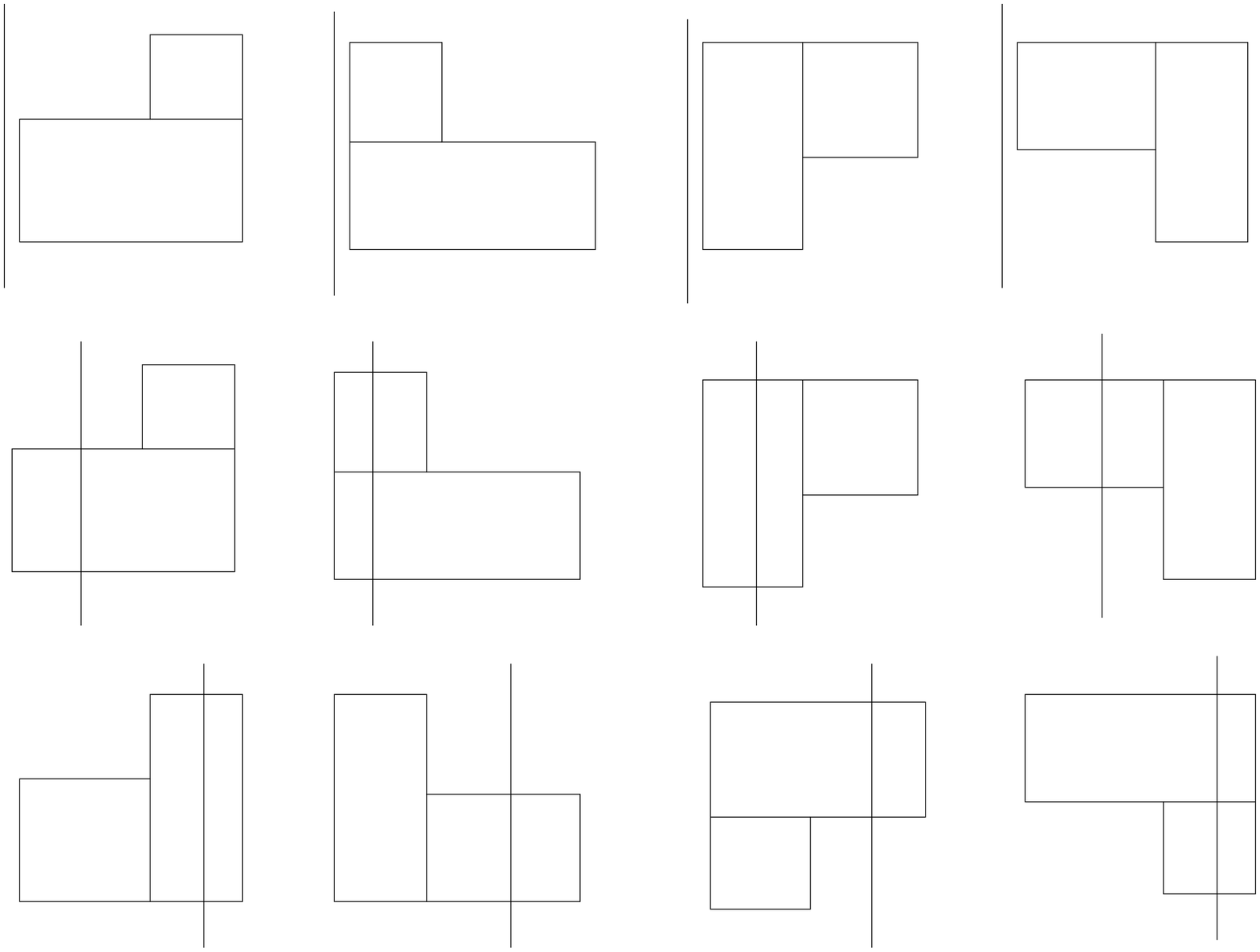}}
\caption{The lexicographically best way to cut a hexagon}\label{fig:cuthexagon}
\end{figure}
\end{proof}

\begin{lem} \label{lem:mainproof}
The minimum chain is the only increasing chain.
\end{lem}

\begin{proof}  Now we  are  trying to  prove  that there  is a  unique
increasing chain.  If possible, let there be two increasing chains $c$
and $c'$. Starting  at $y$, let us assume they  agree up to a generator
$z$.     Let    $D$    be    the    unique    positive    domain    in
$\mathcal{D}^0(x,z)$. Let  $c_1=c\cap[z,x]$ and $c_2=c'\cap[z,x]$. Let
$r$  and $r'$  be the  rectangles  corresponding to  the two  covering
relations on $z$ coming from the  two chains $c_1$ and $c_2$.  We will
show  that  $(s(r),i(r),t(r))=(s(r'),i(r'),t(r'))$  which would  imply
that  $r=r'$, and  thus  $c$ and  $c'$  agree for  at  least one  more
generator, thus concluding the proof.

Now if $D$  does not intersect $l$,  then $s$ is forced to  be $1$. On
the other  hand, if  $D$ does intersect  $l$, then eventually  in both
$c_1$  and  $c_2$  some covering relation  will  have  $s=0$,  and since  both  are
increasing chains, so they both must  start with $s=0$. So we see that
$s$ is fixed.

First we analyze the case when $s=1$. So assume the whole domain $D$
lies to the right of $l$, and let $i_0$ be the minimum number of
$\beta$ circles we have to cross to reach $D$ from $l$ going right
throughout.  Clearly $i$, the second coordinate in the triple
$(s,i,t)$, can never be smaller than $i_0$. Also since the whole
domain $D$ has to be used up in both the chains $c_1$ and $c_2$, so at
some point $i$ will be equal to $i_0$. Since both $c_1$ and $c_2$ are
increasing, we see that this fixes $i=i_0$.

To see that $t$ is also fixed, we need an induction statement. Look at
all $p$ of  the form $z\leftarrow p\preceq x$,  such that the covering
relation $z\leftarrow p$ is by  a rectangle with $i=i_0$. Let $r_0$ be
the thinnest  rectangle among them and  let $t_0$ be  the thickness of
$r_0$. Our  induction claim is  that, at some  point in the  chain, we
have to use a rectangle with $i=i_0$ and $t\leq t_0$. The induction is
done on the length of  the interval $[z,x]$.  Clearly when this length
is $2$,  the statement is  true.  Let us  assume that we do  not start
with  the  thinnest  rectangle,  but  rather start  with  a  rectangle
$\widetilde{r_0}$.  Since  both $r_0$ and  $\widetilde{r_0}$ are index
$1$  domains, they  do  not contain  any  coordinate of  $z$ in  their
interior, and hence the local diagram must look like Figure \ref{fig:thinrect1}.

\begin{figure}[ht]
\psfrag{l}{$l$}
\psfrag{i0}{$i_0$}
\psfrag{t0}{$t_0$}
\psfrag{r0}{$r_0$}
\psfrag{tr0}{$\widetilde{r}_0$}
\psfrag{z1}{$z_1$}
\psfrag{z2}{$z_2$}
\psfrag{z3}{$z_3$}
\begin{center} \includegraphics[width=170pt]{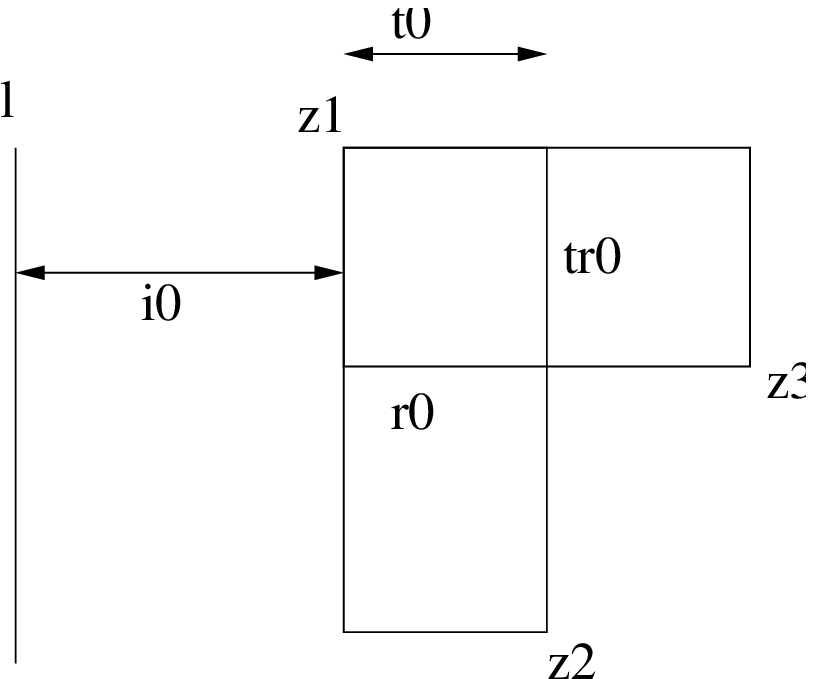}
\end{center}
\caption{Fixing the thickness of the starting rectangle when $s=1$}\label{fig:thinrect1}
\end{figure}

Since $D\setminus \widetilde{r_0}$ has Maslov index $1$ lower than $D$
and has a starting rectangle with $(i,t)=(i_0,t_0)$, so induction
applies finishing the proof. Thus in both the chains $c_1$ and $c_2$,
at some point we have to use a rectangle with $i=i_0$ and $t\leq t_0$.
But since $c_1$ and $c_2$ are increasing, and $(i_0,t_0)$ is the
smallest value of $(i,t)$ that we can start with, we have to start
with $t=t_0$. Thus this fixes $t$.

Now let us assume $s=0$. We need an induction statement to show that
$i$ is fixed.  For each coordinate $z_i$ of $z$, consider the
horizontal line $h_i$ lying on some $\alpha$ curve, which starts at
$z_i$ and ends at $l$ and goes right throughout. We call $z_i$ to be
admissible if every point just below the line $h_i$ belongs to $D$.
Since the starting rectangles in the chains $c_1$ and $c_2$ have
$s=0$, so there is at least one admissible coordinate. Among all the
admissible coordinates, let $z_1$ be the one with $h_i$ having the
smallest length. Let $i_0$ be the smallest length, measured by number
of intersections with $\beta$ curves. Our induction claim is that at
some point in any increasing chain we have to use a rectangle with
$s=0$ and $i\leq i_0$.  The induction is done on the length of
$[z,x]$. Clearly when the length is $2$, the claim is true.  Let us
assume we start with a rectangle $r_0$ with $s=0$ and $i>i_0$. Since
$r_0$ has index one, so it cannot contain any $z$ coordinate in its
interior, and it also cannot contain any horizontal annulus. Thus it
is easy to see that $r_0$ has to be disjoint from $h_1$, and thus
$D\setminus r_0$ has Maslov index one lower than $D$ and still
intersects $l$ and has an admissible coordinate with $h=i_0$. Thus
induction applies, and proves our claim.

Now it is easy to see that the starting rectangles in the chains $c_1$
and $c_2$ must have $s=0$  and $i\geq i_0$.  Since both are increasing
chains, so we must start with a rectangle with $(s,i)=(0,i_0)$. Now we
want to show that $t$ is also fixed. This is also by an induction very
similar to the ones above. Consider all $p$ with $z\leftarrow p\preceq
x$,   such   that   the   covering  relation   $x\leftarrow   p$   has
$(s,i)=(0,i_0)$. Let  $r_0$ be the  thinnest rectangle among  all such
covering  relation, and  let $t_0$  be  the thickness  of $r_0$.   The
induction claim is that at some point in any increasing chain, we have
to  use a  rectangle with  $(s,i)=(0,i_0)$  and $t\leq  t_0$, and  the
induction is done  on the length of $[z,x]$. Again  it is trivial when
the length is $2$. Assume  we start with a rectangle $\widetilde{r_0}$
with   $(s,i)=(0,i_0)$   and    $t>t_0$.    Since   both   $r_0$   and
$\widetilde{r_0}$ have index one, they must look like Figure \ref{fig:thinrect0}.

\begin{figure}[ht]
\psfrag{l}{$l$}
\psfrag{i0}{$i_0$}
\psfrag{t0}{$t_0$}
\psfrag{r0}{$r_0$}
\psfrag{tr0}{$\widetilde{r}_0$}
\psfrag{z1}{$z_1$}
\psfrag{z2}{$z_2$}
\psfrag{z3}{$z_3$}
\begin{center} \includegraphics[width=170pt]{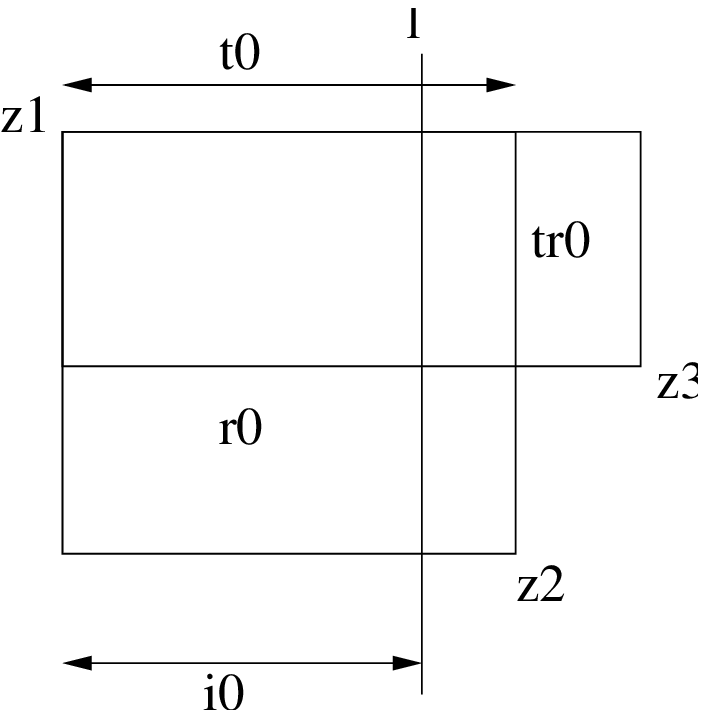}
\end{center}
\caption{Fixing the thickness of the starting rectangle when $s=0$}\label{fig:thinrect0}
\end{figure}

Note that $D\setminus\widetilde{r_0}$ has index one lower than $D$ and
it still  intersects $l$,  and it still  has an  admissible coordinate
with $h=i_0$. Thus  induction applies. Since $c_1$ and  $c_2$ are both
increasing, this  implies that they  both must start with  a rectangle
with $(s,i,t)=(0,i_0,t_0)$. Thus we see that the thickness is fixed.

As explained earlier, this finishes the proof.
\end{proof}

Using the theorems from  section \ref{sec:posets}, this implies the
following.

\begin{thm}  Each subinterval  of an  interval  in the  grid poset  is
shellable. For intervals  of the form $(y,x)$, the  order complex is a
sphere, and for intervals of the form $[y,x]$, $[y,x)$ or $(y,x]$, the
order complex is a ball.
\end{thm}

Thus using the results from Section \ref{sec:griddiagrams}, we see that
$\widehat{G}$, $\widehat{\mathcal{G}_m}$ and $\mathcal{G}_m^{-}$ (in
each Alexander grading $m$) are GSS posets.

Now we concentrate the commutation posets $\widehat{\mathcal{G}_c}$
and $\mathcal{G}_c^{-}$. Let $(G_c,\rho)$ be a commutation diagram.
We are trying to prove that closed intervals in these posets are shellable.
Once more it is enough to restrict our attention to closed intervals
in $\mathcal{G}_c^{-}$.

\begin{thm}
Closed intervals in the commutation poset are shellable.
\end{thm}

\begin{proof}
We do not know if the closed intervals are always EL-shellable. We shall
only prove that the closed intervals are shellable. For
$\widehat{x},\widehat{y}\in\widehat{\mathcal{G}_c}$, let
$D\in\mathcal{D}^0(\widehat{x},\widehat{y})$ be a
positive domain with $n_{O_i}(D)=k_i$. If $x=\widehat{x}$ and
$y=\widehat{y}\prod U_i^{k_i}$, we will prove that the closed
interval $[y,x]$ is shellable. Note $n_{\rho}(D)<1$. So
we prove this by taking cases.

\textbf{Case 1}: $D$ is the unique positive domain joining $x$ to $y$ and
$n_{\rho}(D)\neq \frac{3}{4}$.

We can choose any vertical line $l$ disjoint from all $\beta$ circles
(indeed we can choose a vertical line through $\rho$) and define
$(s,i,t)$ as in the proof of the previous theorem. Essentially the
same proof shows that this provides an EL-shelling. It is important to
note that we can also apply the rotation $R(\frac{\pi}{2})$ (such that
the horizontal commutation becomes a vertical commutation), and then
take a vertical line $l$ (this time disjoint from all the $\alpha$
circles), and then define $(s,i,t)$ which still induces an
EL-shelling. The line $l$ has to be disjoint from $\alpha_c$ and
$\alpha_c^{\prime}$ (which are now vertical circles), and we stipulate
(for defining $i$ and $t$) that both of them are equidistant from $l$.

\textbf{Case 2}: $n_{\rho}=\frac{3}{4}$.

In this case, using Lemma \ref{lem:rho1}, $D$ is the unique positive domain
joining $x$ to $y$. 
Choose a vertical line $l$ passing through $\rho$, the chosen intersection
point between $\alpha_c$ and $\alpha_c^{\prime}$. To each covering relation,
associate a $4$-tuple $(s,i,t,p)$, where $s$, $i$ and $t$ are defined
similarly and $p=1$ if the covering relation corresponds to a
pentagon, and is $0$ otherwise. Thus given $y$, and a $4$-tuple
$(s,i,t,p)$, there is at most one $x$ with $y\leftarrow x$
corresponding to that $4$-tuple. The tuples are ordered
lexicographically, and thus all maximal chains in $[y,x]$ have their
edges labeled by a totally ordered set, and hence themselves get an
induced total ordering. We claim this ordering gives the required
shelling.

We follow the general outline of the proof of Theorem \ref{thm:el}.  Let
$\mathfrak{m}_1$ and $\mathfrak{m}_2$ be two maximal chains, with
$\mathfrak{m}_1<\mathfrak{m}_2$. Let $\mathfrak{m}_1$ and
$\mathfrak{m}_2$ agree from $y$ to $y_1$ and then start to disagree,
and then agree once more at $x_1$ (and then maybe disagree
again). Thus we can restrict our attention on the interval $[y_1,
x_1]$, which has a smaller length. Hence by induction, we will be
done. Thus we can assume $y_1=y$ and $x_1=x$, i.e. $\mathfrak{m}_1$
and $\mathfrak{m}_2$ never agree. The domain $D$ corresponding to
$[y,x]$ might now have $n_{\rho}(D)\ne \frac{3}{4}$. But note $D$ is
still the unique positive domain joining $x$ to $y$, and hence if
$n_{\rho}(D)\ne \frac{3}{4}$, then we have reduced this case to the
previous case. Hence assume $D$ still has $n_{\rho}=\frac{3}{4}$.

If $\mathfrak{m}_2$ has a subchain $y_{k-1}\leftarrow y_k\leftarrow
y_{k+1}$, where the $4$-tuples corresponding to the two covering
relations decrease, and the domain corresponding to
$[y_{k-1},y_{k+1}]$ does not look like any of the two domains in Figure \ref{fig:commspdomain}$(a)$, then we can
change $\mathfrak{m}_2$ by replacing
$y_k$ with $y_k^{\prime}$ with $y_{k-1}\leftarrow
y_k^{\prime}\leftarrow y_{k+1}$. Call such an operation a switching
operation. A case by case analysis shows that the new maximal chain
obtained after a switching operation is smaller than the
original. Call the operation of changing one element of a maximal
chain to get a smaller maximal chain, a generalized switching
operation. Thus a switching operation is a generalized switching operation.

\begin{figure}[ht] 
\psfrag{a}{$(a)$}
\psfrag{b}{$(b)$}
\center{\includegraphics[width=200pt]{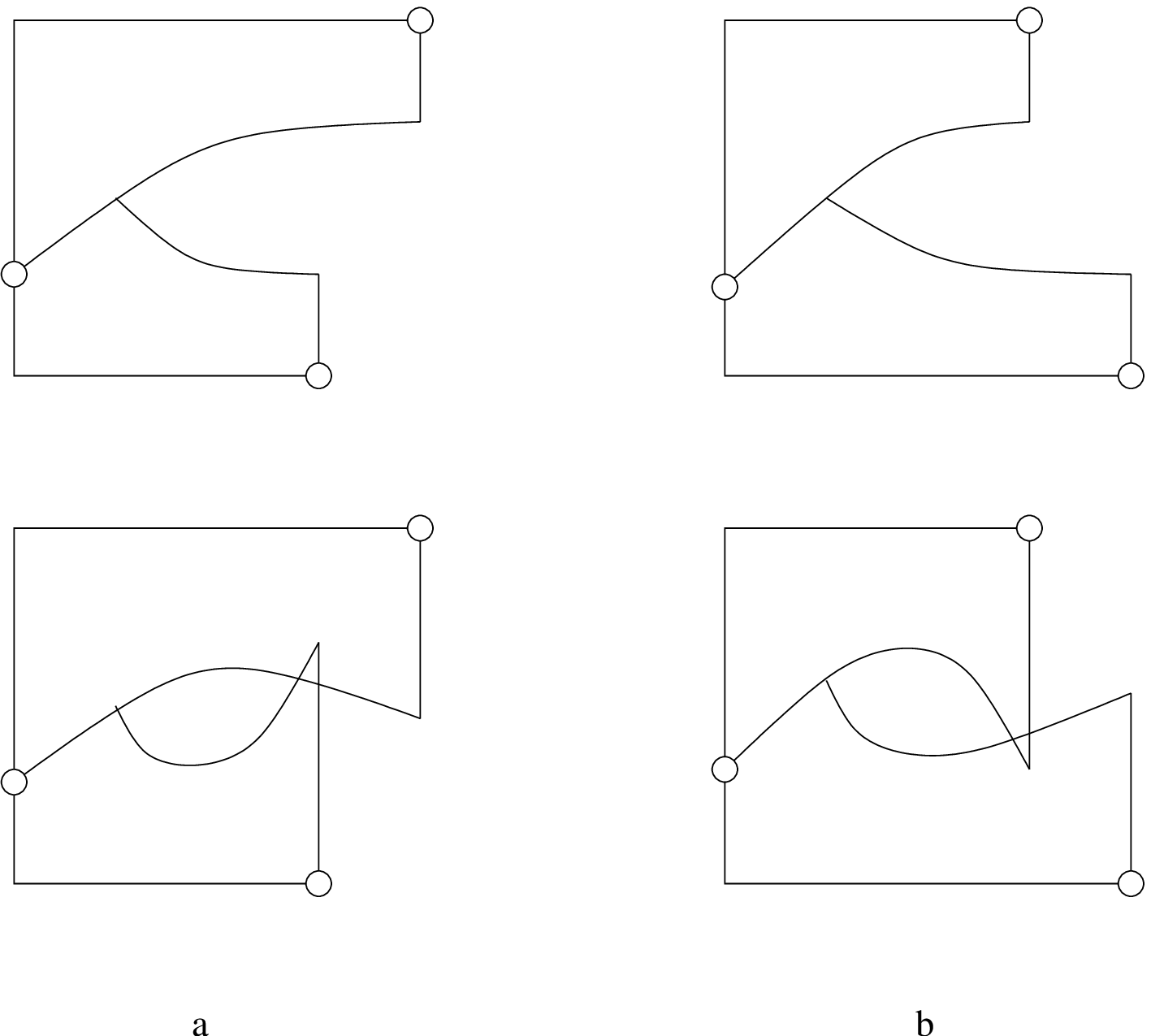}}
\caption{The special index $2$ domain}\label{fig:commspdomain}
\end{figure}

Since we are trying to prove shellability, hence we can assume that
$\mathfrak{m}_2$ does not admit any generalized switching operation.
In that case there is an element $z$ in $\mathfrak{m}_2$,
such that $\mathfrak{m}_2\cap [y,z]$ is an increasing chain in
$(\mathcal{G}')^{-}$ and $\mathfrak{m}_2\cap [z,x]$ is an increasing chain
which starts with an empty pentagon but there is no $x'\in[z,x]$ such
that the domain corresponding to $[z,x']$ looks any of the two domains
like Figure \ref{fig:commspdomain}$(b)$. We call such a maximal chain to be
quasi-increasing. Thus in a quasi-increasing chain, there exists
$y'\in[y,z]$ and $x'\in[z,x]$ such that $y'\leftarrow z\leftarrow x'$
and the index $2$ domain corresponding to $[y',x']$ is one of domains
shown in Figure \ref{fig:commspdomain}. In all the cases, the $z$-coordinates are marked.

Now we want to show that $\mathfrak{m}_2$ is the smallest
chain. This will rule out the possibility of having a chain
$\mathfrak{m}_1$ with $\mathfrak{m}_1<\mathfrak{m}_2$, and thus
finishing the proof. Thus, if possible, let
$\mathfrak{m}_1<\mathfrak{m}_2$. We can do the generalized switching operations as
described above, on $\mathfrak{m}_1$, such that $\mathfrak{m}_1$ also
becomes quasi-increasing. Now if we
show $\mathfrak{m}_1=\mathfrak{m}_2$, we will have the required
contradiction.

Thus we only need to show that there is a unique quasi-increasing
chain. The proof is essentially the same as the proof of uniqueness of
increasing chain in the previous theorem. Thus in this case, we are done.

\textbf{Case 3}: There are exactly two positive domains $D$ and $D'$
joining $x$ to $y$.

By assumption, note that both $D$ and $D'$ have $n_{O_i}=k_i$. Also
both $D$ and $D'$ must have $n_{\rho}=\frac{1}{4}$. For simplicity, we apply the
rotation $R(\frac{\pi}{2})$. After rotation, all the $\alpha$ circles
(incl. $\alpha_c$ and $\alpha_c'$) become vertical circles. Let $D$ be
the domain which has non-zero coefficient in the region immediately to
the left of $\rho$. We choose the vertical circle $l$ to be line
immediately to the left of $\alpha_c$ and $\alpha_c'$. We define
$(s,i,t)$ as in the proof of the previous
theorem. Note that we assume
both $\alpha_c$ and $\alpha_c'$ be to distance $1$ to the right of
$l$. Note that given $y$ and a triple $(s,i,t)$ there is at most one
$x$ with $y\leftarrow x$ corresponding to that triple. Thus each
maximal chain gets a unique labeling. We use this labeling to
totally order all maximal chains that come from $D$, and also all
maximal chains that come from $D'$. We then declare all maximal chains
that come from $D'$ to be smaller than all maximal chains that come
from $D$. We claim that this ordering is a shelling.

Again following the general outline of the proof of Theorem
\ref{thm:el}, let $\mathfrak{m}_1$ and $\mathfrak{m}_2$ be two maximal
chains with $\mathfrak{m}_1<\mathfrak{m}_2$. By restricting to smaller
chains if necessary, we can assume that the two maximal chains are
disjoint. After restricting to smaller chains, we can still assume
that $D$ and $D'$ are two distinct domains joining $x$ to $y$, or else we have
reduced this to an earlier case.

Now we can assume that $\mathfrak{m}_2$ is a non-decreasing chain,
since otherwise we can do a switching operation to make it
smaller. But each domain has a unique non-decreasing chain, which in
addition is the smallest chain among all maximal chains coming from
that domain. Since $\mathfrak{m}_1$ is a maximal chain which
is smaller than $\mathfrak{m}_2$, hence $\mathfrak{m}_2$ must be the
unique non-decreasing chain coming from $D$.

By assumption, the line $l$ lies entirely inside $D$, and hence the
first two covering relations in $\mathfrak{m}_2$ starting at $y$ must
have $(s,i,t)$ as $(0,1,1)$. Thus we can do a switch, where this index
$2$ domain can be replaced another index $2$ domain, which is this
domain minus $D_s$, where $D_s$ is the special domain from Figure
\ref{fig:specialdomain}. After the switch, the new maximal chain comes
from $D'$, and hence is smaller than $\mathfrak{m}_2$. This completes
the proof of shellability.
\end{proof}

Since the commutation poset was already a graded and signed poset,
this completes the proof that it is a GSS poset. We now prove that the
stabilization poset is also shellable.

\begin{thm}
The stabilization poset is shellable.
\end{thm}

\begin{proof}
In both $\widehat{\mathcal{G}_s}$ and $\mathcal{G}_s^{-}$, even
if we are allowed to pass through $X_0$, the proof of shellability
(indeed EL-shellability) follows directly from the proof of
EL-shellability of intervals in the grid poset. There are only two
cases which are slightly different.

The first case that is  slightly different is  when
$\widehat{x}\in\widehat{\mathcal{G}}$ and
$\widehat{y}\in\widehat{\mathcal{H}^{\prime}}$. Here in any maximal
chain, there will be exactly one covering relation corresponding to
the trivial domain. Let us assign the $(s,i,t)$-label to each of those
covering relations as $(-1,0,0)$. It is easy to see that this
labeling still induces an EL-shelling of the interval
$[\widehat{y},\widehat{x}]$. In fact this interval is the Cartesian
product of the  posets $[\widehat{f}^{-1}(\widehat{y}),\widehat{x}]$
and $I$, where $I$ is a chain of length $2$, and since each of the
posets is shellable, their Cartesian product is shellable.

The other case that
presents some difficulties is when $x\in\mathcal{G}^{-}$ and
$y\in\mathcal{H}^{-}$, and there are exactly two domains $D$ and $D'$ joining $x$
to $y$.

In this case, we proceed like the last case in the previous
theorem. If $n_{O_0}(D)=1$,
then we declare maximal chains coming from $D'$ to be smaller than
those coming from $D$. For maximal chains coming from $D$, we choose
$l$ to be the vertical line passing through $X_0$. We define $(s,i,t)$
in the standard way, and this induces a total ordering among all
maximal chains coming from $D$. For maximal chains coming from $D'$,
we apply the rotation $R(\frac{\pi}{2})$, and then choose $l$ to be
the vertical line through $X_0$, and then define $(s,i,t)$ to induce a
total ordering among all maximal chains coming from $D'$.

We choose two maximal chains $\mathfrak{m}_1$ and $\mathfrak{m}_2$
with $\mathfrak{m}_1<\mathfrak{m}_2$. We can assume that they are
disjoint. If both of them come from either $D$ or $D'$, the proof is
very similar to the proof of shellability of intervals of grid
posets. Thus we can assume $\mathfrak{m}_1$ comes from $D'$ and
$\mathfrak{m}_2$ comes from $D$. We can also assume that $\mathfrak{m}_2$ is the unique
non-decreasing chain coming from $D$. Thus the first two covering
relations in $\mathfrak{m}_2$ must have $(s,i,t)$ as $(0,1,1)$ as
their labeling, and
hence we can modify that index $2$ domain by subtracting the vertical annulus
through $X_0$ and adding the horizontal annulus through
$X_0$. After this switch, we get a maximal chain coming from $D'$,
thus completing the proof of shellability.
\end{proof}

\section{Applications}\label{sec:applications}

Given a grid diagram for a knot, the above theorems allow us to define
some CW complexes. The constructions work for any GSS poset, but we
only do it for the grid poset $\mathcal{G}$. (Here $\mathcal{G}$ could
be $\widehat{\mathcal{G}}$, $\widehat{\mathcal{G}_m}$ or
$\mathcal{G}_m^{-}$ for any Alexander grading  $m$). We start with the easiest construction.

\subsection{Order Complex} We  can give a CW complex  structure on the
order  complex where  the  $k$-cells are  closed  intervals of  length
$(k+1)$.   The boundary  map maps  to all  the closed  subintervals of
length $k$. The boundary map is well defined because the union of such
subintervals forms a sphere of the right dimension.

\begin{thm}  The above  defined  CW  complex is  well  defined and  is
homeomorphic to the order complex.
\end{thm}

\begin{proof} Recall that the order  complex of an interval $[y,x]$ of
length $(k+1)$  is a ball of  dimension $k$. The order  complex of the
whole poset is just the union of  all such balls, thus we only need to
understand  the boundary  map. The  boundary of  the order  complex of
$[y,x]$ consists of all submaximal  chains that are covered by exactly
one  chain,  or  in  other   words,  maximal  chains  in  $[y,x)$  and
$(y,x]$. But  the order complex  of each of  $[y,x)$ and $(y,x]$  is a
ball of  dimension $(k-1)$ with  common boundary the order  complex of
$(y,x)$  which  is a  sphere  of  dimension  $(k-2)$. Thus  the  order
complexes of  $[y,x)$ and $(y,x]$ glue  to form a  sphere of dimension
$(k-1)$, and  is the boundary of  the order complex  $[y,x]$. Thus the
boundary map in  the order complex is the same as  the boundary map in
our CW complex. This shows that  the CW complex is well defined and is
same as the order complex.
\end{proof}

\subsection{Fake    moduli   space}    Given    a   positive    domain
$D\in\mathcal{D}^0(x,y)$  with  $\mu(D)=k$, we  construct  a CW  complex
which has many properties of what the actual moduli space should have,
although it is not clear whether the real moduli space will always be
homeomorphic to this space. The
$0$-cells will  correspond to the  the maximal chains in  $[y,x]$, the
$1$-cells  will  correspond  to   the  submaximal  chains  in  $[y,x]$
containing  both  the  endpoints,  and  in general  an  $r$-cell  will
correspond to a chain in $[y,x]$ containing $(k-r+1)$ points including
both the endpoints, and the unique $(k-1)$-cell corresponds to the $2$
element chain $\{y,x\}$. The boundary map is injective and is given by co-inclusion.

\begin{thm}  The above  defined  CW  complex is  well  defined. It  is
homeomorphic to a ball, and  its boundary is homeomorphic to the order
complex of $(y,x)$.
\end{thm}

\begin{proof} Let  us prove  this by induction  on $k$, so  assume the
theorem holds  for $\mu(D)\leq k-1$. By  boundary of our  CW complex, we
mean  everything  except  the  top dimensional  $(k-1)$-cell.   So  by
induction,  the boundary of  our CW  complex is  a $(k-2)$-dimensional
manifold $M$.  All  we need to show is that  $M$ is PL-homeomorphic to
the order  complex of  $(y,x)$. Once we  have proved that,  both being
spheres of dimension $(k-2)$, the attaching map of the $(k-1)$-cell is
forced, thus completing the induction.

Consider the order complex of $(y,x)$. Its $r$-simplices correspond to
chains of length $r$ in $(y,x)$. On the other hand $M$ is a CW complex
whose  $r$-cells   correspond  to   chains  of  length   $(k-1-r)$  in
$(y,x)$. The boundary map of $(y,x)$  is same as the coboundary map of
$M$, which is  given by inclusion. Since the  order complex of $(y,x)$
is a  manifold (in fact a  sphere) of dimension $(k-1)$,  hence $M$ is
just  the dual  triangulation of  the order  complex of  $(y,x)$. This
completes the proof.
\end{proof}

\subsection{Grid spectral  sequence} We try to  construct CW complexes
whose  boundary maps correspond  to the  grid homology  boundary. This
will ensure that the homology of  the CW complex is the grid homology.
We  start with  a  very  simple example.  Consider  the order  complex of
$(y,\infty)$.  It has  a CW complex structure where  the $r$-cells are
elements $z\in\mathcal{G}$  with $y\prec z$ and  $M(z,y)=r+1$, and the
boundary maps correspond to covering relations.

\begin{thm}\label{thm:yinfty} The above  CW complex is well defined  and is homeomorphic
to the order complex of $(y,\infty)$.
\end{thm}

\begin{proof} For any $z$ with $y\prec z$ and $M(z,y)=r+1$, the order
  complex of $(y,z]$ is a ball of dimension $r$, or in other words an
  $r$-cell. The union of such cells make the order complex, thus we
  only need to show that the boundary maps are the same for the order
  complex and the CW complex.  The boundary in the order complex
  corresponds precisely to the maximal chains in $(y,z)$, or in other
  words maximal chains of $(y,p]$ where $p$ is covered by $z$. Since
  $p$ being covered by $z$ in the grid poset is equivalent to saying
  that $p$ appears in $\partial z$ in the grid homology boundary map,
  we conclude that the boundary maps for the order complex are same as
  the ones for the CW complex.
\end{proof}

In the later sections, we will constantly be dealing with pointed CW complexes,
so now is as good a time as any to introduce them. spaces. In a
pointed CW complex $X$, the
$(-1)$-skeleton  $X^{-1}$ is  a  point, which  is  the basepoint,  but
itself is not  considered as a cell. If there  are $k$ $0$-cells, then
the $0$-skeleton  is a  discrete union of $(k+1)$  points.  There
are no attaching maps for the $0$-cells.  The construction of the rest
of the CW complex is standard. We  define a CW complex to be finite if
it has finite number of cells.   A finite CW complex is clearly finite
dimensional.

We define a pointed CW complex to be nice if the following properties hold. 

\begin{itemize}
\item There is a unique $0$-cell (such
that the $0$-skeleton is a discrete union of $2$ points)

\item The
attaching maps for all the other cells are injective. 

\item We  define a partial  order on the  cells of the CW  complex and
  the basepoint, by declaring $a\prec b$ if $a\subseteq  \partial
  b$. This poset is a GSS poset, with the
  grading being the dimension of the cell and the sign being the
  homological sign of the boundary map.
\end{itemize}

We can extend the above theorem and construct a pointed CW
complex whose $(k+r)$-cells correspond to elements $z\in[y,\infty)$
with $M(z,y)=r$ and whose CW complex boundary maps correspond to
covering relations in $[y,\infty)$.

\begin{thm} \label{thm:construction1}For every $k\geq 0$, there is a
  well-defined pointed CW complex $S_y(k)$, such that the cells
  correspond to the elements of $[y,\infty)$, the boundary maps
  correspond to the boundary maps of the chain complex induced from
  $[y,\infty)$ and agrees with any given sign convention on it, the
  cell corresponding to $y$ has dimension $k$, and the boundary map
  every other cell is injective (which implies that $S_y(0)$ is nice).
  We furthermore have $S_y(k)=S_y(0)\wedge^k S^1$.
\end{thm}

\begin{proof}
  We extend the shellable poset $[y,\infty)$ by
  attaching elements $x_0,x_1,y_1,\ldots,x_k,y_k,x_{k+1}$, such that
  $x_0$ is covered by precisely the elements that cover $y$ and with
  the same sign for each covering relation, and each of $x_i$ and
  $y_i$ is covered $x_{i-1}$ and $y_{i-1}$ with positive and negative
  signs respectively.  Using Lemmas \ref{lem:shell1} and
  \ref{lem:shell2}, we see that this new poset is also shellable.  Let
  $P_0$ be the poset defined as $P_0=(x_{k+1},y]\cup(x_{k+1},x_0]$.

Now consider  the order  complex of $(x_{k+1},\infty)$.  It has  a CW
complex structure whose cells correspond to the elements of $(x_{k+1},
\infty)$, and  the boundary maps  represent the covering
relations. But $P_0$ is  a thin shellable  poset, and 
hence the order complex of $P_0$ is a sphere of dimension $k$. Thus we
can treat the order complex of  $P_0$ as the cell corresponding to $y$
in our pointed  CW complex. The order complex  of $(x_{k+1},\infty)$
then has  a pointed  CW complex structure,  whose cells  correspond to
elements of  $[y,\infty)$ and whose  boundary maps correspond  to the
chain complex boundary maps.

Recall that a sign  convention $s$ assigns
$1$ or $-1$ to each  covering relation in the poset $\mathcal{G}$. Two
sign conventions are said to be equivalent if one can be obtained from
another by reversing the orientation of all the the covering relations
$z\leftarrow x$,  where exactly one of  $z$ and $x$  belong some fixed
subset of  generators. A property  that sign conventions must  have is
that  the grid  homology  boundary  map must  actually  be a  boundary
map. This means if $z\leftarrow \{p,q\}\leftarrow x$ is an interval of
length  three, then  the product  of the  signs of  the  four covering
relations is $-1$.

Note that  the boundary maps in  the CW complex  $[y,\infty)$ also has
this  property and this  equivalence. The  equivalence is  obtained by
reversing  the orientation  of the  cells corresponding  to  the fixed
subset  of generators. To  see that  it also  has the  above mentioned
property,  let $z\leftarrow  \{p,q\}\leftarrow  x$ be  an interval  of
length three. The generator $x$  will correspond to an $r$-cell, whose
boundary  will  contain two  $(r-1)$-cells  corresponding  to $p$  and
$q$. These two  cells have a common $(r-2)$-cell  on their boundaries,
coming from $z$. Thus it is easy  to see that the product of the signs
of  the four  boundary maps  has  to be  negative. 

Now we will show that this  equivalence and this property is enough to
determine the  sign in $[y,\infty)$. Fix  a maximal tree  in the graph
$[y,\infty)$. Using the
equivalence, we  can ensure  that all the  edges in this  maximal tree
have positive  sign. Now, we need  to show that the  property will fix
the sign of every other edge. Whenever  we add an edge, we get a cycle
in the  graph $[y,\infty)$,  consisting of that  edge and a  few edges
from the maximal  tree. If we can show that any  cycle is generated by
$4$-cycles coming from intervals of length three, then we are done.

Consider  two maximal  chains  in  $[y,x]$.  They  combine  to form  a
cycle. Call  such cycles to be simple  cycles. It is easy  to see that
any cycle in $[y,\infty)$ is a  sum of simple cycles.  So we only need
to  show that  any simple  cycle is  a sum  of $4$-cycles  coming from
length three  intervals. Let $\mathfrak{m}_1$  and $\mathfrak{m}_2$ be
two maximal chains in $[y,x]$.  Since $[y,x]$ is shellable, it follows
from  definition that  there is  some  total ordering  on the  maximal
chains, such that we can replace the bigger maximal chain by a smaller
one $\mathfrak{m}_3$  after modification by a $4$-cycle  coming from a
length three interval. This completes the proof that there is a unique
sign assignment on $[y,\infty)$ and hence we can choose the
orientations of the cells properly to ensure that the CW complex
boundary maps respect the sign conventions. 

Also note that
during the construction of the CW complex, when we were trying to
attach an $n$-dimensional cell, its boundary had to map injectively to
an $(n-1)$-sphere respecting some sign. Thus throughout there was
only one option, and hence there is only one such CW complex that can
be constructed with the above properties. This shows that the CW
complex is well-defined. Since with the obvious CW complex structure,
$S_y(0)\wedge^k S^1$ is another CW complex with the same properties,
we have $S_y(k)=S_y(0)\wedge^k S^1$.
\end{proof}

Indeed, the above  proof shows that if $y\prec x$ with $M(x,y)=r$,
then the $(k+r)$-ball has  a pointed CW  complex structure,
whose cells  are the generators in  $[y,x]$ and the  boundary maps are
the grid homology boundary maps. This has the following corollaries.

\begin{lem} For  any interval $[y,x]$ with $y\neq x$,  the the homology of  the chain
complex induced from the poset, is trivial.
\end{lem}

\begin{proof} The  homology of the chain complex  induced from $[y,x]$
is the reduced homology of the pointed CW complex whose
cells correspond to the generators of $[y,x]$ and whose boundary maps correspond to the
chain complex boundary map. However since that CW complex is the ball,
hence the reduced homology is trivial.
\end{proof}

\begin{lem}  There are  even  number of  generators  $z$ with  $y\prec
z\prec x$.
\end{lem}

\begin{proof} We can assume $y\prec  x$. Consider the chain complex
  induced from the poset $[y,x]$. Since it has trivial homology, there
  must be even number of generators in $[y,x]$, and hence in $(y,x)$.
\end{proof}

We  digress for a  bit to  explore some  consequences of  the previous
lemma. For the  rest of this subsection, we  work with coefficients in
$\mathbb{F}_2=\mathbb{Z}/2\mathbb{Z}$. The  grid homology boundary map
can be written succinctly as

\begin{eqnarray*}
\partial x=\sum_{y\prec x, M(x,y)=1} y
\end{eqnarray*}

Let us generalize this map to define

\begin{eqnarray*}
\partial_i x=\sum_{y\prec x, M(x,y)=i} y
\end{eqnarray*}

The    above    lemma    implies    that   for    any    $k\geq    2$,
$\sum_{i+j=k}\partial_i  \partial_j=0$. Choosing  $k=2$ tells  us that
$(\mathcal{G},\partial_1)$  is  a  chain  complex, with  homology  say
$\mathcal{G}_1$.  Choosing  $k=3$ tells  us that $\partial_2$  is well
defined  on   $\mathcal{G}_1$,  and  choosing  $k=4$   tells  us  that
$(\mathcal{G}_1,\partial_2)$  is  a chain  complex  with homology  say
$\mathcal{G}_2$.   In  general  $(\mathcal{G}_i,\partial_{i+1})$ is  a
chain complex with homology $\mathcal{G}_{i+1}$. Thus we see that this
in  fact   defines  a  spectral   sequence  starting  with   the  grid
homology. The following suggests that it is not a very exciting one.

\begin{lem} The second map $\partial_2$ is zero on $\mathcal{G}_1$.
\end{lem}

\begin{proof}   Let    $\bm   x$   be   a    homogeneous   element   in
$\mathcal{G}_1$. Thus  $\bm x$ is  a linear combination  of generators
from $\mathcal{G}$.   Recall that we are working  with coefficients in
$\mathbb{F}_2$,  thus $\bm x$  simply corresponds  to a  collection of
generators,  all  with  the   same  grading.   We  are  also  assuming
$\partial_1 {\bm x}  =0$. Let $\bm y$ be  the collection of generators
which are covered by some element  from $\bm x$. Similarly let $\bm z$
be the set  of elements which are covered by some  element of $\bm y$.
Let ${\bm z'}=\partial_2 {\bm x}$. Note that ${\bm z'}$ is a subset of
$\bm z$.  We will show that there  exists a subset ${\bm  y'}$ of $\bm
y$, such  that ${\bm z'}=\partial_1  {\bm y'}$. That would  imply that
${\bm z'}=\partial_2 {\bm x}$ is zero in $\mathcal{G}_1$.

Since $\partial_1 {\bm x}=0$, each element of $\bm y$ is covered by an
even number  of elements from  $\bm x$. Let  ${\bm y'}$ be the  set of
elements in  $\bm y$ that are  covered by $2($mod $4)$  generators from $\bm
x$. We claim $\partial_1 {\bm y'}={\bm z'}=\partial_2 {\bm x}$. Choose
an element $z\in  {\bm z}$. We now consider  the set $\{(y,x)|y\in{\bm
y}, x\in{\bm x}, z\leftarrow y\leftarrow  x\}$. It is easy to see that
there are even number of  elements in this set, and half the
cardinality of this set has the same  parity as the number
of elements in ${\bm y'}$ that  cover $z$, and also the same parity as
the number of elements in $\bm x$ that are bigger than $z$ in the grid
poset. Thus  $z$ appears in $\partial_1  {\bm y'}$ if and  only if $z$
appears in $\partial_2 {\bm x}$. This concludes the proof.
\end{proof}

Note that  the above proof can  easily be generalized to  show that if
$\bm x$ is a homogeneous element in $\mathcal{G}$ with $\partial_1 {\bm
x}=0$ and $\partial_n {\bm x}={\bm z}$, then there exists a homogeneous
element $\bm y$ in $\mathcal{G}$ such that ${\bm z}=\partial_{n-1}{\bm
y}$.  It  is not  clear whether this  is enough  to show that  all the
higher $\partial_n$'s vanish.

\section{CW complexes}\label{sec:cwcomplexes}

In the previous section, we defined a nice CW complex to be a pointed CW
complex with a unique $0$-cell, such that the attaching maps for all
the other cells are injective, and the poset whose elements are the
cells and the basepoint is a GSS poset. Hence the order complex of any
closed interval of the poset is a ball, and the order complex of any
interval of the form $(-\infty,a]$ is also a ball.

In fact, using Theorem \ref{thm:construction1}, given a suitable GSS
poset, there is one and only one nice pointed CW complex satisfying
these properties. Furthermore, we can fix the orientation of the
$0$-cell arbitrarily, but once that orientation is fixed the
orientation of every other cell and the basepoint is fixed by the sign
convention of the GSS poset.

\begin{figure}[ht]
\psfrag{b}{$b$}
\psfrag{e01}{$e^0_1$}
\psfrag{e11}{$e^1_1$}
\psfrag{e12}{$e^1_2$}
\psfrag{e13}{$e^1_3$}
\psfrag{e21}{$e^2_1$}
\psfrag{ab}{$b$}
\psfrag{ae01}{$e^0_1$}
\psfrag{ae11}{$e^1_1$}
\psfrag{ae12}{$e^1_2$}
\psfrag{ae13}{$e^1_3$}
\psfrag{ae21}{$e^2_1$}
\begin{center} \includegraphics[width=250pt]{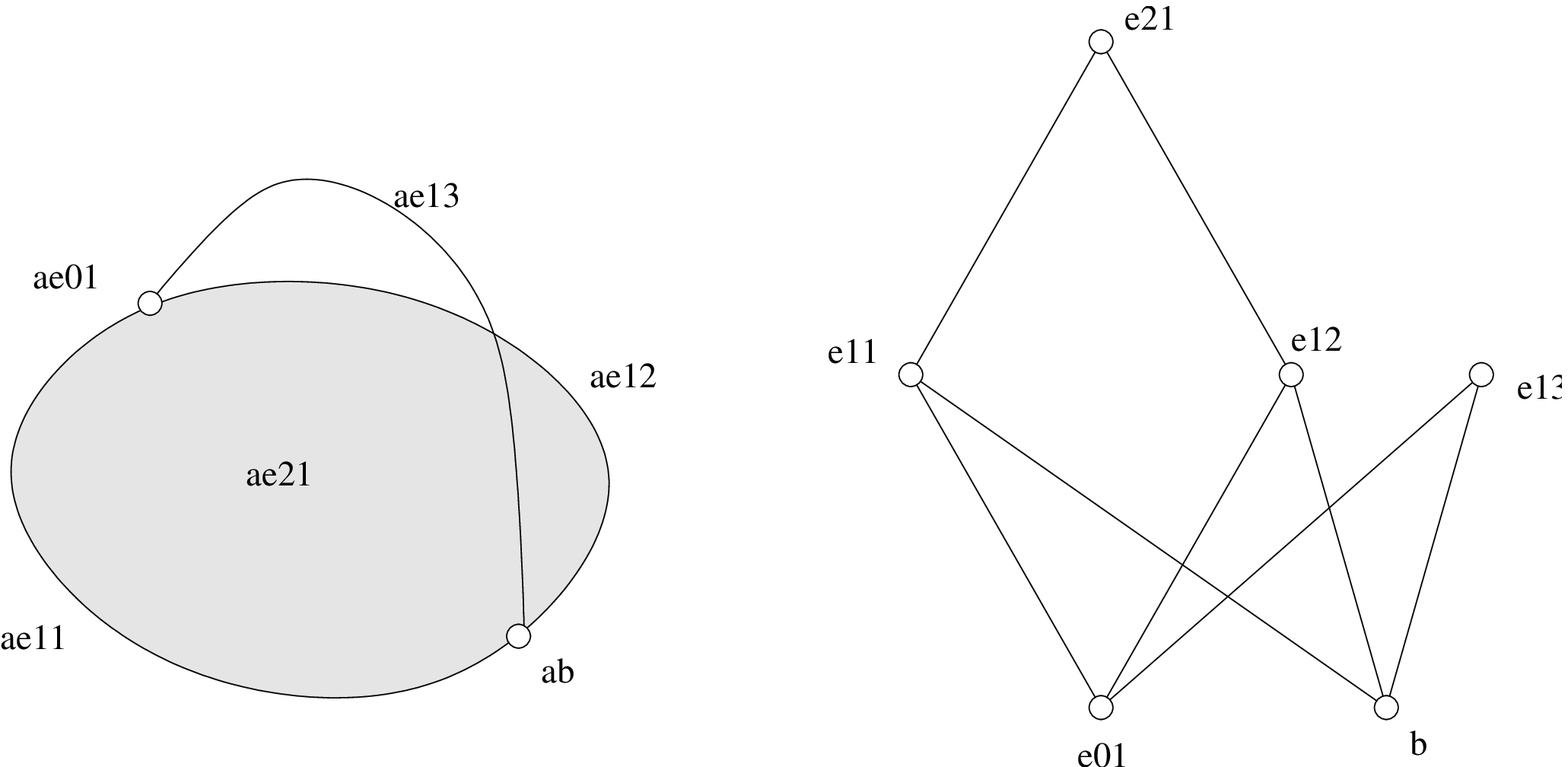}
\end{center}
\caption{A nice pointed CW complex and the poset corresponding to it}\label{fig:exampleposet}
\end{figure}

Let $X$ be a nice finite CW  complex. Let the dimension $k$ cells of
$X$ be $e^k_1,e^k_2,\ldots,e^k_{n_k}$. We define its dual in the
following way. We first fix a map $P$ from the discrete union of all cells to
$\mathbb{R}^2$, such that each cell maps to a single point in
$\mathbb{R}^2$, and different cells map to different points in
$\mathbb{R}^2$. Let the image of the cell $e^k_i$ be $p^k_i$; let
$\mu^k_i:[0,1]\rightarrow \mathbb{R}^2$ be the straight line path from
the origin to $p^k_i$ at constant speed, and let $g^k_i\subset
\mathbb{R}^2\times\mathbb{R}$ be the graph of the function $\mu^k_i$.

Given such a map $P$, we will construct a PL-embedding $f_P$ of $X$ in
$\mathbb{R}^n$, with $n\geq 3d+1$, where $d$ is the dimension of
$X$. We will embed the $(k-1)$-skeleton $X^{k-1}$ in $\mathbb{R}^{3k-2}$, and
then view $\mathbb{R}^{3k-2}$ as the subspace
$\mathbb{R}^{3k-2}\times\{0\}^3$ in $\mathbb{R}^{3k+1}=\mathbb{R}^{3k-2}\times\mathbb{R}^3$
and extend this embedding to the $k$-skeleton. Thus we will be
able to embed $X$ in $\mathbb{R}^{3d+1}$ which we view as the
subspace $\mathbb{R}^{3d+1}\times\{0\}^{n-3d-1}$ in $\mathbb{R}^n=\mathbb{R}^{3d+1}\times\mathbb{R}^{n-3d-1}$.

\begin{thm} Given a map $P$, there is a PL-embedding $f_P$ of the type as
  described in the previous paragraph.
\end{thm}

\begin{proof} 
For clarity, we explicitly write down the embedding of $X^k$ for a few
small values of $k$. We embed the $0$-skeleton in $\mathbb{R}$ by mapping the basepoint to
the origin and the $0$-cell to $1$. 

\begin{figure}[ht]
\psfrag{ab}[][][0.75]{$b$}
\psfrag{ae01}[][][0.75]{$e^0_1$}
\psfrag{ae11}[][][0.75]{$e^1_1$}
\psfrag{ae12}[][][0.75]{$e^1_2$}
\psfrag{ae13}[][][0.75]{$e^1_3$}
\psfrag{a1}{The embedding of the $1$-skeleton}
\psfrag{b1}{The disk $d^2_1$ is constructed as an}
\psfrag{b2}{embedding of the order complex of $[e^0_1,e^2_1]$}
\psfrag{be01}[][][0.75]{$e^0_1$}
\psfrag{be11}[][][0.75]{$e^1_1$}
\psfrag{be12}[][][0.75]{$e^1_2$}
\psfrag{be21}[][][0.75]{$e^2_1$}
\psfrag{be01e21}[][][0.75]{$e^0_1\prec e^2_1$}
\psfrag{be01e11}[][][0.75]{$e^0_1\prec e^1_1$}
\psfrag{be01e12}[][][0.75]{$e^0_1\prec e^1_2$}
\psfrag{be11e21}[][][0.75]{$e^1_1\prec e^2_1$}
\psfrag{be12e21}[][][0.75]{$e^1_2\prec e^2_1$}
\psfrag{be01e11e21}[][][0.75]{$e^0_1\prec e^1_1\prec e^2_1$}
\psfrag{be01e12e21}[][][0.75]{$e^0_1\prec e^1_2\prec e^2_1$}
\psfrag{cb}[][][0.75]{$b$}
\psfrag{ce01}[][][0.75]{$e^0_1$}
\psfrag{ce11}[][][0.75]{$e^1_1$}
\psfrag{ce12}[][][0.75]{$e^1_2$}
\psfrag{ce13}[][][0.75]{$e^1_3$}
\psfrag{ce21}[][][0.75]{$e^2_1$}
\psfrag{c1}{The embedding of the $2$-skeleton}
\psfrag{c2}{with $e^k_i$ embedded as a $k$-cell}
\psfrag{db}{$b$}
\psfrag{d1}{The same embedding as previous figure}
\psfrag{d2}{viewed as an embedding of the order complex}
\psfrag{de01}[][][0.75]{$e^0_1$}
\psfrag{de11}[][][0.75]{$e^1_1$}
\psfrag{de12}[][][0.75]{$e^1_2$}
\psfrag{de13}[][][0.75]{$e^1_3$}
\psfrag{de21}[][][0.75]{$e^2_1$}
\psfrag{de01e21}[][][0.75]{$e^0_1\prec e^2_1$}
\psfrag{de01e11}[][][0.75]{$e^0_1\prec e^1_1$}
\psfrag{de01e12}[][][0.75]{$e^0_1\prec e^1_2$}
\psfrag{de01e13}[][][0.75]{$e^0_1\prec e^1_3$}
\psfrag{de11e21}[][][0.75]{$e^1_1\prec e^2_1$}
\psfrag{de12e21}[][][0.75]{$e^1_2\prec e^2_1$}
\psfrag{de01e11e21}[r][][0.75]{$e^0_1\prec e^1_1\prec e^2_1$}
\psfrag{de01e12e21}[][][0.75]{$e^0_1\prec e^1_2\prec e^2_1$}
\begin{center} \includegraphics[width=0.9\textwidth]{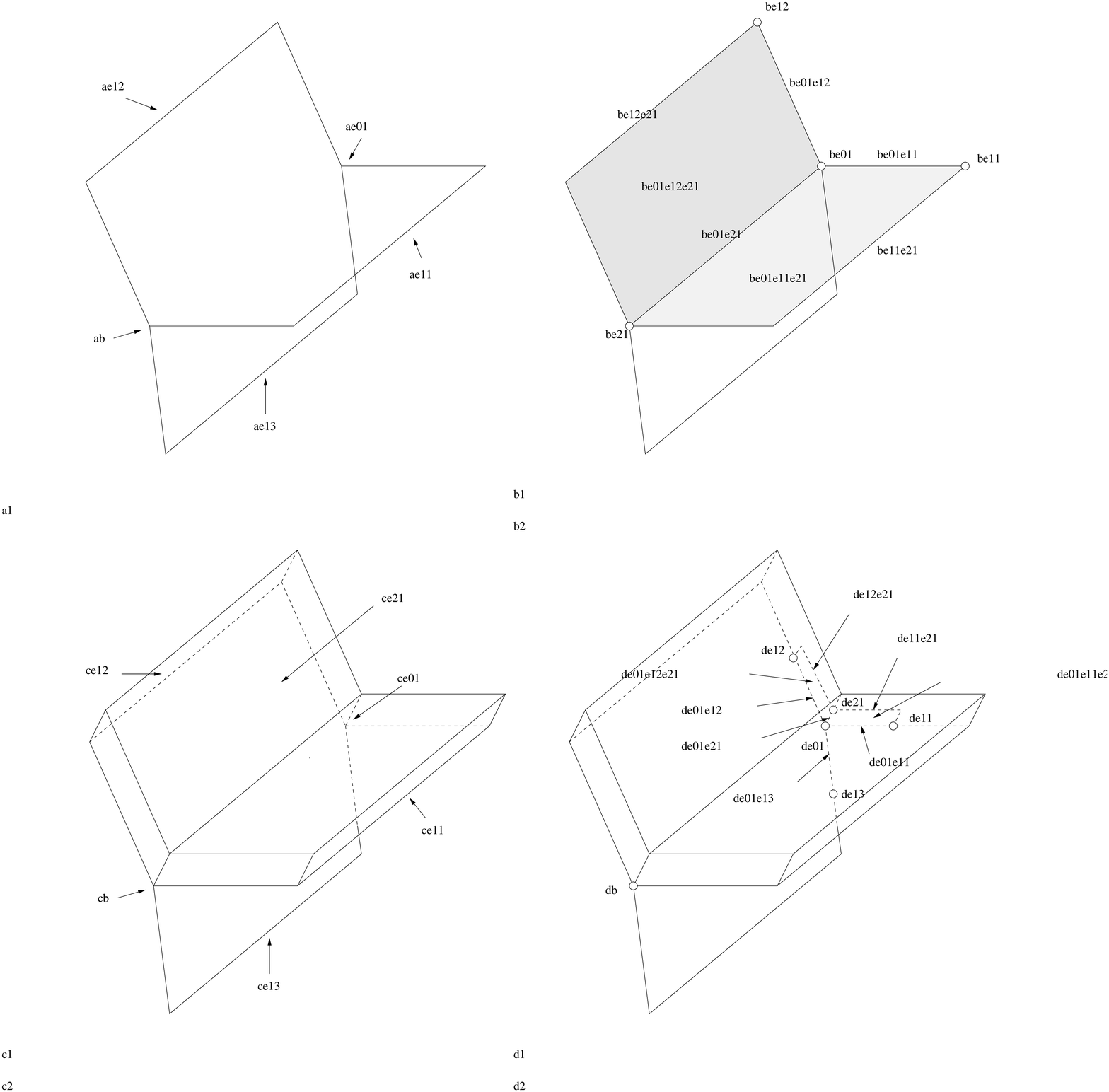}
\end{center}
\caption{Extending embedding of the $1$-skeleton of the CW complex of Figure
  \ref{fig:exampleposet} to the $2$-skeleton. This happens in
  $\mathbb{R}^7$, but the ambient space has been flattened out.}
\end{figure}

The $1$-cells are
$e^1_1,e^1_2,\ldots,e^1_{k_1}$. We view $\mathbb{R}^4$ as
$\mathbb{R}\times\mathbb{R}^3$, and embed $e^1_i$ as an union of
$\partial e^1_i \times \{\mu^1_i(t)\}\times \{t\}$ for $t\in[0,1]$,
and $[0,1]\times \{p^1_i\}\times\{1\}$. Note that since $p^1_i$'s are
distinct, this is indeed an embedding.

There is a different way of viewing the above process. For each
$1$-cell $e^1_i$, its boundary is a $0$-sphere $s^0_i$ in $\mathbb{R}$, and bounds a
disk $d^1_i$ in $\mathbb{R}$ (which in our case always happens to be
the unit interval $I$). We then embed the $1$-cell $e^1_i$ as a union
of an annulus $S^0\times I$ embedded in $\mathbb{R}\times\mathbb{R}^3$
as $s^0_i \times g^1_i$, and a
disk $D^1$ embedded in $\mathbb{R}\times \mathbb{R}^3$ as
$d^1_i\times\{p^1_i\}\times\{1\}$.

Now to embed $X^k$ in $\mathbb{R}^{3k+1}$, we proceed inductively. We
assume $X^{k-1}$ is already embedded in $\mathbb{R}^{3k-2}$, and we
view $\mathbb{R}^{3k+1}=\mathbb{R}^{3k-2}\times\mathbb{R}^3$. For
each $k$-cell $e^k_i$, its boundary is a $(k-1)$-sphere $s^{k-1}_i$
embedded in $\mathbb{R}^{3k-2}$. If that sphere $s^{k-1}_i$ bounds a
disk $d^k_i$ in $\mathbb{R}^{3k-2}$, then we can embed the $k$-cell
$e^k_i$ as an union of an annulus $S^{k-1}\times I$ embedded in
$\mathbb{R}^{3k+1}$ as $s^{k-1}_i\times g^k_i$, and a disk $D^k$ embedded in $\mathbb{R}^{3k+1}$ as
$d^k_i\times\{p^1_i\}\times\{1\}$. Note that since $p^k_i$'s are
distinct points in $\mathbb{R}^2$, this is still an embedding.

Thus to show that there is a well-defined embedding depending only on
the choice of the map $P$, we need to produce a disk $d^k_i$ bounding
$s^{k-1}_i$, which does not depend on anything other than the choice of
the map $P$. Without loss of generality let $i=1$. Let $s^{k-1}_1$ be the boundary of a $k$-cell
$e^k_1$. Note that the order complex of $[e^0_1,e^k_1]$ is a disk of
the same dimension as $d^k_1$. So we will produce an embedding of this
order complex with the proper boundary.

To present a clearer picture, let us explicitly describe how we define
the embeddings of the vertices and edges of this order complex.
We embed $e^0_1$ and $e^k_1$ as
$\{1\}\times\{0\}^{3k-3}$ and $\{0\}^{3k-2}$ respectively. For
$1\leq l\leq k-1$, we embed $e^l_i$ as
$\{1\}\times\{0\}^{3l-3}\times\{p^l_i\}\times\{1\}\times\{0\}^{3k-3l-3}$.
The edge joining $e^k_1$ to $e^0_1$ is $I\times\{0\}^{3k-3}$, the
edge joining $e^l_i$ to $e^0_1$ is
$\{1\}\times\{0\}^{3l-3}\times\ g^l_i \times\{0\}^{3k-3l-3}$, the edge joining
$e^k_1$ (resp. $e^{l'}_{i'}$) to $e^l_i$ is $\{0\}^{3l-2}\times
g^l_i\times\{0\}^{3k-3l-3}$ followed by
$I\times\{0\}^{3l-3}\times\{p^l_i\}\times\{1\}\times\{0\}^{3k-3l-3}$
(resp. $\{1\}\times\{0\}^{3l-3}\times
g^l_i\times\{0\}^{3l'-3l-3}\times\{p^{l'}_{i'}\}\times\{1\}\times\{0\}^{3k-3l'-3}$ followed by
$\{1\}\times\{0\}^{3l-3}\times\{p^l_i\}\times\{1\}\times\{0\}^{3l'-3l-3}\times
g^{l'}_{i'}\times\{0\}^{3k-3l'-3}$).

Now let us describe in general how a simplex of this order complex
coming from a chain $e^0_1\prec e^{l_1}_{i_1}\prec\cdots
\prec e^{l_m}_{i_m}$ is embedded in $\mathbb{R}^{3k-2}$ with $l_m <k$. We
embed this as the disk $\{1\}\times\{0\}^{3l_1-3}\times
g^{l_1}_{i_1}\times \{0\}^{3l_2-3l_1-3}\times g^{l_2}_{i_2}\times\cdots
\times g^{l_m}_{i_m}\times\{0\}^{3k-3l_m-3}$. For the rest of this
paragraph, let us call this subspace as $\{1\}\times A$, where $A$ is
a subspace of $\mathbb{R}^{3k-3}$. The simplex of the order complex coming from the chain $e^{l_1}_{i_1}\prec\cdots
\prec e^{l_m}_{i_m}$ is a suitable part of the boundary of the above
order complex, and again for the rest of this paragraph, let us denote
that subspace to be $\{1\}\times B$, where $B$ is a subspace of
$\partial A$. Then the simplex of the order complex coming from a chain $e^{l_1}_{i_1}\prec\cdots
\prec e^{l_m}_{i_m}\prec e^k_1$ is embedded as the union of
$\{0\}\times A$ followed by $I\times B$ and the simplex of the order
complex coming from a chain $e^0_1\prec e^{l_1}_{i_1}\prec\cdots\prec
e^{l_m}_{i_m}\prec e^k_1$ is embedded as $I\times A$.

Thus we have embedded the order complex of $[e^0_1,e^k_1]$ in
$\mathbb{R}^{3k-2}$, and this is the required disk $d^k_1$ bounding
$s^k_1$. Using such disks $d^k_i$'s, we can then embed $X^k$ in
$\mathbb{R}^{3k+1}$, thus completing the proof.
\end{proof}

There are a few observations that we should make now. The only choice
we made in defining the embedding is the choice of the map $P$. But we
can connect any two such maps $P$ and $P'$ by an isotopy of $\mb{R}^2$, and this
induces an isotopy in $\mathbb{R}^n$ connecting the embeddings $f_P$
and $f_{P'}$.

Furthermore, this embedding is also an embedding of the order complex
of the whole poset coming from the CW complex. The basepoint is
embedded as the origin, the $0$-cell is embedded as
$\{1\}\times\{0\}^{n-1}$, and the vertex corresponding to the cell
$e^k_i$ is embedded as
$\{1\}\times\{0\}^{3k-3}\times\{\frac{1}{2} p^k_i\}\times\{\frac{1}{2}\}\times
\{0\}^{n-3k-1}$. A simplex of this order complex
coming from a chain $e^0_1\prec e^{l_1}_{i_1}\prec\cdots
\prec e^{l_m}_{i_m}$ is embedded in $\mathbb{R}^n$, as the disk $\{1\}\times\{0\}^{3l_1-3}\times
\frac{1}{2} g^{l_1}_{i_1}\times \{0\}^{3l_2-3l_1-3}\times \frac{1}{2} g^{l_2}_{i_2}\times\cdots
\times \frac{1}{2} g^{l_m}_{i_m}\times\{0\}^{n-3l_m-1}$. Once more for the rest of this
paragraph, let us call this subspace as $\{1\}\times A$, where $A$ is
a subspace of $\mathbb{R}^{n-1}$. The simplex of the order complex
coming from the chain $e^{l_1}_{i_1}\prec\cdots
\prec e^{l_m}_{i_m}$ is a suitable part of the boundary of the above
order complex, and again for the rest of this paragraph, let us denote
that subspace to be $\{1\}\times B$, where $B$ is a subspace of
$\partial A$. Then the simplex of the order complex coming from a
chain $b\prec e^{l_1}_{i_1}\prec\cdots
\prec e^{l_m}_{i_m}$, where $b$ is the basepoint, is embedded as the
union of the closure of $(\{1\}\times 2A)\setminus (\{1\}\times A)$, followed by $I\times 2B$ followed
by $\{0\}\times 2A$.
Note that this embedding of the order complex is slightly different
from the one that we used in the previous proof.

Thus the closure of a regular neighborhood  of $X$ in  $\mathbb{R}^n$ will  give an
$n$-dimensional manifold  $N$ (with boundary) with  same homotopy type
as that  of $X$. We construct $N$ in the following way. Let $N_k$ be
the set of all points with $L^2$ distance less than or equal to $\epsilon_k$ from
$X^k$. We assume $\epsilon_k$'s are decreasing in $k$ and we choose positive $\epsilon_0$ to be small
enough such that the interior of $(\cup N_k)$ is a regular
neighborhood of $X$. For each $k>1$ (resp. $k=1$), after we have already chosen $\epsilon_0,\ldots,\epsilon_{k-1}$ we
choose positive $\epsilon_k$ to be sufficiently small such that
$N_k\cap\partial (\cup_{j=0}^{k-1}N_j)$ has exactly one component (resp. exactly two
components) for each $k$-cell $e^{k}_i$. We define $N=\cup_i
N_i$. Note that $\partial N$ is not a smooth manifold.

Let  $b$ be  the image of  the basepoint $X^{-1}$  in the
embedding, and  let $B$ be the small neighborhood of $b$, lying in
the interior of $N_0$. Let  us view $W=N\setminus B$ as a  cobordism from
$\partial  B$ to  $\partial N$. Note that this cobordism is obtained
by starting with $\partial B$, adding disks corresponding to the
embeddings of the order complexes of $(-\infty,e^k_i]$, and then
taking a regular neighborhood. Now let us assume that there is a Morse
function and a corresponding gradient-like flow for this cobordism,
such that the flow is transverse to $\partial N$ and $\partial B$,
the only index $k$ critical points are the images of the vertices in the order
complex  corresponding to $e^k_i$ and the left-handed disks are the
embeddings of the order complexes corresponding to
$(-\infty,e^k_i]$. Then the original pointed CW complex
$X$ can be recovered from this gradient-like flow in the following
way. Quotient out $\partial B$ to the basepoint, and the cells for the
CW complex are the left-handed disks with the attaching map being
given by the flow. We construct the dual of $X$ in a very similar
way. We look at the right-handed disks, and regard the cobordism as
obtained from  $\partial N$ by adding those disks and then taking a
regular neighborhood. Thus to construct the pointed CW complex dual to $X$,
we should  quotient out $\partial N$ to the basepoint, and have cells
correspond to right-handed disks with attaching maps given by the
flow. However to define the dual in this way, we first need to find a
Morse function and a corresponding gradient-like flow satisfying the
above conditions. The dual then might depend on the choice of the
Morse function and the gradient-like flow and also on the map $P$. We
will bypass  the construction of the Morse function and the
gradient-like flow, and define the right-handed disks directly,
depending only on the choice of the map $P$.

We will define the right-handed disk $r^k_i$ corresponding to the
critical point coming from the vertex $e^k_i$ of the order complex in
several stages. Recall that the regular neighborhood $N$ is
constructed as a union $\cup_j N_j$. Let $r^k_{i,j}=N_j\cap r^k_i$. We
will define $r^k_{i,j}$ starting at $j=0$, then gradually
extending the definition to $j=1,2\ldots$, and finally define
$r^k_i=\cup_j r^k_{i,j}$.

Furthermore, note that $r^k_{i,j}=\varnothing$ for $j<k$. So for $j=0$,
we only need to define $r^0_{1,0}$. We define $r^0_{1,0}$ as the connected
component of $N_0$ not containing $\partial B$. For $j=1$, define
$r^1_{i,1}$ as the intersection of $N_1$ with the hyperplane
$\mathbb{R}^3\times\{\frac{1}{2}\}\times\mathbb{R}^{n-4}$ and extend
$r^0_{1,0}$ to $r^0_{1,1}$ as the set of all points in $N_1$ whose
$L^{\infty}$ distance from $e^0_1$ (embedded as
$\{1\}\times\{0\}^{n-1}$) is at most $\frac{1}{2}$. It is easy to see
that $\partial r^0_{1,1}$ lies in the union of $\partial N_1$ and
$r^1_{i,1}$ and each right-handed disk is still a ball of the correct
dimension.  The way we extended the definition of $r^0_{1,0}$ to that
of $r^0_{1,1}$ can also be described as follows. Since $r^0_{1,0}$ is one of
the components of $N_0$, $N_1\cap \partial r^0_{1,0}$ is a disjoint
union of $(n-1)$-dimensional balls, one for each $e^1_i$. We then take
the ball corresponding to $e^1_i$ and extend it like a horn in the direction of
$e^1_i$ until we reach the vertex corresponding to $e^1_i$. Since
different balls on $\partial r^0_{1,0}$ corresponding to different
$e^1_i$'s are disjoint, after extending the horns, $r^0_{1,1}$ is
still a ball of dimension $n$. Suitable parts of $\partial r^0_{1,1}$
are defined as $r^1_{i,1}$.

Now to define $r^k_{i,j}$, by induction, let us assume, we have
defined $r^k_{i,j'}$ for all $j'<j$. We define $r^j_{i,j}$ as the
intersection of $N_j$ with the plane
$\mathbb{R}\times(\mathbb{R}^2\times\{0\})^{j-1}\times\mathbb{R}^2\times\{\frac{1}{2}\}\times\mathbb{R}^{n-3j-1}$.
For $k<j$, by induction $r^k_{i,j-1}$ is already
defined. $N_j\cap\partial r^k_{i,j-1}$ is a disjoint union of
$(n-k-1)$-dimensional balls, one for each $e^j_{i'}$ with $e^k_i\prec
e^j_{i'}$. We extend the ball corresponding to $e^j_{i'}$ in the
direction given by embedding of the order complex of
$[e^k_i,e^j_{i'}]$ until we reach the boundary of the order
complex. We define $r^k_{i,j}$ as $r^k_{i,j-1}$ after these
extensions. Since $r^k_{i,j-1}$ was a $(n-k)$-dimensional ball, and
we extended along disks starting at different portions of $\partial
r^k_{i,j-1}$, $r^k_{i,j}$ is still a ball of the correct
dimension. Note that $r^0_{1,j}$ is still the set of all points in
$N_j$ whose $L^{\infty}$ distance from $e^0_1$ is at most
$\frac{1}{2}$, and thus it is particularly easy to see that
$r^0_{1,j}$ is an $n$-dimensional ball, since $N_j$ is an
$n$-dimensional manifold. Finally, we define $r^k_i=\cup_j
r^k_{i,j}$. Note that $r^{k'}_{i'}$ lies in the boundary of $\partial
r^k_i$ if and only if $e^k_i\prec e^{k'}_{i'}$ and in that case, it is
actually embedded. 

\begin{figure}[ht]
\psfrag{r010}{$r^0_{1,0}$}
\psfrag{r011}{$r^0_{1,1}$}
\psfrag{r121}{$r^1_{2,1}$}
\psfrag{r111}{$r^1_{1,1}$}
\psfrag{r131}{$r^1_{3,1}$}
\psfrag{r01}{$r^0_1$}
\psfrag{r12}{$r^1_2$}
\psfrag{r11}{$r^1_1$}
\psfrag{r13}{$r^1_3$}
\psfrag{r21}{$r^2_1$}
\psfrag{e01}{$e^0_1$}
\psfrag{e11}{$e^1_1$}
\psfrag{e12}{$e^1_2$}
\psfrag{e13}{$e^1_3$}
\psfrag{e21}{$e^2_1$}
\begin{center}
\includegraphics[width=200pt]{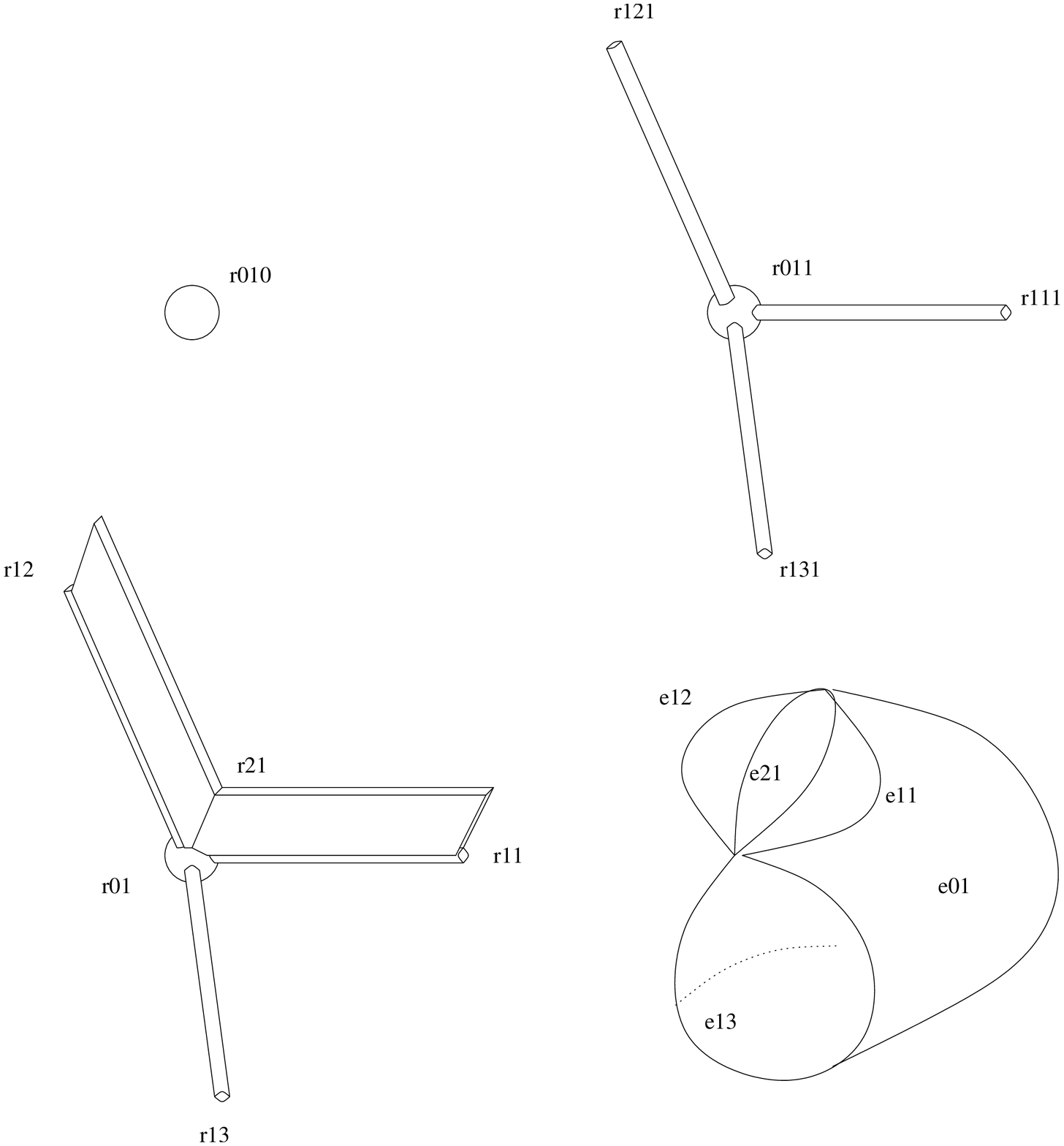}
\end{center}
\caption{The dual of the CW complex from Figure \ref{fig:exampleposet}}
\end{figure}

We then quotient out $\partial N$ to the basepoint, and define the
dual CW complex using the right-handed disks as described above. Note
that the union of all the right-handed disks is simply $r^0_1$, and
thus if $\widetilde{X_n}=(r^0_1,\partial N\cap r^0_1)$, we can also
construct the dual by starting with $\widetilde{X_n}$ and then
quotienting out $\partial N\cap r^0_1$ to the basepoint. This
construction might a priori depend on the map $P$, but we can connect
any two such maps $P$ and $P'$ by an isotopy of $\mathbb{R}^2$. During
the isotopy, for each $k>1$ (resp. $k=1$) we can make the
$\epsilon_k$'s used in the definition of $N_k$'s sufficiently small
such that the condition about $N_k\cap\partial (\cup_{j=0}^{k-1}N_j)$
having exactly one component (resp. exactly two components) for each
$k$-cell $e^{k}_i$ holds. Then this induces an isotopy joining the two
$\widetilde{X_n}$'s, and hence induces a homeomorphism between the two
duals.  Thus the dual of a nice pointed CW complex $X$ does not depend
on the map $P$ and depends only on the ambient dimension $n$. Let us
denote this dual by $\overline{X}_n$.

Before we prove any other properties of the dual, we need to
understand the dependence of $\overline{X}_n$ on $n$. The following
result makes this precise.

\begin{thm}\label{thm:smash}
For a nice pointed CW complex $X$, we have
$\overline{X}_{n+1}=\overline{X}_n\wedge S^1$, where $\wedge$ denotes
the smash product.
\end{thm}

\begin{proof}
After fixing a map $P$, we can construct an embedding of $X$ in
$\mathbb{R}^n$ in a well-defined way, and we extend this embedding to
an embedding into $\mathbb{R}^{n+1}$ by embedding  $\mathbb{R}^n$ into
$\mathbb{R}^{n+1}$ as $\mathbb{R}^n\times\{0\}$. After fixing an
embedding to $\mathbb{R}^m$, we define $\widetilde{X_m}$ as a pair
$(A_m,B_m)$ with $B_m$ lying in $\partial A_m$, and we define
$\overline{X}_m$ as a quotient of $\widetilde{X_m}$ obtained by
quotienting out $B_m$ to the basepoint.

However $A_{n+1}$ is homeomorphic to $A_n\times[-\epsilon,\epsilon]$
and $B_{n+1}$ is homeomorphic to $(A_n\times\{\pm
\epsilon\})\cup(B_n\times[-\epsilon,\epsilon])$. Since
$[-\epsilon,\epsilon]/\{\pm \epsilon\}$ is the circle $S^1$, hence
$\overline{X}_{n+1}=A_{n+1}/B_{n+1}=(A_n/B_n)\wedge
S^1=\overline{X}_n\wedge S^1$.
\end{proof}

Now we are in a position to state and prove the following important
properties of duals. Let $X$ be a nice pointed CW complex, and let $Y$
be a subcomplex. $Y$ is also clearly nice and pointed. We can thus
define the duals $\overline{X}_n$ and $\overline{Y}_n$ for $n$
sufficiently large (in fact $n$ simply has to be larger than $3d$
where $d$ is the dimension of $X$). Then the following holds,

\begin{thm}\label{thm:quotient}For $Y$ a subcomplex of a nice CW complex $X$, the dual
  $\overline{Y}_n$ can be obtained from $\overline{X}_n$ by
  quotienting out the cells corresponding to the cells in $Y$ that are
  not in $X$.
\end{thm} 

\begin{proof}
Note that it is enough to prove the case when there is exactly one
cell $e^k_1$ that is in $Y$ but not in $X$. Thus to embed $Y$
in $\mathbb{R}^n$, we embed $X$ in $\mathbb{R}^n$ and then delete the
cell $e^k_1$ (which was embedded as an embedding of
the order complex of $(-\infty,e^k_1]$). Another way to see this is the
following. Take the embedding of $X$, view it as an embedding of the
order complex, and delete the vertex corresponding to $e^k_1$. Then
the new space deform retracts to the embedding of $Y$. Let $M$ and $N$ be regular
neighborhoods of $X$ and $Y$ respectively, as defined earlier in this
section. Let $R$ be a small neighborhood of the right-handed disk
$r^k_1$ of $e^k_1$ in the embedding of $X$. Then $N\setminus\mathring{R}$ deform retracts to $M$. 

The right-handed disks required for defining the dual $\overline{Y}_n$
come from the manifold $M$. The right-handed disks required for
defining the dual $\overline{X}_n$ come from the manifold $N$, and
when these right-handed disks are restricted to
$N\setminus\mathring{R}$, they define the quotient complex of
$\overline{X}_n$ obtained by quotienting out $r^k_1$, the cell
corresponding to $e^k_1$. A properly chosen deformation retract of
$N\setminus \mathring{R}$ to $M$ gives the required homeomorphism
between this quotient complex and $\overline{Y}_n$.
\end{proof}

A very similar property holds for quotient complexes. However if $X$
is a nice pointed CW complex, quotient complexes of $X$ in general
will not be nice. Let $e^1_1$ be an $1$-cell of $X$, and consider the
quotient complex $Z$ of $X$ obtained by keeping only the cells $e^k_i$
with $e^1_1\preceq e^k_i$ and quotienting out everything else. Let us
assume that there is a nice pointed CW complex $Y$, such that $Y\wedge
S^1$ with the natural CW complex structure is the same CW complex as
$Z$ (in fact using Theorem \ref{thm:construction1}, we can always
assume this). If $d$ is the dimension of $X$, then for $n>3d$, we can
define the duals $\overline{X}_n$ and $\overline{Y}_n$. Then the
following is true.

\begin{thm}\label{thm:sub}
The dual $\overline{Y}_{n-1}$ is homeomorphic to the subcomplex of
$\overline{X}_n$ obtained by considering only the 
cells corresponding to the ones present in $Z$. 
\end{thm}

\begin{proof}
First observe that the order complex of the poset coming from $X$ restricted to the cells
of $Z$, can also be obtained from the order complex of $Y$ by removing
the element corresponding to the basepoint. Now choose an
embedding of $X$ (which is also an embedding of the order complex
of the poset coming from $X$) to $\mathbb{R}^n$. Let us restrict to the order
complex of $Z\cup\{b\}$, where $b$ is the basepoint in $X$, and delete all
the simplices which use the edge coming from $b\leftarrow e^1_1$. This
is same as the order complex coming from $Y$. Thus an embedding of
$X$ in $\mathbb{R}^n$ gives an embedding of this order complex in
$\mathbb{R}^n$. We will now modify this embedding such that it agrees
with a standard embedding of $Y$ in $\mathbb{R}^{n-1}$. Observing how
the right-handed disks change under this modification will complete
the proof.

At time $t$ for $t\in[0,1]$, $e^1_1$ is embedded as
$\{1\}\times\{\frac{1}{2}p^1_1\}\times\{\frac{1}{2}\}\times\{0\}^{n-4}$,
the basepoint $b$ is embedded as
$\{0\}\times\{\frac{t}{2}p^1_1\}\times\{\frac{t}{2}\}\times\{0\}^{n-4}$
and a vertex $e^k_i$ for $k>1$ is embedded as
$\{1\}\times\{\frac{t}{2}p^1_1\}\times\{\frac{t}{2}\}\times\{0\}^{3k-6}\times\frac{1}{2}g^k_i
\times\{0\}^{n-3k-1}$. The simplex coming from a chain that does not
involve $e^1_1$ is a shifted version of the original, with the second, third and
fourth coordinate being changed from $\{0\}^3$ to
$\{\frac{t}{2}p^1_1\}\times\{\frac{t}{2}\}$. The simplex coming from a
chain that involves $e^1_1$ is a truncated version of the original,
where we delete the part that intersects with
$\mathbb{R}\times\{\frac{t}{2}p^1_1\}\times\{\frac{t}{2}\}\times\mathbb{R}^{n-4}$.
Note that at $t=0$, this is an embedding of the order complex of $Y$
as induced from an embedding of $X$. At $t=1$, this is the standard embedding of
the order complex of $Y$ in
$\mathbb{R}^3\times\{\frac{1}{2}\}\times\mathbb{R}^{n-4}=\mathbb{R}^{n-1}$. To complete
the proof, we should observe how the right-handed disks change during
this isotopy. At time $t$, we can define the right-handed disk of
$e^1_1$ as a truncated version of the original right-handed disk by
deleting the part that intersects with
$\mathbb{R}\times\{\frac{t}{2}p^1_1\}\times\{\frac{t}{2}\}\times\mathbb{R}^{n-4}$
and the right-handed disk of $e^k_i$ for $k>1$ as a shifted version of
the original right-handed disk with the second, third and fourth
coordinate shifted from $\{0\}^3$ to
$\{\frac{t}{2}p^1_1\}\times\{\frac{t}{2}\}$. This gives an explicit isotopy
connecting the subcomplex of $\overline{X}_n$ coming from the cells
corresponding to those in $Z$, to $\overline{Y}_{n-1}$
\end{proof}

Thus given a GSS poset with one minimum, by Theorem
\ref{thm:construction1}, we can construct a nice CW complex
corresponding to the poset, and then construct its dual. We can assign
an orientation to the top dimensional cell (the one corresponding to
the unique minimum in the poset) arbitrarily, but once that is fixed
the orientation of the rest of the cells is determined by the sign
convention on the GSS poset. This extra information coming from the
orientation of the top-dimensional cell allows us to strengthen
Theorem \ref{thm:sub}. In that theorem, we showed that there is an
isomorphism between $\overline{Y}_{n-1}$ and
a subcomplex of $\overline{X}_n$, but there might be more than one
such isomorphism. However after we orient the top-dimensional cells in
both $\overline{X}_n$ and $\overline{Y}_{n-1}$ (and hence using the
sign convention on the poset of $X$, orient every cell in these two CW
complexes), we choose the isomorphism that matches the
orientations. Thus for oriented CW complexes, there is a well-defined
isomorphism between $\overline{Y}_{n-1}$ and a subcomplex of
$\overline{X}_n$. This will be of use to us in Section \ref{sec:gridhomotopy}.

Before concluding this section, we should note that our explicit
construction of a dual actually agrees with the Alexander dual, which
is obtained by embedding the space $X$ in the sphere $S^n$, and  then
taking the homotopy type of the complement. Thus the Alexander dual is
homotopic to $A(X)=S^n\setminus N$. The way to see this is as
follows. Let us embed $X$ as described above into $\mathbb{R}^n$ and
let $S^n$ be viewed as the one point compactification of that
$\mathbb{R}^n$ with that extra point being denoted by $*$. Let
$\overline{b}$ be the basepoint in the dual $\overline{X}_n$ and let
$A(X)=S^n\setminus N$ be the Alexander dual. If $\sim$
denotes the homotopy equivalence of pairs of spaces, we have

$$(\overline{X}_n, \overline{b})\sim(N\setminus B,\partial
N)\sim(S^n\setminus B, (S^n\setminus N))\sim(S^n\setminus\{b\},A(X))$$

However we have an exact sequence of spaces

$$\xymatrix{(A(X),*)\ar@{^{(}->}[r] &(S^n\setminus\{b\},*)\ar@{->>}[r]
  &(S^n\setminus\{b\},A(X))\sim(\overline{X}_n,\overline{b})}$$

This induces the Puppe map from  $\overline{X}_n$ to $A(X)\wedge S^1$,
and since $H_*(S^n\setminus\{b\},*)=0$, the map induces isomorphism in
$H_*$ and hence induces a homotopy equivalence.

\section{Grid homotopy}\label{sec:gridhomotopy}

Let $P$ be a GSS poset. For most of the time, $P$ will be a grid
poset, a commutation poset or a stabilization poset.

If we take the poset $P$, and reverse the partial order, observe that
each closed interval in the new poset is still shellable. This follows
from the definition of shellability.  Thus the applications of Section
\ref{sec:applications} can all be constructed. In particular, we can
construct a pointed CW complex corresponding to the interval
$(-\infty,x]$ in the old poset, whose cells are elements of
$(-\infty,x]$ and the attaching maps correspond to the coboundary maps
in the chain complex induced from the poset.

\begin{thm} \label{thm:construction}There is a well-defined nice CW complex $P_x$,
  such that the cells correspond to the elements of $(-\infty,x]$, the
  boundary maps correspond to the coboundary map of the chain complex
  induced from $(-\infty, x]$ and agrees with any given sign convention on
  it, the cell corresponding to $x$ has dimension $0$, and the
  boundary map every other cell is injective.
\end{thm}

\begin{proof}
We reverse the partial order of the poset $(-\infty,x]$ and construct
the pointed CW complex $S_x(0)$ as described in Theorem
\ref{thm:construction1}. This is the required pointed CW complex $P_x$.
\end{proof}

We now state and prove the main result of this section.

\begin{thm}
Given a GSS poset $P$, for sufficiently large $n$, there is a well-defined
CW complex $X_P(n)$ whose $k$-cells correspond to the elements of $P$
of grading $(k+n)$ and whose boundary maps correspond to the covering
relations in $P$ even up to sign.
\end{thm}

\begin{proof}
If $M_1$ and $M_2$ are the maximum and the minimum gradings in the poset,
then we choose $n>2M_1-3M_2$. For each $x\in P$, we construct $P_x$ as
in the Theorem \ref{thm:construction}. Each of these CW complexes is a
nice pointed CW complex, and hence we can construct their duals
$\overline{(P_x)}_{g(x)+n}$ where $g(x)$ is the grading of $x$. In each of these CW complexes, we orient
the top-dimensional cell arbitrarily, and that fixes an orientation of
every cell. For $y\prec x$, $(-\infty, y]$ is a subcomplex of
$(-\infty,x]$. A repeated application of Theorem \ref{thm:sub} allows us
to construct a well-defined injection of $\overline{(P_y)}_{g(y)+n}$ to
$\overline{(P_x)}_{g(x)+n}$ which matches the orientations. Thus we have
a space for each $x\in P$ and a map for each pair $x,y\in P$ with
$y\prec x$. We take the discrete union of all these spaces and glue
them together using these maps and call it $X_P(n)$. It is easy to see
that $X_P(n)$ is well-defined and satisfies the conditions of the theorem.
\end{proof}

However note that the same poset $P$ can carry two different
non-equivalent sign conventions. Figure \ref{fig:diffposets} demonstrates
that such posets can indeed give rise to different spaces. In the
diagram we have significantly reduced the dimensions of the spaces.

\begin{figure}[ht]
\psfrag{aa}{$a$}
\psfrag{ab}{$b$}
\psfrag{ax}{$x$}
\psfrag{ay}{$y$}
\psfrag{az}{$z$}
\psfrag{aw}{$w$}
\psfrag{ax1}{$x$}
\psfrag{ay1}{$y$}
\psfrag{az1}{$z$}
\psfrag{aw1}{$w$}
\psfrag{aon1}{$1$}
\psfrag{aon2}{$1$}
\psfrag{amon1}{$-1$}
\psfrag{amon2}{$-1$}
\psfrag{a1}{$=$ the torus}
\psfrag{a2}{with the circle $a$ pinched}
\psfrag{ba}{$a$}
\psfrag{bb}{$b$}
\psfrag{bx}{$x$}
\psfrag{by}{$y$}
\psfrag{bz}{$z$}
\psfrag{bw}{$w$}
\psfrag{bx1}{$x$}
\psfrag{by1}{$y$}
\psfrag{bz1}{$z$}
\psfrag{bw1}{$w$}
\psfrag{bon1}{$1$}
\psfrag{bon2}{$1$}
\psfrag{bmon1}{$-1$}
\psfrag{bmon2}{$-1$}
\psfrag{b1}{$=$ the Klein bottle}
\psfrag{b2}{with the circle $a$ pinched}
\begin{center}
\includegraphics[width=250pt]{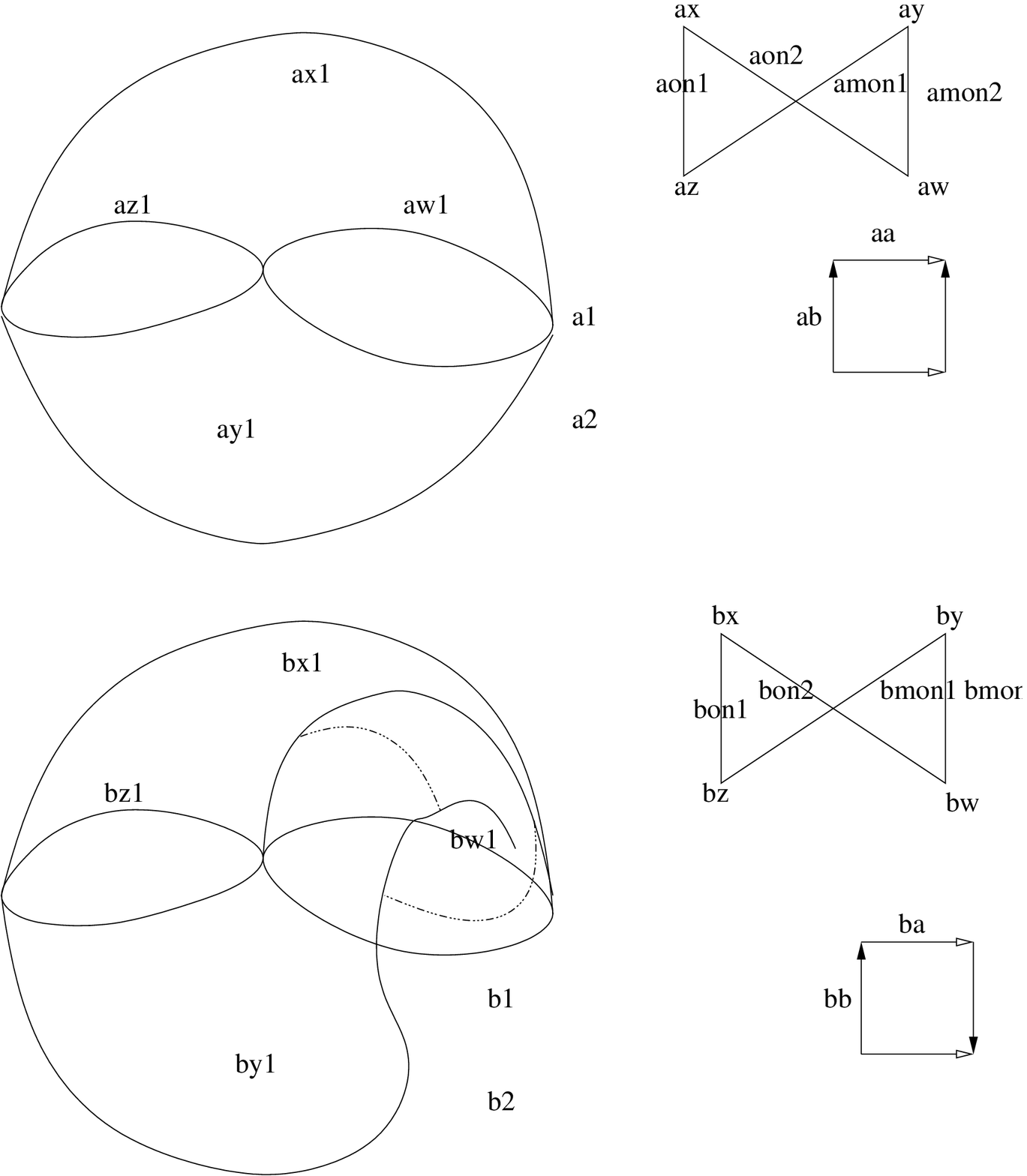}
\end{center}
\caption{Different spaces coming from the same poset}\label{fig:diffposets}
\end{figure}

Now we want to state and prove certain properties of this space
$X_P(n)$.

\begin{thm}\label{thm:subquotient}
If $P$ is a GSS poset, $Q$ is a subposet and $R$ is a quotient poset,
then for $n$ sufficiently large, the following are true.

$X_P(n+1)=X_P(n)\wedge S^1$.

$X_Q(n)$ is a subcomplex of $X_P(n)$ containing only the cells
corresponding to the elements in $Q$.

$X_R(n)$ is a quotient complex of $X_P(n)$ containing only the cells
corresponding to the elements in $R$.
\end{thm}

\begin{proof}
The space $X_P(n)$ is constructed as a union of spaces of the form
$\overline{(P_x)}_{g(x)+n}$, and the proof follows after observing that
each of these spaces has the three above mentioned properties as
proved in Theorems \ref{thm:smash}, \ref{thm:quotient} and \ref{thm:sub}.
\end{proof}

Thus by taking $P$ to be $\widehat{\mathcal{G}}$,
$\widehat{\mathcal{G}_m}$ or $\mathcal{G}_m^{-}$ (for any Alexander
grading $m$), and for $n$ 
sufficiently large, we can
construct CW complexes $X_P(n)$. In fact for $n$ sufficiently large,
$X_{\widehat{\mathcal{G}}}(n)= \vee_{m=-\infty}^{\infty}
X_{\widehat{\mathcal{G}_m}}(n)$, where $\vee$ is the wedge sum.

Since $X_P(n+1)=X_P(n)\wedge S^1$ we can associate
finite spectra $\mathcal{S}(P)$ to each GSS poset $P$, whose $n$-th
space is $X_P(n)$. The previous note implies that
$\mathcal{S}(\widehat{\mathcal{G}}) =\vee_m
\mathcal{S}(\widehat{\mathcal{G}_m})$. We can also define a spectrum
$\mathcal{S}(\mathcal{G}^{-})$ corresponding to $\mathcal{G}^{-}$ by defining it to be
$\vee_{m=-\infty}^{\infty} \mathcal{G}_m^{-}$. 

Now we want  to show that some of these objects that we associate to grid diagrams
of knots are actually knot invariants. First note  that any two grid diagrams for the
same   knot  are   related   by  a   sequence   of  commutations   and
stabilizations. We will consider each of the cases in great detail.

\subsection{Commutation} 

\begin{thm} If two grid diagrams $G$ and $G'$ differ by a commutation,
  then for any Alexander grading $m$, and with $n$ sufficiently large
  $X_{\widehat{\mathcal{G}_m}}(n)$ (resp.  $X_{\mathcal{G}_m^{-}}(n)$)
    and $X_{\widehat{\mathcal{G}_m^{\prime}}}(n)$ (resp.  $X_{(\mathcal{G}_m^{\prime})^{-}}(n)$) are homotopic.
\end{thm}

\begin{proof}
For the rest of the proof, let $\mathcal{G}$ (resp. $\mathcal{G}'$)
denote $\widehat{\mathcal{G}_m}$ or $\mathcal{G}_m^{-}$
(resp. $\widehat{\mathcal{G}_m^{\prime}}$ or
$(\mathcal{G}_m^{\prime})^{-}$) as the case may be. With similar
conventions, let $\mathcal{G}_c$ be the relevant commutation poset.

Since $\mathcal{G}'$ is a subcomplex of $\mathcal{G}_c$ and
$\mathcal{G}$ is the corresponding quotient complex, we have a long
exact sequence  of spaces

$$\xymatrix{X_{\mathcal{G}'}(n-1)\ar@{^{(}->}[r] &X_{\mathcal{G}_c}(n)\ar@{->>}[r]
  &X_{\mathcal{G}}(n)}$$

This induces the Puppe map from $X_{\mathcal{G}}(n)$ to
$X_{\mathcal{G}'}(n-1)\wedge S^1=X_{\mathcal{G}'}(n)$. As proved in
\cite{CMPOZSzDT}, this map induces an isomorphism in homology, and
since we can choose $n$ large enough to ensure that both the sides are simply
connected, the map is a homotopy equivalence.
 \end{proof}

\subsection{Stabilization}
The situation for stabilization is slightly different. For the case of
$\widehat{\mathcal{G}}$ we can no longer hope 
for any sort of homotopy equivalence.

\begin{thm} If $H$ and $G$ are the grid diagrams before and after
  stabilization, then for $m$ any Alexander grading and $n$
  sufficiently large, $X_{\widehat{\mathcal{G}_m}}(n)$
  (resp. $X_{\mathcal{G}_m^{-}}(n)$) is homotopic to
  $X_{\widehat{\mathcal{G}_m}}(n)\vee X_{\widehat{\mathcal{G}_{m+1}}}(n)$ (resp. $X_{\mathcal{G}_m^{-}}(n)$).
\end{thm}

\begin{proof}
In case $(a)$ (resp. case $(b)$) of the stabilization, both
$\widehat{\mathcal{G}_s}$ and $\mathcal{G}_s^{-}$ have a subcomplex
(resp. quotient complex) corresponding to either one or two copies of the complex for $H$ and the
corresponding quotient complex (resp. subcomplex) corresponds to the
complex for $G$. Following the lines  of the previous proof, we
observe that these spaces then fit into an exact sequence. Thus the
Puppe map gives a map between the two spaces corresponding to the two
complexes. This map induces a chain map between the
two complexes. Thus if both sides are stabilized sufficiently so as to
ensure that they are simply connected, and if the Puppe map induces
isomorphism in homology, then the Puppe map would be a homotopy
equivalence. Thus we only need to show that the map induced in
homology is an isomorphism. 

Following the lines of the proof in \cite{CMPOZSzDT}, we prove that it is a
quasi-isomorphism. Note that since we then prove that this map is
induced from a homotopy equivalence of spaces, the map actually
becomes a chain homotopy equivalence.

We fix some Alexander grading $m$, and only work with generators of
that grading. On $(\mathcal{G}_s^{-})_m$ (and hence on
$\mathcal{G}^{-}_m$ and $\mathcal{H}^{-}_m$), we introduce additional
filtrations given by powers of $U_2,U_3,\ldots,U_N$. We then put
special markings on every square of $G$ other than the ones on the
vertical or the horizontal annulus through $X_0$. On the associated
graded object, obtained after the  filtration by the powers of $U_2,U_3,\ldots,U_N$, we put an
additional filtration by counting how many times a domain passes
through the extra markings. We call this filtration
$\mathcal{F}'$, and on the associated graded objects of $\mathcal{F}'$ we put an additional
filtration $\mathcal{F}$ given by sum of powers of $U_0$ and $U_1$. (Note that while
working in the hat version, the filtrations coming from the powers of
$U_i$'s are unimportant).

Restricted to the generators coming from $G$, the objects in the
associated graded object of $\mathcal{F}$ are similar in the hat
version and the minus version. For now, we concentrate on the
associated graded object in $\widehat{\mathcal{G}}$ after the
filtration $\mathcal{F}$. We only work with case $(a)$ of the
stabilization. Similar results hold true for case $(b)$ after the
rotation $R(\frac{\pi}{2})$, but some of the maps are in the opposite
direction.

Recall that $\rho$ is the intersection between $\alpha_s$ and
$\beta_s$. Let $p$ be the intersection point immediately to the right
of $\rho$. Let $\widehat{\mathcal{I}}$ (resp. $\widehat{\mathcal{J}}$) be all the points
in $\widehat{\mathcal{G}}$ whose one of the coordinates is $\rho$
(resp. $p$). 
We name the $\alpha$ (resp. $\beta$) circle just below $\alpha_s$
(resp. right of $\beta_s$) as $\alpha_o$ (resp. $\beta_o$). Let
$\widehat{\mathcal{N}}$ be all the generators which do not have any
coordinate among the $4$ points of intersection among
$\alpha_s,\alpha_o,\beta_s,\beta_o$. All the other types of generators
are shown in Figure \ref{fig:gentypes}. 

\begin{figure}[ht]
\psfrag{x}{$x$}
\psfrag{y}{$y$}
\psfrag{i1}{$I_1$}
\psfrag{i2}{$I_2$}
\psfrag{j1}{$J_1$}
\psfrag{j2}{$J_2$}
\psfrag{r}{$R$}
\psfrag{s}{$S$}
\psfrag{t}{$T$}
\psfrag{u}{$U$}
\psfrag{o0}{$O_0$}
\psfrag{o1}{$O_1$}
\psfrag{x0}{$X_0$}
\psfrag{rh}{$\rho$}
\psfrag{p}{$p$}
\psfrag{as}{$\alpha_s$}
\psfrag{ao}{$\alpha_o$}
\psfrag{bs}{$\beta_s$}
\psfrag{bo}{$\beta_o$}
\begin{center}
\includegraphics[width=350pt]{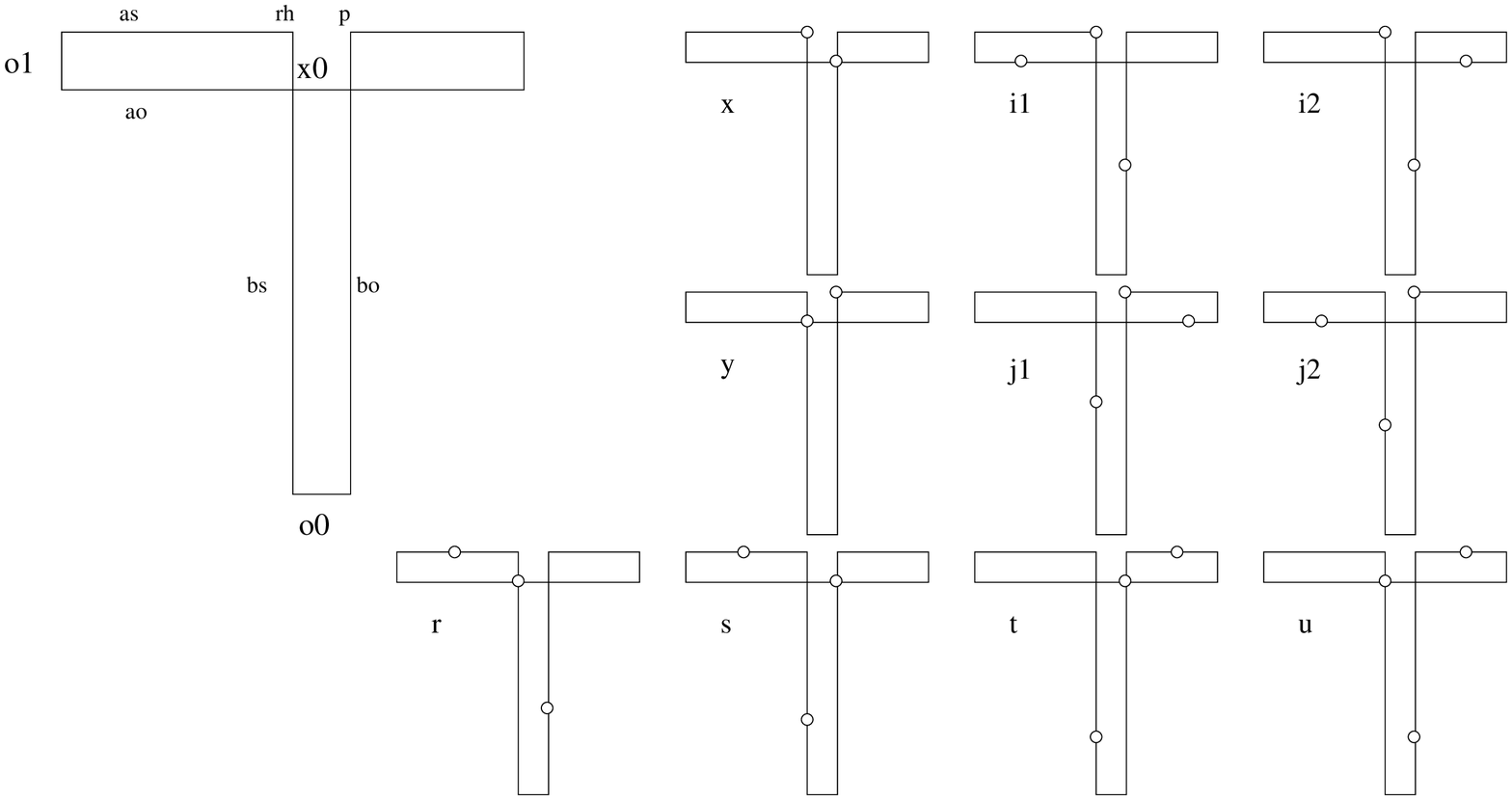}
\end{center}
\caption{Different types of generators after filtration $\mathcal{F}$}\label{fig:gentypes}
\end{figure}

Note that $\widehat{\mathcal{I}}$ (resp. $\widehat{\mathcal{J}}$) consists
of $I_1$ (resp $J_1$), $I_2$ (resp. $J_2$) and the
special generator $x$ (resp. $y$). After the filtration $\mathcal{F}$
domains have to lie in the union of the horizontal and the vertical
annulus through $X_0$ and are not allowed to pass through
$X_0,O_0,O_1$. Hence the chain
complex decomposes as a direct sum of the chain complexes
$\widehat{\mathcal{N}}$, $x$, $y$, $I_2$, $J_2$, $W_I$ and $W_J$,
where $W_I$ (resp. $W_J$) consists of the generators from $I_1,R$ and
$S$ (resp. $J_1,T$ and $U$). It is easy to see that the homology of $\widehat{\mathcal{N}}$
is zero, and there is no differential in the next $4$ summands. The
differentials in $W_I$ map one element (say $r$) of $I_1$ and one
element (say $s(r)$) of $S$
to one element of $R$, and the differentials in $W_J$ map one element
of $T$ to one element in $J_1$ (say $t$) and one element of $U$ (say $u(t)$). Thus the
homology of $W_I$ (resp. $W_J$) is freely generated by elements like $r\pm
s(r)$ (resp. $t$ or $u(t)$).

This gives us the generators for the homology of the associated graded
object of $\widehat{\mathcal{G}}$ after the filtration
$\mathcal{F}$. We want to show that the map coming from the covering
relations between elements of $\widehat{\mathcal{G}}$ and
$\widehat{\mathcal{H}}\cup\widehat{\mathcal{H}'}$ in
$\widehat{\mathcal{G}_s}$ is a quasi-isomorphism. For that it is
enough to show that the map induces isomorphism on the homology of the
associated graded object of  $\mathcal{F}$. This is easy, since the
map from $\widehat{\mathcal{G}}$ to $\widehat{\mathcal{H}}$ (resp. $\widehat{\mathcal{H}'}$) maps $y$
to $x$ and is a bijection from $J_1$ to $I_2$ and from $J_2$ to $I_1$
(resp. induces identity map on $\widehat{\mathcal{I}}$ and maps $S$ to
$0$). Note that this is independent of the sign convention chosen.

For the minus version we have to do a little bit more work.
Let us use the shorthand $U^k$ to denote terms of the form
$U_0^{k_0}U_1^{k_1}$. The domains are now allowed to pass through
$O_0$ and $O_1$, and due to the result of the previous part, we are
only interested in domains connecting elements of the form
$U^k\widehat{\mathcal{I}}$ or $U^k S$ to elements of the form
$U^k\widehat{\mathcal{J}}$ or $U^k U$.

For the special generator $x\in\widehat{\mathcal{I}}$, there are domains
connecting $U^k x$ to $U_0 U^k y$ and $U_1 U^k y$. Thus if we consider
all the generators of the form $U_0^{k_0}U_1^{k_1}x$ with $k_0+k_1=k$
and all the generators of the form $U_0^{l_0}U_1^{l_1}y$ with
$l_0+l_1=k+1$, the chain complex looks like Figure \ref{fig:tree} and
it is easy to see that the homology is carried by $U_1^{k+1}y$,
irrespective of the sign convention.

\begin{figure}[ht]
\begin{center}
\psfrag{x00}{$U_0^2 x$}
\psfrag{x01}{$U_0 U_1 x$}
\psfrag{x11}{$U_1^2 x$}
\psfrag{y000}{$U_0^3 y$}
\psfrag{y001}{$U_0^2 U_1 y$}
\psfrag{y011}{$U_0 U_1^2 y$}
\psfrag{y111}{$U_1^3 y$}
\includegraphics[width=150pt]{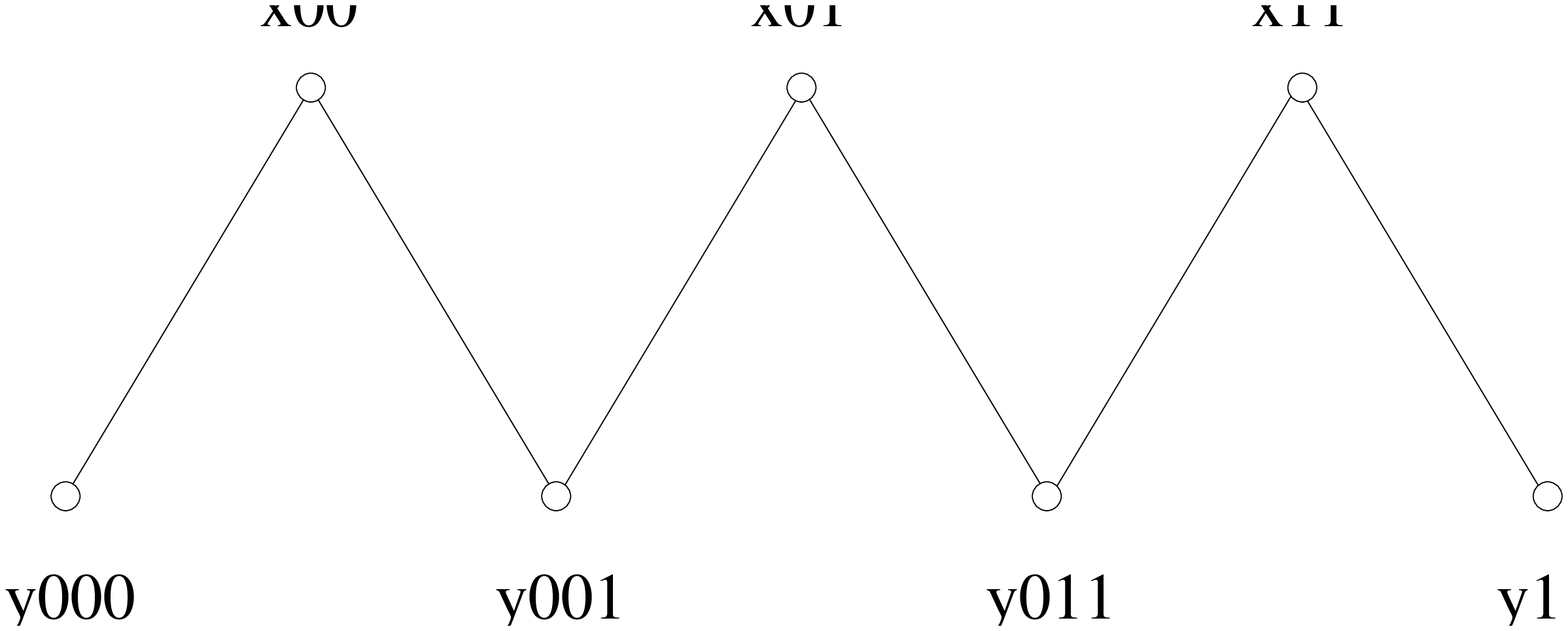}
\end{center}
\caption{The boundary maps  for the filtration $\mathcal{F}'$}\label{fig:tree}
\end{figure}

For a generator $x'\in I_2$, the domains connect $U^k x'$ to
$U_0 U^k t$ and $U_1 U^k u(t)$ for some generators $t\in J_1$ and
$u(t)\in U$. Since $u(t)$ is same as $t$ in the homology of the
associated graded object of $\mathcal{F}$, hence the diagram once more
looks like Figure \ref{fig:tree}. Thus the homology is again carried by
$U_1^{k+1} t$. 

Similarly, for a generator $r\in I_1$ (resp. $s(r)\in S$), there is
only domain connecting $U^k r$ (resp. $U^k s(r)$) to $U_0 U^k y'$
(resp. $U_1 U^k y'$) for some generator $y'\in J_2$. A similar
argument shows that the homology is once more carried by $U_1^{k+1} y'$. 

Thus the homology of the associated graded  object is freely generated by
elements of the form $U_1^k \widehat{\mathcal{J}}$. The map which we are
trying to show is a  quasi-isomorphism induces a bijection between
elements of that form and generators of $\mathcal{H}^{-}$.
This completes the proof of the fact that the relevant maps are quasi-isomorphisms, and as argued
earlier this completes the proof of the theorem.
\end{proof}

Thus to every knot $K$ and every Alexander grading $m$, we can associate an invariant spectrum
$\mathcal{S}(\mathcal{G}^{-}_m)$, and hence after taking an infinite
wedge,  the spectrum $\mathcal{S}(\mathcal{G}^{-})$. We call these
spectra $\mathcal{S}^{-}_m$ and $\mathcal{S}^{-}$ to stress the fact
that they only depend on the knot $K$, and not on the grid diagram
representing $K$. Thus any invariant of the spectrum is also a knot
invariant. The homology of the spectrum $\mathcal{S}^{-}$ is the
well-known invariant $HFK^{-}(K)$. Stable homotopy groups can
constitute an interesting collection of invariants. Another invariant
to consider is the Steenrod operations. For simplicity, let us just
consider the $Sq$ operation acting on the cohomology with
$\mathbb{F}_2$ coefficients which increases the grading. Since
cohomology of $\mathcal{G}^{-}$ is same as $HFK^{-}(r(K))$ where $r(K)$ is
the reverse of the knot (the isomorphism being obtained by applying a
rotation $R(\frac{\pi}{2})$ on a grid diagram $G$ for the knot $K$),
the Steenrod squares act on $HFK^{-}(K, \mathbb{F}_2)$ by reducing the
grading.

A very  natural question is whether $\mathcal{S}^{-}$ computes
anything new. It will interesting to find two knots $K_1$ and $K_2$,
such that $\mathcal{S}^{-}(K_1)$ and $\mathcal{S}^{-}(K_2)$ have the
same homology, but are not homotopic to one another.

For the hat version, unfortunately we do not have a knot
invariant. Given an index $N$  grid diagram $G$ for a knot we can
construct finite spectra $\mathcal{S}(\widehat{\mathcal{G}_m})$ and
their wedge $\mathcal{S}(\widehat{\mathcal{G}})$.

It is not clear whether the homotopy type of these spectra depend  only on
$K$ and $N$. However there is some partial answer to this question. Let
$g$ be the highest Alexander grading $m$ such that the homology of
$\widehat{\mathcal{G}_m}$ is non-trivial. It is easy to see that $g$
depends only only the knot $K$. Then we claim that
the homotopy type of the spectrum
$\mathcal{S}(\widehat{\mathcal{G}_g})$ also depends only on the knot $K$,
and henceforth we will denote it by $\widehat{\mathcal{S}}_g$. The way
to see this is  as follows. For sufficiently large $k$ and $m>g$, the spaces
$X_{\widehat{\mathcal{G}_m}}(k)$ are acyclic as they are simply
connected and have trivial homology. Commutation does not change the
homotopy type of $\mathcal{S}(\widehat{\mathcal{G}_g})$, and when we
stabilize to go from a grid diagram $H$ to a grid diagram $G$, for
sufficiently large $k$, we have
$X_{\widehat{\mathcal{G}_g}}(k)=X_{\widehat{\mathcal{H}_g}}(k)\vee
X_{\widehat{\mathcal{H}_{g+1}}}(k)\sim X_{\widehat{\mathcal{H}_g}}(k)$
since the second space is acyclic. 

In fact this proof shows a possible way to
answer the above question positively.  We are trying to show that the
homotopy type of $\mathcal{S}(\widehat{\mathcal{G}_m})$ depends only
on $K$, the Alexander grading $m$ and the grid number $N$. 
First note that it is enough to prove the following fact. If the
stabilizations of $G$ and $G'$ have spectra that are homotopy
equivalent, then the spectra for $G$ and $G'$ are homotopy
equivalent. We have already proved this for $m\geq g$. So by induction
assume it is true for all Alexander grading bigger than $m$. Thus for
$k$ sufficiently large, we
have $X_{\widehat{\mathcal{G}_m}}(k)\vee
X_{\widehat{\mathcal{G}_{m+1}}}(k)\sim X_{\widehat{\mathcal{G}_m^{\prime}}}(k)\vee
X_{\widehat{\mathcal{G}_{m+1}^{\prime}}}(k)$. But by induction, we
already know  $X_{\widehat{\mathcal{G}_{m+1}}}(k)\sim X_{\widehat{\mathcal{G}_{m+1}^{\prime}}}(k)$.
Thus our proof would be complete if for finite CW complexes $X$, $Y$
and $A$, $X\vee A$ being homotopic to $Y\vee A$ would imply that $X$  is stably
homotopic to $Y$.

However, irrespective of that, we can still construct certain stable
homotopy invariants from the spectra $\mathcal{S}(\widehat{\mathcal{G}})$ which
depend only on $K$ and $N$. One such example is the stable homotopy groups.

\begin{thm} The stable homotopy groups of
  $\mathcal{S}(\widehat{\mathcal{G}_m})$  depend only on $K$, $m$ and $N$.
\end{thm}

\begin{proof} 
  We just mimic our attempted proof for showing the homotopy type of
  $\mathcal{S}(\widehat{\mathcal{G}_m})$ depends only on $K$, $m$ and
  $N$. Call two grid diagrams $G$ and $G'$ to be $r$-equivalent if
  after stabilizing both of them $r$ times, the two diagrams can be
  related by commutations. We are trying to prove that
  $\pi^s_i(\mathcal{S}(\widehat{\mathcal{G}_m}))=
  \pi^s_i(\mathcal{S}(\widehat{\mathcal{G}_m^{\prime}}))$ for two
  $r$-equivalent diagrams $\wh{\mc{G}_m}$ and
  $\wh{\mc{G}_m^{\prime}}$. This is true if either $r=0$ or the
  Alexander grading $m$ is sufficiently large. We prove this by an
  induction on the pairs $(r,-m)$ ordered lexicographically.

If two diagrams  $G$ and $G'$ are $r$-equivalent, then their
stabilizations are $(r-1)$-equivalent, and hence from the induction on
$(r,-m)$, we get $\pi^s_i(\mathcal{S}(\widehat{\mathcal{G}_m})\vee
\mathcal{S}(\widehat{\mathcal{G}_{m+1}})) = \pi^s_i(\mathcal{S}(\widehat{\mathcal{G}_m^{\prime}})\vee
\mathcal{S}(\widehat{\mathcal{G}_{m+1}}))$. However for spectra coming
from finite CW complexes, the stable homotopy groups are finitely
generated and abelian, and for wedges, they are products, and hence
(using the classification of finitely generated abelian groups) we
get $\pi^s_i(\mathcal{S}(\widehat{\mathcal{G}_m}))=
\pi^s_i(\mathcal{S}(\widehat{\mathcal{G}_m^{\prime}}))$.
\end{proof}

\section{Examples}\label{sec:examples}

In  this section we give examples of some other GSS  posets $P$, and
construct the spaces $X_P(n)$ corresponding to them. We will conclude
the section by computing the homotopy type of
$X_{\widehat{\mathcal{G}}}(n)$ for the grid diagram $G$ of
the trefoil as shown in Figure \ref{fig:trefoil1}.

\begin{figure}[ht]
\psfrag{X1}{X}
\psfrag{X2}{X}
\psfrag{X3}{X}
\psfrag{X4}{X}
\psfrag{X5}{X}
\psfrag{O1}{O}
\psfrag{O2}{O}
\psfrag{O3}{O}
\psfrag{O4}{O}
\psfrag{O5}{O}
\psfrag{a0}{$0$}
\psfrag{a1}{$1$}
\psfrag{a2}{$2$}
\psfrag{a3}{$3$}
\psfrag{a4}{$4$}
\psfrag{b0}{$0$}
\psfrag{b1}{$1$}
\psfrag{b2}{$2$}
\psfrag{b3}{$3$}
\psfrag{b4}{$4$}
\begin{center}
\includegraphics[width=120pt]{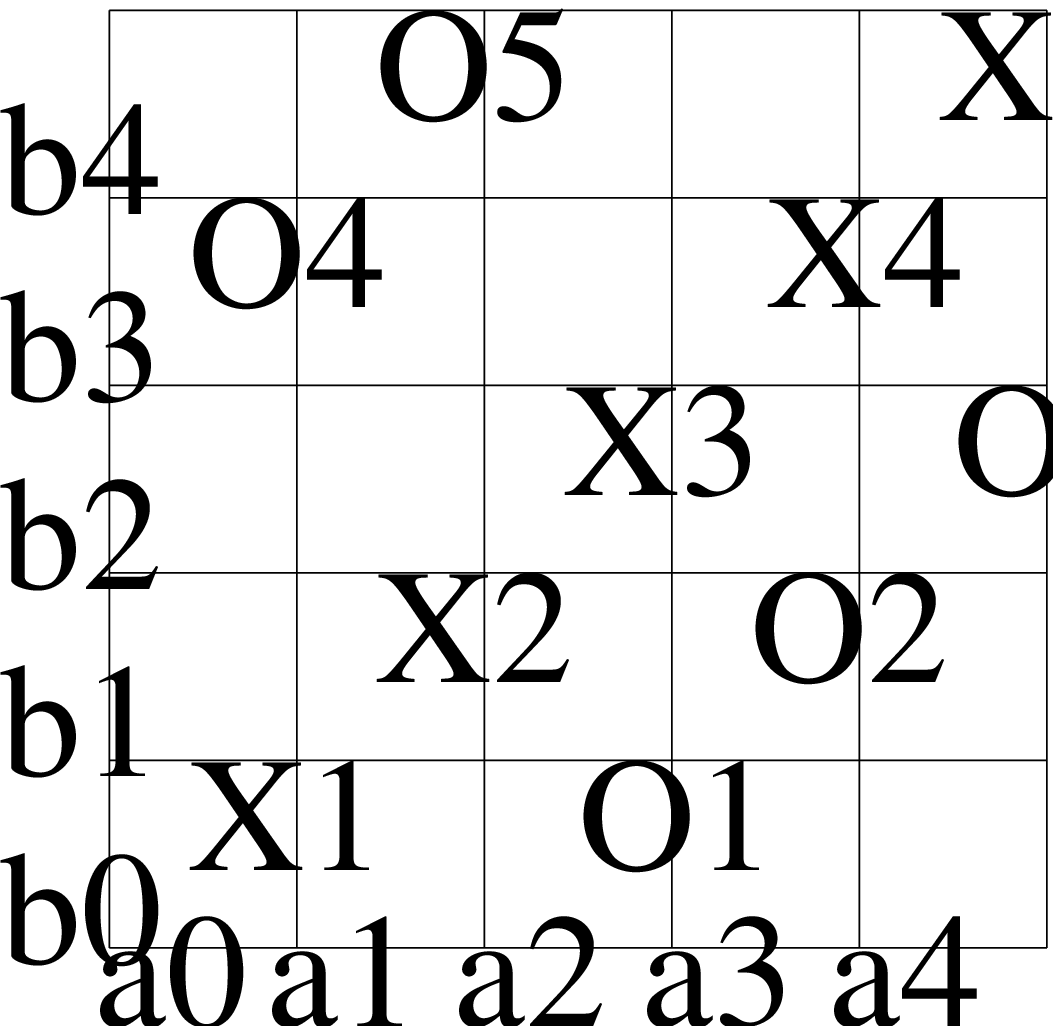}
\end{center}
\caption{Grid diagram for the trefoil}\label{fig:trefoil1}
\end{figure}

As a warm-up exercise, let us first consider the crown poset $C_n$. In this poset there is a unique
minimum $a$ (resp. unique maximum $d$), which is  covered by the elements
$b_1,b_2,\ldots,b_n$ (resp. which covers $c_1,c_2,\ldots c_n$). Furthermore,
each $b_i$ is covered by $c_i$ and $c_{i+1}$ with the counting done
modulo $n$. This can be made into a graded poset by assigning gradings
of $0,1,2$ and $3$ to $a$, $b_i$, $c_i$ and $d$ respectively.

There is also a  sign assignment which assigns $+1$ to each edge that
involves either $a$ or $d$, and to each edge of the form
$b_i\leftarrow c_i$ and assigns $-1$ to every other edge. Since  the
poset has a unique minimum, this is the unique sign assignment (this
actually follows from the fact that $C_n$ is shellable).

It is easy to check that this poset is  also shellable. Let us draw a
graph whose vertices are maximal chains, and there is an edge joining
two vertices if and only if the two maximal chains agree at exactly $3$
elements. It is clear that the graph is a $(2n)$-cycle. Let us
now delete one of the edges of this graph, and put a direction on the
remaining edges, such that there is at most one edge flowing into a
vertex and there is at most one edge flowing out of a vertex. The
shellable total order that we  put on the maximal chains is the
following. We declare a maximal chain $\mathfrak{m}_1$ to be smaller
than a maximal chain $\mathfrak{m}_2$ if we can go from
$\mathfrak{m}_2$ to $\mathfrak{m}_1$ along directed edges in the
modified  graph. It is easy to check that this ordering suffices.

Thus $C_n$ is a GSS poset and  we can associate a  pointed CW complex
$X_{C_n}(m)$ to it such that the reduced homology
$\widetilde{H}_{*}(X_{C_n}(m))$ is the homology of the chain complex
associated to $C_n$. However it is immediate that the chain complex
has trivial homology, and hence for sufficiently large $m$,
$X_{C_n}(m)$ is a simply connected space with trivial homology and
hence is homotopic to a point.

Now let us consider some other families of examples. Let $I$ be the poset of two elements $0$ and
$1$, with $0\preceq 1$. Let $I^n$ be the $n$-fold Cartesian product of
$I$ with itself. For very natural reasons, let us call this the
$n$-cube poset.

The elements of $I^n$ look like $n$-tuples $a=(a_1,a_2,\ldots, a_n)$
where each $a_i$ is $0$ or $1$. We put a grading on this poset by
declaring the grading of $a$ to be the number of $1$'s in the
$n$-tuple. We can also put a sign assignment on this poset in the
following way. Observe that if $a\leftarrow b$, then there is a unique
$k$ for which $a_k=0$ and $b_k=1$, and for every $i\neq k$,
$a_i=b_i$. We assign a sign of $(-1)^{\sum_{i=1}^k a_i}$ to this
covering relation, and it is easy to check that this is indeed a sign
assignment. Since  there is a unique minimum, this is the only sign
assignment up to equivalence.

This poset is also EL-shellable. In an edge $a\leftarrow b$, if $k$ is
the unique place where $a_k<b_k$, we label the edge by the integer
$k$. It is easy  to see that this map from the covering relations to
integers totally ordered in the standard way, is indeed an
EL-shelling. However once more since the homology of the chain complex
associated to $I^n$ is trivial, the  CW complex $X_{I^n}(m)$ is
contractible for sufficiently large $m$.

The $n$-cube poset is naturally isomorphic to the subset poset, whose
elements are subsets of $\{x_1,\ldots,x_n\}$ partially ordered by
inclusion. An element $a$ of $I^n$ corresponds to a subset $S$, such
that $x_i\in S$ if and only if $a_i=1$. 

Now consider the $(n+1)$-cube poset restricted to the elements of
positive grading. Let us reduce the grading of each element by $1$,
and then rename it as the simplex poset $\Delta_n$ since the grading
$k$ elements of this poset correspond to $k$-simplices lying inside an
$n$-simplex $\Delta^n$, with partial order being given by inclusion.
Thus $\Delta_n$ is graded with $k$-simplices having grading $k$. It
has a sign assignment obtained by restricting the sign assignment of
the subset poset, and this is the unique sign assignment, since
$\Delta_n$ has a unique maximum.  It is also shellable, since it is
isomorphic to the interval $(\varnothing,\infty)$ of the subset poset.
So $\Delta_n$ is a GSS poset.

We can also construct the reduced simplex poset
$\widetilde{\Delta}_n$, where we label one of the vertices of the
$n$-simplex $\Delta^n$ to be the basepoint $b$, and define
$\widetilde{\Delta}_n=\Delta_n\setminus\{ b\}$ with the same partial
order. This poset also has grading and sign assignments, and each
closed interval in this poset is still shellable since closed
intervals in the poset $\Delta_n$ are shellable. Thus
$\widetilde{\Delta}_n$ is also a GSS poset.

\begin{thm}\label{thm:simplex}
For $m$ large enough, there is a well-defined homeomorphism
$h_{\Delta_n,m}$ between
$X_{\Delta_n}(m)$ and $(\Delta^n\cup\{b\})\wedge S^m$, where
$\Delta^n\cup\{b\}$ is the one-point compactification of $\Delta^n$
with $b$ being the basepoint. 
\end{thm}

\begin{proof}
Let $d_n$ be the maximum element in $\Delta_n$. Since the poset
$(-\infty, d_n)$ is shellable and thin, its order complex is
$S^{n-1}$. After reversing the partial order, if we recall the
construction from Theorem \ref{thm:yinfty}, then we see that this
partially ordered set comes from a CW complex structure, whose $k$-cells
correspond to elements of grading $n-1-k$. However $S^{n-1}$ can also
be thought of as the boundary of the $n$-simplex with the inherited
simplicial structure where $k$-cells correspond to elements of grading
$k$. It is relatively easy to check that this  is the dual
triangulation of the CW complex structure.

Now recall how $X_{\Delta_n}(m)$ is defined. We embed the order
complex of the reverse of $\Delta_n$ into $\mathbb{R}^{n+m}$ in some
standard way. We take the image of the point corresponding to $d_n$
(denoted in Section \ref{sec:cwcomplexes} as $e^0_1$),
and construct the first step of its right-handed disk $r^0_{1,0}$
which is simply an $(n+m)$-dimensional ball $B^{n+m}$. We extend
$r^0_{1,j}$ to $r^0_{1,j+1}$ by marking some  thickened  $j$-cells
lying on $\partial B^{n+m}=S^{n+m-1}$, one for
each element of grading $(n-j-1)$ in $\Delta_n$. Finally, we quotient out
everything in $S^{n+m}$ that  is not marked, to a point to
obtain $X_{\Delta_n}(m)$.

Thus we are embedding $S^{n-1}$ with the CW complex structure as
described in the first paragraph into $S^{n+m-1}=\partial B^{n+m}$, taking a
regular neighborhood  of that in $S^{n+m-1}$, and then
quotienting out its complement in $S^{n+m-1}$ to a point to obtain
$X_{\Delta_n}(m)$. But the dual triangulation of that $S^{n-1}$ is the
simplicial complex $\partial \Delta^n$, and since $m$ is sufficiently
large, this embedding of $\partial \Delta^n$ in $\partial B^{n+m}$ can
be extended in a standard way to a proper embedding of $\Delta^n$ in
$B^{n+m}$. For for $m$ large enough, we then can view $B^{n+m}$ as
$\Delta^n\times D^m$, where $D^m$ is the $m$-dimensional disk, with
the closure of the regular neighborhood of $\partial \Delta^n$ in
$S^{n+m-1}$ being $\partial\Delta^n\times D^m$. The space
$X_{\Delta_n}$ is obtained by quotienting $\Delta^n\times \partial
D^m$ to a point. This is illustrated in Figure \ref{fig:deltan} for
$n=2$ and $m=1$.

\begin{figure}[ht]
\psfrag{d}{$D^1$}
\psfrag{del}{$\Delta^2$}
\begin{center}
\includegraphics[width=120pt]{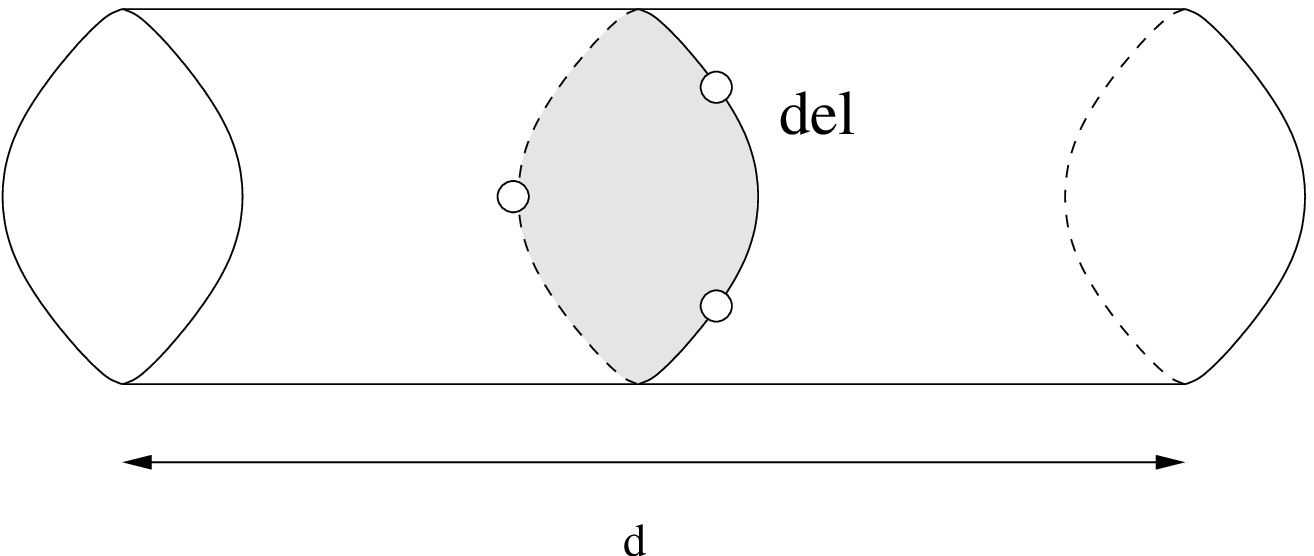}
\end{center}
\caption{Construction of $X_{\Delta_2}(1)$}\label{fig:deltan}
\end{figure}

We end the proof by noting that 

$(\Delta^n\times D^m)/(\Delta^n\times\partial D^m)=(\Delta^n\times
S^m)/(\Delta^n\times\{pt\})=(\Delta^n\cup\{b\})\wedge S^m$
\end{proof}

\begin{thm}\label{thm:redsimplex}
For $m$ large enough, there is a well-defined homeomorphism
$h_{\widetilde{\Delta}_n,m}$ between
$X_{\widetilde{\Delta}_n}(m)$ and $\Delta^n\wedge S^m$, where $b$ is
the basepoint in $\Delta^n$.
\end{thm}

\begin{proof}
We construct $X_{\widetilde{\Delta}_n}(m)$ in a similar way. View
$\widetilde{\Delta}_n$ as a quotient poset of $\Delta_n$ and let $d_n$
be the maximum element of $\Delta_n$. We embed $S^{n-1}$,
the order complex of the reverse of $(-\infty,d_n)$ (which itself is
being thought of a subposet of $\Delta_n$) into $S^{n+m-1}$, the
boundary of $B^{n+m}$. We know that the dual triangulation of $S^{n-1}$ is
the simplicial structure on $\partial \Delta^n$, and for $m$ large
enough we can view $B^{n+m}$ as $\Delta^n\times D^m$, with $S^{n-1}$
being embedded as $\partial\Delta^n\times\{pt\}$, and the closure of its  regular
neighborhood being $\partial\Delta^n\times D^m$.

However, since we are working with
$\widetilde{\Delta}_n=\Delta_n\setminus\{b\}$, we actually embed the
reverse of $(-\infty, d_n)$, now thought of as a subposet of
$\widetilde{\Delta}_n$. This order complex is $S^{n-1}$ minus the
$(n-1)$-dimensional cell corresponding to the vertex $b$ in
$\Delta^n$. Thinking in terms of the dual triangulation, it is
$\partial \Delta^n\setminus N(b)$, where $N(b)$ is a small
neighborhood  of the basepoint $b\in\Delta^n$. The complement of a
regular neighborhood of this order complex in $\partial(\Delta^n\times  D^m)$ can be thought
of as $(\Delta^n\times \partial D^m)\cup(N(b)\times D^m)$. We obtain
$X_{\widetilde{\Delta}_n}(m)$ by starting with $\Delta^n\times D^m$
and then quotienting out $(\Delta^n\times
\partial D^m)\cup(N(b)\times D^m)$ to a point.

We once  more end the proof by noting

\noindent $(\Delta^n\times D^m)/((\Delta^n\times \partial D^m)\cup(N(b)\times
D^m))=(\Delta^n\times D^m)/((\Delta^n\times \partial
D^m)\cup(\{b\}\times D^m))=\Delta^n\wedge S^m$
\end{proof}

For the next theorem, let $h_{\Delta_n}$ denote either $h_{\widetilde{\Delta}_n}$
or $h_{\Delta_n}$ depending on whether $\Delta^n$ contains
a special marked  vertex $b$ or not. Similarly let $\Delta_n$ denote
either $\widetilde{\Delta}_n$ or $\Delta_n$, and correspondingly let
$S^m(\Delta_n)$ denote either $\Delta_n\wedge S^m$ or
$(\Delta_n\cup\{b\})\wedge S^m$.

\begin{thm}
Let $\Delta^{n-1}$ be a codimension-$1$ face in $\Delta^n$. There can
be three cases regarding the role of the basepoint $b$, namely,
$b\in\Delta^{n-1}$, $b\in(\Delta^n\setminus\Delta^{n-1})$ or
$b\notin\Delta^n$. In either case for sufficiently large $m$, the following diagram commutes

$$\xymatrix{X_{\Delta_{n-1}}(m)\ar[r]^{h_{\Delta_{n-1}}}\ar@{^{(}->}[d]
&S^m(\Delta_{n-1})\ar@{^{(}->}[d]\\
X_{\Delta_n}(m)\ar^{h_{\Delta_n}}[r] &S^m(\Delta_n)}$$

where the inclusion on the left is given by Theorem
\ref{thm:subquotient}, and the inclusion on the right is induced from
the inclusion of $\Delta^{n-1}$ into $\Delta^n$.
\end{thm}

\begin{proof}
Let us just do the case when $b\notin \Delta^n$. Recall that
$X_{\Delta_n}(m)$ is obtained from $\Delta^n\times D^m$ by quotienting
out $\Delta^n\times \partial D^m$. However the inclusion of
$\Delta^{n-1}$ into $\Delta^n$ induces both the inclusion on the left
and the one on the right, and hence the diagram commutes.
\end{proof}

The above theorems have a very interesting corollary which shows that
the CW  complexes $X_P(m)$ can be quite complicated.

\begin{thm}
Let $K$ be a simplicial complex with a special vertex marked as the
basepoint $b$. Then there exists a GSS poset $P$,
such that for sufficiently large $m$, $X_P(m)=K\wedge S^m$.
\end{thm}

\begin{proof}
Let us construct a poset $P$ whose elements in grading $k$ are the
$k$-simplices of $K$ partially ordered by inclusion, and then delete the element corresponding to
$b$. Let us also fix an orientation on every simplex of $K$, and then
assign signs $\pm 1$ based on whether the attaching map preserves orientation
or reverses it. The closed intervals in this poset are isomorphic to the
subset poset, and hence are shellable. Thus $P$ is a GSS poset. For
large enough $m$, let us consider the pointed CW complex $X_P(m)$.

The $(m+k)$-cells of $X_P(m)$ correspond to $k$-cells in $K\setminus\{b\}$, and the boundary
maps of $X_P(m)$ correspond to the boundary maps in $K$. Observe that
$K\wedge S^m$ with its natural pointed CW complex structure also has
this property. Now recall how we construct $X_P(m)$. For each element
$x\in P$, we construct a CW complex corresponding to the poset
$(-\infty,x]$, and whenever $y\preceq x$, there is an embedding of the
CW complex corresponding to $y$ into the CW complex corresponding to
$x$. Since such an inclusion can be viewed as a composition of
inclusions coming from covering relations like $y\leftarrow x$, we can just restrict our
attention to those maps.

If $x$ corresponds to an $n$-simplex, then the poset $(-\infty,x]$ is
either $\widetilde{\Delta}_n$ or $\Delta_n$ depending on whether or
not $b$ is in $\Delta_n$. From the previous theorems, we know that the
CW complex corresponding to $x$ is either $\Delta^n\wedge S^m$ or
$(\Delta^n\cup\{b\})\wedge S^m$, and the inclusion maps coming from
$y\leftarrow x$ are induced from inclusions of simplices in $K$. Thus
$X_P(n)$ and $K\wedge S^m$ have the same CW complex structure, and
hence are homeomorphic.
\end{proof}

\begin{thm}
There exist GSS posets $P_1$ and $P_2$ with the same homology, but
with different homotopy types of their associated spectra.
\end{thm}

\begin{proof}
We want to find GSS  posets $P_1$ and $P_2$ with same homology, such
that $X_{P_1}(m)$ is not homotopic to $X_{P_2}(m)$ for all $m$. We
choose $P_1$ to be a poset consisting of only two elements, which are
non-comparable and have gradings $2$ and $4$. We choose $P_2$ to be
poset coming from a simplicial complex structure on
$\mathbb{CP}^2$. Clearly both have  homology $\mathbb{Z}^2$ supported
in gradings $2$ and $4$.

Furthermore, $X_{P_1}(m)=S^{m+2}\vee S^{m+4}$ and
$X_{P_2}(m)=\mathbb{CP}^2\wedge S^m$. We want to show that these two
spaces are not homotopic for any $m$, or in other words, we want to
show that $S^2\vee S^4$ is not stably homotopic  to
$\mathbb{CP}^2$. This can be seen in several ways. If $a_2$ and $a_4$
(resp. $b_2$ and $b_4$) denote the generators in $H^2$ and $H^4$ of
$S^2\vee S^4$ (resp. $\mathbb{CP}^2$) with coefficients in
$\mathbb{F}_2$, then $Sq^2(a_2)=0$ but $Sq^2(b_2)=b_4$, where $Sq^2$
is the second Steenrod square operation. Also $\pi^s_3(S^2\vee
S^4)=\mathbb{Z}/2\mathbb{Z}$ and $\pi^s_3(\mathbb{CP}^2)=0$, where
$\pi^s_3$ is the third stable homotopy group. 
\end{proof}

Now as promised at the beginning of the section, we do the computation
for the hat version of the trefoil presented in the grid
as shown in Figure \ref{fig:trefoil1}. We use the notation $a_0 a_1 a_2 a_3
a_4$ to denote the element $a\in\widehat{\mathcal{G}}$ which contains
the points of intersection between the vertical lines marked $i$ and
the horizontal lines marked $a_i$. By components of
$\widehat{\mathcal{G}}$, we mean path connected components of the
graph that represents the partial order on $\widehat{\mathcal{G}}$. A
simple computation shows that there are $25$ components in
$\widehat{\mathcal{G}}$ of which $22$ of them contain only one
element. There are two components $C_1$ and $C_2$ with $26$ elements
each, and homology $\mathbb{Z}^6$, and there is one component $D$ with
$46$ elements and homology $\mathbb{Z}^{14}$.

The CW complex $X_{\widehat{\mathcal{G}}}$ is a wedge of the CW
complexes coming from the different components. The spaces coming from
the components with only one element are simply spheres of the right
dimension, so we can restrict our attention to $C_1,C_2$ and $D$. Let
us first consider the case of $C_i$.

Each of $C_1$ and $C_2$ has a unique element of maximum Maslov
grading (however neither of them have a unique maximum), which happens
to be $12340$ and $23401$ respectively. However these two
generators swap when we apply a rotation of $R(\pi)$ and reverse the
roles of $X$'s and $O$'s (which can be done in the hat version). This
shows that $C_1$ is isomorphic to $C_2$ as posets and hence we can
work with $C_1$. The following are the elements of $C_1$.

\begin{itemize}
\item Maslov grading $2$: $12340$
\item Maslov grading $1$: $12304$, $02341$, $21340$, $13240$, $12430$
\item Maslov grading $0$: $20134$, $12034$, $03124$, $02314$, $21304$, $41203$,
  $13204$, $01423$, $01342$, $40231$, $31240$, $03241$, $14230$, $02431$, $21430$
\item Maslov grading $-1$: $21034$, $31204$, $03214$, $04231$, $01432$
\end{itemize}

The homology $\mathbb{Z}^6$ lies entirely in grading $0$. There are
six maxima in $C_1$ which are $20134$, $03124$, $41203$, $01423$, $40231$ in
grading $0$ and $12340$ in grading $2$.  Let $C$ be the poset
$(-\infty,12340]$ which turns out to be
$C_1\setminus\{20134,03124,41203,01423,40231\}$. Since $C$ is a
subposet of $C_1$, $X_{C_1}(m)$ is obtained by adding five $m$-cells
to $X_C(m)$. However the homology of $C$ is $\mathbb{Z}$ in grading
$0$, hence $\widetilde{H}_i(X_C(m))=0$ for all $i<m$. Since we can
assume all spaces to be simply connected, we have
$\pi_{m-1}(X_C(m))=0$, and hence homotopically there is a unique way
to add the five $m$-cells. Thus we get $X_{C_1}(m)\sim X_C(m)\vee
S^m\vee S^m\vee S^m\vee S^m\vee S^m$.

Thus to find the stable homotopy type of $X_{C_1}$, we only need to
find the stable homotopy type of $X_C$. For convenience, we number the
grading $-1$ elements in $C$ as $a_1,\ldots,a_5$, the grading $0$
elements in $C$ as $b_1,\ldots, b_{10}$, the grading $1$ elements in
$C$ as $c_1,\ldots, c_5$ and the unique grading $2$ element as $d$
(with the numbering being done left to right as they appear in listing
above). The partial order is shown in Figure \ref{fig:complicated},
with the elements in each grading again being numbered from left to right.

\begin{figure}[ht]
\psfrag{a}{$\alpha$}
\psfrag{bbbb}{$B_{\beta}$}
\begin{center} \includegraphics[width=150pt]{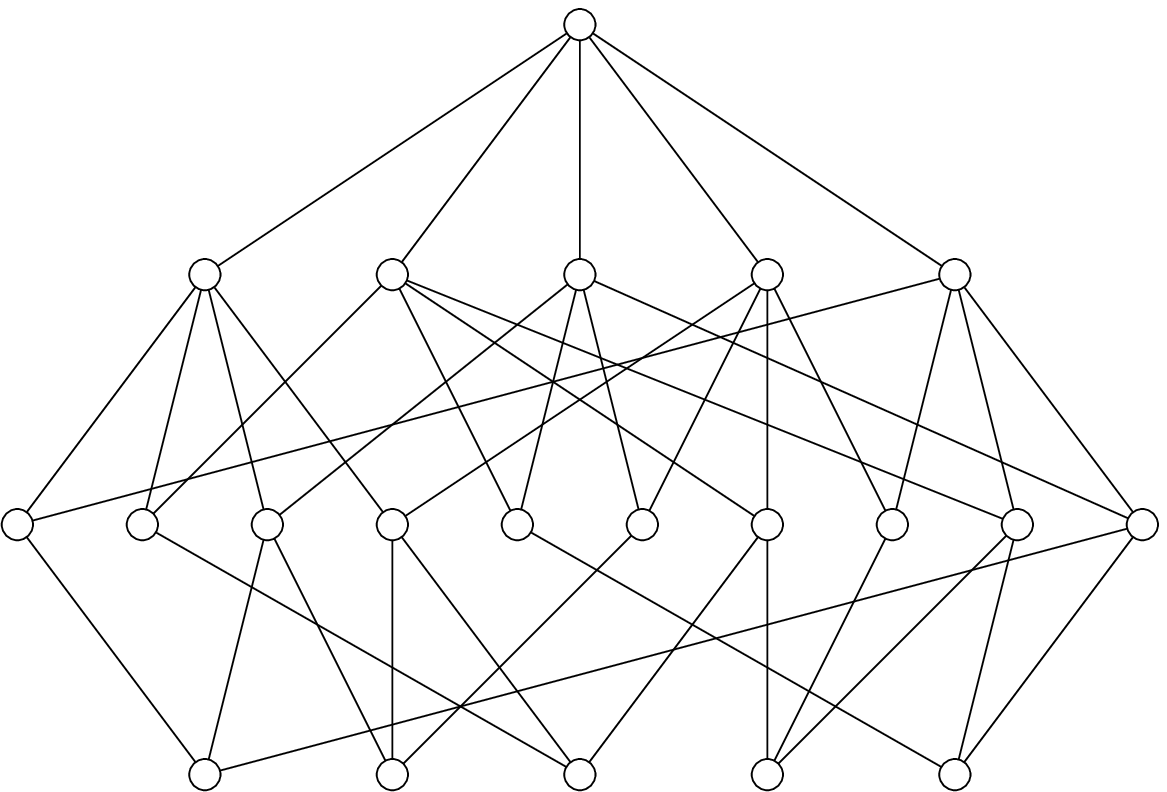}
\end{center}
\caption{The poset $C$}\label{fig:complicated}
\end{figure}

Now we associate a CW complex $P_C$ to $C$, which is closely related to  the
order complex of the reverse of $C$, and then construct $X_C(m)$ as
the Alexander dual of $P_C$. However since the Alexander dual of a
space $X$ in $S^m$ is the Alexander dual of $X\wedge S^1$ in
$S^{m+1}$, we might work with a sufficiently high suspension
$P_C\wedge S^k$ of $P_C$.

Let us now try to understand this  space
$P_C$. We start with the $0$-sphere $S^0$ corresponding to $d$, and we
attach $k$-cells for elements of grading $2-k$, such that the boundary maps
correspond to the covering relations in the reverse of $C$. Throughout the
rest of the section, $\sim$ denotes stable homotopy equivalence, instead of
the usual homotopy equivalence.

We start with $P_{\{d\}}=S^0$. If we attach the $1$-cell for $c_1$, we
get $P_{\{c_1,d\}}\sim\{pt\}$. After attaching the remaining four
$1$-cells, we get $P_{\{c,d\}}=\vee_{i=2}^5 S^1_i$, where $S^1_i$
corresponds to $c_i$. Now we will attach the $2$-cells corresponding to
the elements $b_i$. Since we can take high enough suspensions, while
attaching the $2$-cells, we only care about
$\pi^s_1(P_{\{c,d\}})=\mathbb{Z}^4$. It is easy to see that the
$2$-cells corresponding to $b_1,b_2,b_3$ and $b_4$ kill the generators
in $\pi^s_1(P_{\{c,d\}})$ corresponding to $a_5,a_2,a_3$ and $a_4$
respectively. Thus $P_{\{b_1,b_2,b_3,b_4,c,d\}}\sim\{pt\}$, and hence
after attaching the remaining six $2$-cells, we get
$P_{\{b,c,d\}}=\vee_{i=5}^{10}S^2_i$, where $S^2_i$ corresponds to
$b_i$. We now attach $3$-cells corresponding to $a_i$'s, and since
$\pi^s_2(P_{\{b,c,d\}})=\mathbb{Z}^6$,  the $3$-cells corresponding to
$a_1,a_2$ and $a_3$ kill the generators corresponding to $b_{10},b_6$
and $b_7$ respectively. Thus we have
$P_{\{a_1,a_2,a_3,b,c,d\}}=S^2_5\vee S^2_8\vee S^2_9$ where $S^2_i$ is
still a $2$-sphere corresponding to $b_i$. The $3$-cells coming from
$a_4$ and $a_5$ identify $b_9$ to $b_8$ and $b_5$ respectively, and
hence $P_C\sim S^2_9=S^2$.

Since the Alexander dual of a sphere is a sphere, we get
$X_C(m)\sim S^m$. As discussed before, this implies $X_{C_1}(m)\sim
\vee_{i=1}^6 S^m$. Also note that the construction is entirely
independent of the choice of a sign convention. In fact, $C_1$ has only
one  sign assignment up to equivalence. This is because $C$ being a GSS
poset with a unique maximum has only one sign assignment, and that
extends uniquely to $C_1$ since every element of $C_1\setminus C$
covers exactly one element in $C_1$.

In $D$, there are six elements in grading $0$, thirty elements in
grading $-1$ and ten elements in grading $-2$. Consider a subposet
$D_1$ of $D$ consisting of the elements $\{42103$, $10423$, $20143$, $43120$,
$40321$, $13024$, $20314$, $14203$, $41320$, $03142\}$ in grading $-1$ and all
the ten elements  in grading $-2$. The poset $D_1$ has ten components,
and each component is isomorphic to $I$, the chain of length
$2$. Hence $X_{D_1}(m)\sim\{pt\}$. Let $D_2$ be the subposet of $D$
consisting of all the elements in gradings $-1$ and $-2$. Since $D_1$
is a subposet of $D_2$, $X_{D_2}(m)$ is obtained by adding twenty
$(m-1)$-cells to $X_{D_1}(m)$, and there is only one way of doing that,
leading to $X_{D_2}(m)\sim\vee_{i=1}^{20} S^{m-1}$. The space $X_D(m)$
is obtained from $X_{D_2}(m)$ by attaching six $m$-cells to it, and
the choice depends on
$\pi^s_{m-1}(X_{D_1}(m))=\mathbb{Z}^{20}$. However in $D$, the six elements of
grading $0$ cover disjoint elements, and hence after attaching those
six $m$-cells, we get $X_D(m)\sim\vee_{i=1}^{14}S^{m-1}$. Notice once
more that this is entirely independent of the sign assignment.

\chapter{What lies beyond}
The purpose of this chapter is to briefly summarize what we have
talked about so far, and to outline a probable course of future
research.

In Chapter 2, we described nice Heegaard diagrams and how they can be
used to compute the hat version of the Heegaard Floer homology. It
will be an interesting exercise to prove the invariance of the hat
version of Heegaard Floer homology combinatorially using only nice
Heegaard diagrams and some collection of moves among the nice Heegaard
diagram which do not change the underlying three-manifold.

In Chapter 3, we concentrated on knots inside $S^3$, represented by
grid diagrams. A grid diagram being a nice Heegaard diagram, it
allowed us to compute all versions of knot Floer
homology. Furthermore, using a grid diagram we could also associate a
CW complex to a knot whose stable homotopy type is a knot invariant,
and whose homology (with coefficients in $\F_2$) is the knot Floer
homology (also with coefficients in $\F_2$). This leads to more
questions than it answers, some of which I would like to pursue in the
future. Two such questions are whether this result can be extended to
links, and whether the stable homotopy invariant contains any new
information in addition to the homology.

Thus ends our brief tour of my personal corner in the Heegaard Floer
homology universe. In conclusion I would like to thank Princeton
University for providing me with the financial support and the
opportunity to do this research. It had been a pleasant journey, and
one that I had enjoyed thoroughly.

\bibliographystyle{amsalpha}

\addcontentsline{toc}{chapter}{Reference}
\bibliography{concise}

\end{document}